\documentclass[11pt]{amsart}

\date{\today}

\usepackage[margin=1in]{geometry}
\usepackage{tikz}
\usetikzlibrary{calc, shapes.geometric, positioning, shapes, snakes}
\usepackage{setspace}
\usepackage{amsmath}
\usepackage{amsthm}
\usepackage{amssymb}
\usepackage{mathtools}
\usepackage{mathdots}
\usepackage{enumitem}
\usepackage{graphicx}
\usepackage{caption}
\usepackage[labelformat=simple,labelfont={}]{subcaption}
\usepackage{color}
\definecolor{darkblue}{rgb}{0, 0, .6}
\usepackage{hyperref}
\hypersetup{
	colorlinks=true,
	linkcolor=darkblue,
	anchorcolor=darkblue,
	citecolor=darkblue,
	pagecolor=darkblue,
	urlcolor=darkblue,
	pdftitle=darkblue,
	pdfauthor=darkblue
}

\theoremstyle{definition}
\newtheorem{theorem}{Theorem}[section]
\newtheorem{lemma}[theorem]{Lemma}

\newtheorem{corollary}[theorem]{Corollary}
\newtheorem{proposition}[theorem]{Proposition}

\newtheorem{example}[theorem]{Example}
\newtheorem{remark}[theorem]{Remark}

\numberwithin{equation}{section}

\newcommand{\Z}{\mathbb{Z}}
\newcommand{\N}{\mathbb{N}}
\newcommand{\A}{\mathcal{A}}
\newcommand{\C}{\widetilde{C}}
\renewcommand{\O}{\mathcal{O}}
\newcommand{\E}{\mathcal{E}}
\newcommand{\z}{\mathsf{z}}
\newcommand{\x}{\mathsf{x}}

\renewcommand{\u}{\mathsf{u}}
\renewcommand{\v}{\mathsf{v}}
\newcommand{\wtri}{\vartriangle}
\newcommand{\btri}{\blacktriangle}
\renewcommand{\a}{\mathbf{a}}
\DeclareMathOperator{\TL}{TL}
\DeclareMathOperator{\DTL}{\mathbb{D}TL}
\renewcommand{\P}{\mathcal{P}}
\newcommand{\V}{\mathcal{V}}
\newcommand{\D}{\mathbb{D}}

\newcommand{\wcirc}{\circ}

\newcommand{\bcirc}{\bullet}

\newcommand{\supp}{\mathrm{supp}}
\renewcommand{\L}{\mathcal{L}}
\newcommand{\R}{\mathcal{R}}

\newcommand{\w}{\mathsf{w}}
\renewcommand{\H}{\mathcal{H}}
\DeclareMathOperator{\FC}{FC}

\newcommand\xxaxis{0}
\newcommand\yyaxis{90}

\newcommand\heapblock[3]{\fill[draw=black, fill=gray!30, rounded corners, line width=1.1pt, shift={(\xxaxis:#1)},shift={(\yyaxis:#2)}] (-1,-0.5) rectangle (1,0.5);\node at (#1,#2) {\scriptsize $#3$};}

\newcommand\heapblank[2]{\fill[fill=white, dotted, draw=black, line width=1.1pt, rounded corners, shift={(\xxaxis:#1)},shift={(\yyaxis:#2)}] (-1,-0.5) rectangle (1,0.5);}

\newcommand\lp[2]{\draw[out=90,in=90] (#1,#2) to (#1 + 1,#2) [out=-90,in=-90] to (#1,#2);}

\newcommand\blacktrilp[2]{\draw[out=90,in=90] (#1,#2) to node[blacktri, pos=0.5]{} (#1 + 1,#2) [out=-90,in=-90] to (#1,#2);}

\newcommand\whitetrilp[2]{\draw[out=90,in=90] (#1,#2) to node[whitetri, pos=0.5]{} (#1 + 1,#2) [out=-90,in=-90] to (#1,#2);}

\newcommand\blackwhitetrilp[2]{\draw[out=90,in=90] (#1,#2) to node[blacktri, pos=0.5]{} (#1 +1,#2) [out=-90,in=-90] to node[whitetri, pos=0.5]{} (#1,#2);}

\begin{document}

\tikzstyle{blacktri} = [draw=black, fill=black, rotate=-90, sloped, regular polygon, regular polygon sides=3, inner sep=1.3pt]
\tikzstyle{whitetri} = [draw=black, fill=white, rotate=-90, sloped, regular polygon, regular polygon sides=3, inner sep=1.3pt]
\tikzstyle{blackcirc} = [draw=black, fill=black, shape=circle, inner sep=1.8pt]
\tikzstyle{whitecirc} = [draw=black, fill=white, shape=circle, inner sep=1.8pt]

\title[Diagram calculus for a type affine $C$ Temperley--Lieb algebra, II]{Diagram calculus for a type affine $C$ \\ Temperley--Lieb algebra, II}

\author[D.C.~Ernst]{Dana C.~Ernst}
\address{Department of Mathematics and Statistics, Northern Arizona University, Flagstaff, AZ 86011}
\email{dana.ernst@nau.edu}
\urladdr{http://danaernst.com}



\begin{abstract}
In a previous paper, we presented an infinite dimensional associative diagram algebra that satisfies the relations of the generalized Temperley--Lieb algebra having a basis indexed by the fully commutative elements of the Coxeter group of type affine $C$.  We also provided an explicit description of a basis for the diagram algebra.  In this paper, we show that the diagrammatic representation is faithful and establish a correspondence between the basis diagrams and the so-called monomial basis of the Temperley--Lieb algebra of type affine $C$.
\end{abstract}

\maketitle


\begin{section}{Introduction}\label{sec:intro}

The Temperley--Lieb algebra $\TL(A_n)$, invented by Temperley and Lieb in 1971~\cite{Temperley1971}, is a finite dimensional associative algebra that arose in the context of statistical mechanics.  Penrose~\cite{Penrose1971} and Kauffman~\cite{Kauffman1987} showed that this algebra can be realized as a diagram algebra having a basis given by certain diagrams, where multiplication is determined by applying local combinatorial rules to the diagrams. Jones~\cite{Jones1999} showed that $\TL(A_n)$ occurs naturally as a quotient of the Hecke algebra of type $A_n$.  

The realization of the Temperley--Lieb algebra as a Hecke algebra quotient was generalized by Graham~\cite{Graham1995} to the case of an arbitrary Coxeter system. A key feature of the quotient algebra is that it retains some of the relevant structure of the Hecke algebra, yet is small enough that computation of the leading coefficients of the Kazhdan--Lusztig polynomials is often much simpler. In this paper, we study the generalized Temperley--Lieb algebra of type $\C_{n}$, denoted $\TL(\C_{n})$, and describe a special basis, called the monomial basis, which is indexed by the fully commutative elements of the underlying Coxeter group. Our goal is to determine a faithful diagrammatic representation of the Temperley--Lieb algebra of type $\C_n$.

In~\cite{Ernst2012}, we constructed an infinite dimensional associative diagram algebra $\DTL(\C_n)$ and easily verified that this algebra satisfies the relations of $\TL(\C_n)$, thus showing that there is a surjective algebra homomorphism from $\TL(\C_n)$ to $\DTL(\C_n)$.  Moreover, we described a set of ``admissible diagrams" by providing a combinatorial description of the allowable edge configurations involving diagram decorations, and we accomplished the more difficult task of proving that this set of diagrams forms a basis for $\DTL(\C_n)$.  However, due to length considerations, it remained to be shown that our diagrammatic representation is faithful and that each admissible diagram corresponds to a unique monomial basis element of $\TL(\C_n)$.  

The main result of this paper (Theorem~\ref{thm:main result}) establishes the faithfulness of our diagrammatic representation and the correspondence between the admissible diagrams and the monomial basis elements of $\TL(\C_n)$.  The diagram algebra $\DTL(\C_n)$ presented here and in~\cite{Ernst2012} is the first faithful representation of an infinite dimensional non-simply-laced generalized Temperley--Lieb algebra. Unlike in the finite dimensional case, we cannot rely on counting arguments.  Instead, we make use of the classification in~\cite{Ernst2010} of the non-cancellable elements in Coxeter groups of type $\C_n$.

Besides being visually appealing, studying these types of diagrammatic representations can provide insight into the underlying algebraic structure.  In particular, there are direct applications related to Kazhdan--Lusztig theory. Similar to the work carried out in~\cite{Green2007a,Green2007}, in a future paper, we plan to construct a Jones-type trace on the Hecke algebra of type $\C_n$ using our diagrammatic representation of $\TL(\C_n)$, thus allowing us to quickly compute the leading coefficients of the infinitely many Kazhdan--Lusztig polynomials indexed by pairs of fully commutative elements.  

It turns out that the faithfulness of a diagrammatic representation of an arbitrary generalized Temperley--Lieb algebra is intimately related to Property B of~\cite{Green2007a}, which is a statement about the existence of a symmetric, anti-associative, nondegenerate bilinear form on the Temperley--Lieb algebra.  Property B was first conjectured to hold for arbitrary generalized Temperley--Lieb algebras in~\cite{Green2001} and remains an open problem.  Broadly speaking, Property B is more or less an algebraic reformulation of the existence of a diagram calculus for a canonical basis as in~\cite{Green1999}. If one could prove Property B directly, it could be used to identify a candidate basis of diagrams as the canonical basis. Conversely, if the canonical basis of diagrams is known, one can prove Property B for the corresponding Temperley--Lieb algebra.

The paper is organized as follows. Section~\ref{sec:prelim} of this paper is concerned with introducing the necessary notation and terminology of Coxeter groups, fully commutative elements, heaps, Hecke algebras, and generalized Temperley--Lieb algebras.  The concept of a heap introduced in Section~\ref{subsec:heaps} will be our main tool for visualizing combinatorial arguments required to prove several technical lemmas appearing in Section~\ref{subsec:prep lemmas}.  In Section~\ref{sec:combinatorics}, we study some of the combinatorics of Coxeter groups of type $\C_n$.  In particular, we introduce weak star reductions and the type I, type II, and non-cancellable elements of a Coxeter group of type $\C_n$, as well as establish several intermediate results.   The goal of Section~\ref{sec:diagram algebras} is to familiarize the reader with the conventions and terminology of diagram algebras necessary to define the diagram algebra $\DTL(\C_n)$ and to describe the admissible diagrams.  Our main result, which establishes the faithful diagrammatic representation of $\TL(\C_n)$ by $\DTL(\C_n)$, finally comes in Section~\ref{subsec:injectivity} after proving a few additional lemmas in Sections~\ref{subsec:hom} and~\ref{subsec:additional prep lemmas}.  Lastly, we discuss the implications of our results and future research in Section~\ref{sec:closing}.

This paper is an adaptation of results from the author's PhD thesis~\cite{Ernst2008}, which was directed by R.M.~Green at the University of Colorado Boulder.  However, the notation has been improved and many of the arguments have been streamlined.
\end{section}


\begin{section}{Preliminaries}\label{sec:prelim}


\begin{subsection}{Coxeter groups}\label{subsec:coxeter groups}

A \emph{Coxeter system} is a pair $(W,S)$ consisting of a distinguished (finite) set $S$ of generating involutions and a group $W$, called a \emph{Coxeter group}, with presentation
\[
W = \langle S \mid (st)^{m(s, t)} = e \text{ for } m(s, t) < \infty \rangle,
\]
where $e$ is the identity, $m(s,t) = 1$ if and only if $s = t$, and $m(s,t) = m(t,s)$. It turns out that $m(s, t)$ is the order of $st$.  Since the elements of $S$ have order two, the relation $(st)^{m(s,t)} = e$ can be written as
\[
\underbrace{sts \cdots}_{m(s,t)} = \underbrace{tst \cdots}_{m(s,t)}
\]
with $m(s,t) \geq 2$ factors.

Given a Coxeter system $(W,S)$, a word $s_{x_1}s_{x_2}\cdots s_{x_m}$ in the free monoid $S^*$ is called an \emph{expression} for $w\in W$ if it is equal to $w$ when considered as a group element. If $m$ is minimal among all expressions for $w$, the corresponding word is called a \emph{reduced expression} for $w$. In this case, we define the \emph{length} of $w$ to be $\ell(w)=m$. Each element $w \in W$ can have several different reduced expressions that represent it.  If we wish to emphasize a fixed, possibly reduced, expression for $w\in W$, we represent it in \textsf{sans serif} font, say $\w=s_{x_1}s_{x_2}\cdots s_{x_m}$, where each $s_{x_i} \in S$.  A product of group elements $w_{1}w_{2}\cdots w_{r}$ with $w_{i} \in W$ is called \emph{reduced} if $\ell(w_{1}w_{2}\cdots w_{r})=\sum \ell(w_{i})$.  

Matsumoto's Theorem~\cite[Theorem 1.2.2]{Geck2000} states that if $w \in W$, then every reduced expression for $w$ can be obtained from any other by applying a sequence of \emph{braid moves} of the form 
\[
{\underbrace{sts \cdots }_{m(s,t)} } \mapsto {\underbrace{tst \cdots}_{m(s,t)}}
\]
where $s,t \in S$ and each factor in the move has $m(s,t)$ letters.  The \emph{support} of an element $w \in W$, denoted $\supp(w)$, is the set of all generators appearing in any reduced expression for $w$.

Given a reduced expression $\w$ for $w \in W$, we define a \emph{subexpression} of $\w$ to be any subsequence of $\w$. We will refer to a subexpression consisting of a string of consecutive symbols from $\w$ as a \emph{subword} of $\w$.

The sets $\L(w)=\{s \in S\mid \ell(sw) < \ell(w)\}$ and $\R(w)=\{s \in S\mid \ell(ws) < \ell(w)\}$ are called the \emph{left} and \emph{right descent sets} of $w$, respectively.  It turns out that $s \in \L(w)$ (respectively, $\R(w)$) if and only if $w$ has a reduced expression beginning (respectively, ending) with $s$.

Given a Coxeter system $(W,S)$, the associated \emph{Coxeter graph} is the graph $\Gamma$ with vertex set $S$ and edges $\{s,t\}$ labeled with $m(s,t)$ for all $m(s,t)\geq 3$.  If $m(s,t)=3$, it is customary to leave the corresponding edge unlabeled.  Given a Coxeter graph $\Gamma$, we can uniquely reconstruct the corresponding Coxeter system $(W,S)$.  In this case, we say that $(W,S)$, or just $W$, is of type $\Gamma$. If $(W,S)$ is of type $\Gamma$, for emphasis, we may write $W$ and $S$ as $W(\Gamma)$ and $S(\Gamma)$, respectively.  Note that generators $s$ and $t$ are connected by an edge in the Coxeter graph $\Gamma$ if and only if $s$ and $t$ do not commute~\cite{Humphreys1990}.

The main focus of this paper will be the Coxeter systems of types $B_n$ and $\C_n$, which are defined by the Coxeter graphs in Figures~\ref{typeB} and~\ref{typeC}, respectively, where $n\geq 2$.

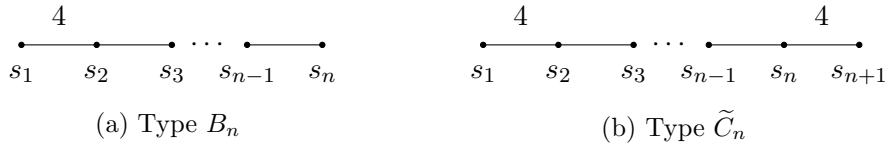
\begin{figure}[!ht]
\subcaptionbox{Type $B_n$\label{typeB}}[.4\textwidth]{
\begin{tikzpicture}
\draw[fill=black] \foreach \x in {1,2,3,4,5} {(\x,10) circle (1pt)};
\draw \foreach \x in {1,2,3} {(\x,10) node[label=below:$s_{\x}$]{}};
\draw (1.5,10) node[label=above:$4$]{};
\draw {(4,10) node[label=below:$s_{n-1}$]{}};
\draw {(5,10) node[label=below:$s_{n}$]{}};
\draw {(3.5,10) node[]{$\cdots$}};
\draw[-] (1,10) -- (3,10);
\draw[-] (4,10) -- (5,10);
\end{tikzpicture}
}
\subcaptionbox{Type $\C_n$\label{typeC}}[.4\textwidth]{
\begin{tikzpicture}
\draw[fill=black] \foreach \x in {1,2,3,4,5,6} {(\x,10) circle (1pt)};
\draw \foreach \x in {1,2,3} {(\x,10) node[label=below:$s_{\x}$]{}};
\draw (1.5,10) node[label=above:$4$]{};
\draw {(4,10) node[label=below:$s_{n-1}$]{}};
\draw {(5,10) node[label=below:$s_{n}$]{}};
\draw {(6,10) node[label=below:$s_{n+1}$]{}};
\draw {(3.5,10) node[]{$\cdots$}};
\draw (5.5,10) node[label=above:$4$]{};
\draw[-] (1,10) -- (3,10);
\draw[-] (4,10) -- (6,10);
\end{tikzpicture}
}
\caption{Coxeter graphs corresponding to Coxeter systems of types $B_{n}$ and $\C_{n}$.}
\end{figure}

We can obtain $W(B_{n})$ from $W(\C_{n})$ by removing the generator $s_{n+1}$ and the corresponding relations~\cite[Chapter 5]{Humphreys1990}.  We also obtain a Coxeter group of type $B_n$ if we remove the generator $s_{1}$ and the corresponding relations.  To distinguish these two cases, we let $W(B_{n})$ denote the subgroup of $W(\C_{n})$ generated by $\{s_{1}, s_{2}, \dots, s_{n}\}$ and we let $W(B'_{n})$ denote the subgroup of $W(\C_{n})$ generated by $\{s_{2}, s_{3}, \dots, s_{n+1}\}$.  It is well known that $W(\C_{n})$ is an infinite Coxeter group while $W(B_{n})$ and $W(B'_{n})$ are both finite~\cite[Chapters 2 and 6]{Humphreys1990}.

\end{subsection}


\begin{subsection}{Fully commutative elements}\label{subsec:FC}

Let $(W,S)$ be a Coxeter system of type $\Gamma$ and let $w \in W$. Following Stembridge~\cite{Stembridge1996}, we define a relation $\sim$ on the set of reduced expressions for $w$.  Let $\w$ and $\w'$ be two reduced expressions for $w$.  We define $\w \sim \w'$ if we can obtain $\w'$ from $\w$ by applying a single commutation move of the form $st \mapsto ts$, where $m(s,t)=2$.  Now, define the equivalence relation $\approx$ by taking the reflexive transitive closure of $\sim$.  Each equivalence class under $\approx$ is called a \emph{commutation class}. If $w$ has a single commutation class, then we say that $w$ is \emph{fully commutative} (FC).  According to Stembridge~\cite{Stembridge1996}, an element $w$ is FC if and only if no reduced expression for $w$ contains a subword of the form $sts \cdots$ of length $m(s,t) \geq 3$.  The set of FC elements of the Coxeter system $(W,S)$ of type $\Gamma$ is denoted by $\FC(\Gamma)$.

The elements of $\FC(\C_{n})$ are precisely those whose reduced expressions avoid consecutive subwords of the following types:
\begin{enumerate}
\item $s_{i}s_{j}s_{i}$ for $|i-j|=1$ and $1< i,j < n+1$;
\item $s_{i}s_{j}s_{i}s_{j}$ for $\{i,j\}=\{1,2\}$ or $\{n,n+1\}$.
\end{enumerate}
Note that the FC elements of $W(B_{n})$ and $W(B'_{n})$ avoid the respective subwords above.

In~\cite{Stembridge1996}, Stembridge classified the Coxeter groups that contain a finite number of FC elements.  The Coxeter group $W(\C_{n})$ contains an infinite number of FC elements, while $W(B_{n})$ (and hence $W(B'_n)$) contains finitely many~\cite[Theorem 5.1]{Stembridge1996}.  There are examples of infinite Coxeter groups that contain a finite number of FC elements.  For example, $W(E_n)$ is infinite for $n\geq 9$, but contains only finitely many FC elements.

\end{subsection}


\begin{subsection}{Heaps}\label{subsec:heaps}

Every reduced expression can be associated with a partially ordered set called a heap that will allow us to visualize a reduced expression  while preserving the essential information about the relations among the generators.  The theory of heaps was introduced in~\cite{Viennot1986} by Viennot and visually captures the combinatorial structure of the Cartier--Foata monoid of~\cite{Cartier1969}.  In~\cite{Stembridge1996,Stembridge1998}, Stembridge studied heaps in the context of FC elements, which is our motivation here.  In this section, we mimic the development found in~\cite{Billey2007,Ernst2010,Stembridge1996}.

Let $(W,S)$ be a Coxeter system.  Suppose $\w = s_{x_1} \cdots s_{x_r}$ is a fixed reduced expression for $w \in W$.  As in~\cite{Stembridge1996}, we define a partial ordering on the indices $\{1, \dots, r\}$ by the transitive closure of the relation $\lessdot$ defined via $j \lessdot i$ if $i < j$ and $s_{x_i}$ and $s_{x_j}$ do not commute.  In particular, since $\w$ is reduced, $j \lessdot i$ if $i < j$ and $s_{x_i} = s_{x_j}$ by transitivity.  This partial order is referred to as the \emph{heap} of $\w$, where $i$ is labeled by $s_{x_i}$. Note that for simplicity, instead of labelling the heap by elements of the underlying poset $\{1, 2,\ldots,r\}$, we are using the names of the corresponding generators, namely $s_{x_1},\ldots,s_{x_r}$.

It follows from~\cite[Proposition 2.2]{Stembridge1996} that heaps are well-defined up to commutation class.  That is, if $\w$ and $\w'$ are two reduced expressions for $w \in W$ that are in the same commutation class, then the heaps of $\w$ and $\w'$ are equal.  In particular, if $w$ is FC, then it has a single commutation class, and so there is a unique heap associated to $w$.

\begin{example}\label{ex:first heap}
Let $\w = s_3 s_2 s_1 s_2 s_5s_{4}s_{6}s_{5}$ be a reduced expression for $w \in \FC(\C_{5})$.  We see that $\w$ is indexed by $\{1, 2, 3, 4, 5, 6, 7, 8\}$.  As an example, $3 \lessdot 2$ since $2 < 3$ and the second and third generators do not commute.  The labeled Hasse diagram for the heap poset of $w$ is shown in Figure~\ref{fig:hasse}.
\end{example}

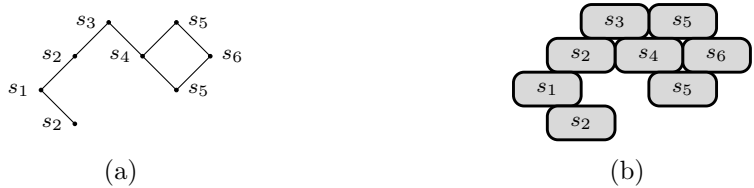
\begin{figure}[!ht]
\subcaptionbox{\label{fig:hasse}}[.4\textwidth]{ 
\begin{tikzpicture}[scale=0.45,inner sep = 2pt]
\draw [fill=black] (2,1) circle (1.5pt);
\draw (2,1) node[label=left:\scriptsize $s_2$]{};
\draw [fill=black] (1,2) circle (1.5pt);
\draw (1,2) node[label=left:\scriptsize $s_1$]{};
\draw [fill=black] (5,2) circle (1.5pt);
\draw (5,2) node[label=right:\scriptsize $s_5$]{};
\draw [fill=black] (2,3) circle (1.5pt);
\draw (2,3) node[label=left:\scriptsize $s_2$]{};
\draw [fill=black] (4,3) circle (1.5pt);
\draw (4,3) node[label=left:\scriptsize $s_4$]{};
\draw [fill=black] (6,3) circle (1.5pt);
\draw (6,3) node[label=right:\scriptsize $s_6$]{};
\draw [fill=black] (3,4) circle (1.5pt);
\draw (3,4) node[label=left:\scriptsize $s_3$]{};
\draw [fill=black] (5,4) circle (1.5pt);
\draw (5,4) node[label=right:\scriptsize $s_5$]{};
\draw [color=white] (1,.5) circle (0pt);
\draw (2,1)--(1,2)--(2,3)--(3,4)--(4,3)--(5,2)--(6,3)--(5,4)--(4,3);
\end{tikzpicture}
}
\subcaptionbox{\label{fig:first heap}}[.4\textwidth]{
\begin{tikzpicture}[scale=0.45]
\heapblock{5}{4}{s_5};
\heapblock{3}{4}{s_3};
\heapblock{6}{3}{s_6};
\heapblock{4}{3}{s_4};
\heapblock{2}{3}{s_2};
\heapblock{5}{2}{s_5};
\heapblock{1}{2}{s_1};
\heapblock{2}{1}{s_2};
\end{tikzpicture}
}
\caption{Labeled Hasse diagram and possible lattice point representation for the heap of an element from $\FC(\C_5)$.}
\end{figure}

Let $\w$ be a fixed reduced expression for $w \in W(\C_{n})$.  As in~\cite{Billey2007,Ernst2010}, we can represent a heap for $\w$ as a set of lattice points embedded in $\{1,2,\ldots,n+1\} \times \mathbb{N}$.  To do so, we assign coordinates (not unique) $(x,y) \in \{1,2,\ldots, n+1\} \times \mathbb{N}$ to each entry of the labeled Hasse diagram for the heap of $\w$ in such a way that:
\begin{enumerate}
\item An entry with coordinates $(x,y)$ is labeled $s_i$ in the heap if and only if $x = i$; 
\item An entry with coordinates $(x,y)$ is greater than an entry with coordinates $(x',y')$ in the heap if and only if $y > y'$.
\end{enumerate}

Recall that a finite poset is determined by its covering relations.  In the case of a Coxeter group of type $\C_{n}$ (and any straight line Coxeter graph), it follows from the definition that $(x,y)$ covers $(x',y')$ in the heap if and only if $x = x' \pm 1$, $y > y'$, and there are no entries $(x'', y'')$ such that $x'' \in \{x, x'\}$ and $y'< y'' < y$.  This implies that we can completely reconstruct the edges of the Hasse diagram and the corresponding heap poset from a lattice point representation. The lattice point representation of a heap allows us to visualize potentially cumbersome arguments.  Note that our heaps are upside-down versions of the heaps that appear in~\cite{Billey2007} and several other papers.  That is, in this paper entries at the top of a heap correspond to generators occurring to the left, as opposed to the right, in the corresponding reduced expression.  Our convention aligns more naturally with the typical conventions of diagram algebras.

Let $\w$ be a reduced expression for $w \in W(\C_{n})$.  We let $H(\w)$ denote a lattice representation of the heap poset in $\{1,2,\ldots,n+1\} \times \N$ described in the preceding paragraphs.  If $w$ is FC, then the choice of reduced expression for $w$ is irrelevant, in which case, we will often write $H(w)$ (note the absence of \textsf{sans serif} font) and we will refer to $H(w)$ as the heap of $w$.

Given a heap, there are many possible coordinate assignments, yet the $x$-coordinates for each entry will be fixed for all of them.  In particular, two entries labeled by the same generator may only differ by the amount of vertical space between them while maintaining their relative vertical position to adjacent entries in the heap.

Let $\w=s_{x_1}\cdots s_{x_r}$ be a reduced expression for $w \in \FC(\C_{n})$.  If $s_{x_i}$ and $s_{x_j}$ are adjacent generators in the Coxeter graph with $i<j$, then we must place the point labeled by $s_{x_i}$ at a level that is \emph{above} the level of the point labeled by $s_{x_j}$.  Because generators that are not adjacent in the Coxeter graph do commute, points whose $x$-coordinates differ by more than one can slide past each other or land at the same level.  To emphasize the covering relations of the lattice representation we will enclose each entry of the heap in a rectangle (with rounded corners) in such a way that if one entry covers another, the rectangles overlap halfway.

\begin{example}\label{ex:second heap}
If $w$ is the element from Example~\ref{ex:first heap}, then one possible lattice point representation for $H(w)$ appears in Figure~\ref{fig:first heap}.
\end{example}

When $w$ is FC, we wish to make a canonical choice for the representation of $H(w)$ by assembling the entries in a particular way.  To do this, we give all entries corresponding to elements in $\L(w)$ the same vertical position and all other entries in the heap should have vertical position as high as possible.  For example, the representation of $H(w)$ given in Figure~\ref{fig:first heap} is the canonical representation.  Note that our canonical representation of heaps of FC elements corresponds precisely to the unique heap factorization of~\cite[Lemma 2.9]{Viennot1986} and to the Cartier--Foata normal form for monomials~\cite{Cartier1969,Green2006a}.  When illustrating heaps, we will adhere to this canonical choice, and when we consider the heaps of arbitrary reduced expressions, we will only allude to the relative vertical positions of the entries, and never their absolute coordinates.  

Let $\w=s_{x_1}\cdots s_{x_r}$ be a reduced expression for $w \in \FC(\C_n)$ and suppose $s_{x_i}$ and $s_{x_j}$ equal the same generator $s_k$, so that the corresponding entries in $H(w)$ have $x$-coordinate $k$.  We say that $s_{x_i}$ and $s_{x_j}$ are \emph{consecutive} if there is no other occurrence of $s_{k}$ occurring between them in $\w$.

Let $\w=s_{x_{1}} \cdots s_{x_{r}}$ be a reduced expression for $w \in W(\C_{n})$.  We define a heap $H'$ to be a \emph{subheap} of $H(\w)$ if $H'=H(\w')$, where $\w'=s_{y_1}s_{y_2} \cdots s_{y_k}$ is a subexpression of $\w$.  We emphasize that the subexpression need not be a subword.

A subheap $H'$ of $H(\w)$ is called a \emph{saturated subheap} if whenever $s_{i}$ and $s_{j}$ occur in $H'$ such that there exists a saturated chain from $i$ to $j$ in the underlying poset for $H(\w)$, there also exists a saturated chain $i=i_{k_{1}}<i_{k_{2}}< \cdots < i_{k_{l}}=j$ in the underlying poset for $H'$ such that the same chain is also a saturated chain in the underlying poset for $H(\w)$.

Recall that a subposet $Q$ of $P$ is called convex if $y \in Q$ whenever $x < y < z$ in $P$ and $x, z \in Q$.  A \emph{convex subheap} is a subheap in which the underlying subposet is convex.

\begin{example}\label{ex:third heap}
Consider $\w$ and $w$ from Example~\ref{ex:first heap} and let $\u=s_3 s_1 s_{5}$ be the subexpression of $\w$ that results from deleting all but the first, third, and last generators of $\w$.  Then $H(\u)$ is equal to the heap in Figure~\ref{fig:subheap} and is a subheap of $H(\w)$.  However, $H(\u)$ is neither a saturated nor a convex subheap of $H(\w)$ since there is a saturated chain in $H(\w)$ from the lower occurrence of $s_{5}$ to the occurrence of $s_{3}$, but there is not a chain between the corresponding entries in $H(\u)$.  Now, let $\v=s_{5}s_{4}s_{5}$ be the subexpression of $\w$ that results from deleting all but the fifth, sixth, and last generators of $\w$.  Then $H(\v)$ equals the heap in Figure~\ref{fig:convex subheap} and is a saturated subheap of $H(\w)$, but it is not a convex subheap since there is an entry in $H(\w)$ labeled by $s_{6}$ occurring between the two consecutive occurrences of $s_{5}$ that does not occur in $H(\v)$.  However, if we do include the entry labeled by $s_{6}$, then the heap in Figure~\ref{fig:saturated subheap}  is a convex subheap of $H(\w)$.  
\end{example}

\begin{figure}[!ht]
\subcaptionbox{\label{fig:subheap}}[.3\textwidth]{
\begin{tikzpicture}[scale=0.45]
\heapblock{5}{1}{s_5};
\heapblock{3}{1}{s_3};
\heapblock{1}{1}{s_1};
\end{tikzpicture}
}
\subcaptionbox{\label{fig:convex subheap}}[.3\textwidth]{
\begin{tikzpicture}[scale=0.45]
\heapblock{5}{3}{s_5};
\heapblock{4}{2}{s_4};
\heapblock{5}{1}{s_5};
\end{tikzpicture}
}
\subcaptionbox{\label{fig:saturated subheap}}[.3\textwidth]{
\begin{tikzpicture}[scale=0.45]
\heapblock{5}{3}{s_5};
\heapblock{6}{2}{s_6};
\heapblock{4}{2}{s_4};
\heapblock{5}{1}{s_5};
\end{tikzpicture}
}
\caption{Various subheaps of the heap given in Figure~\ref{fig:first heap}.}
\label{fig:subheaps}
\end{figure}
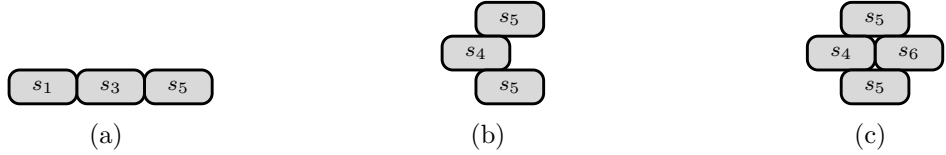

From this point on, if there can be no confusion, we will not specify the exact subexpression that a subheap arises from.  The following fact is implicit in the literature (in particular, see the proof of~\cite[Proposition 3.3]{Stembridge1996}) and follows easily from the definitions.

\begin{proposition}
Let $(W,S)$ be a Coxeter system of type $\Gamma$ and let $w \in \FC(\Gamma)$. Then $H'$ is a convex subheap of $H(w)$ if and only if $H'$ is the heap for some subword of some reduced expression for $w$.   \qed
\end{proposition}

It will be extremely useful for us to be able to recognize when a heap corresponds to an element in $\FC(\C_{n})$.  The next lemma is a special case of~\cite[Proposition 3.3]{Stembridge1996} and follows quickly after one considers the consecutive subwords that are forbidden in reduced expressions for elements in $\FC(\C_{n})$.

\begin{lemma}\label{lem:impermissible heap configs}
If $w \in \FC(\C_{n})$, then $H(w)$ cannot contain any of the convex subheaps of Figure~\ref{fig:impermissible heap configs}, where $1<k<n+1$ and we use a rectangle with a dotted boundary to emphasize that no element of the heap occupies the corresponding position.  \qed
\end{lemma}

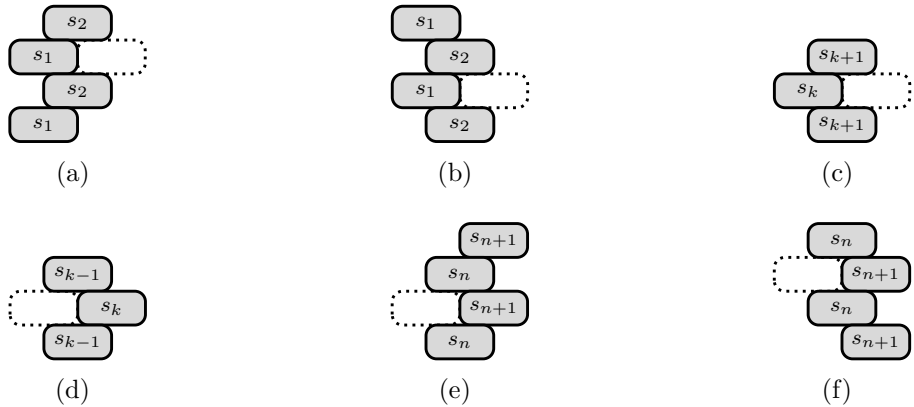
\begin{figure}[!ht]
\subcaptionbox{}[.3\linewidth]{
\begin{tikzpicture}[scale=0.45]
\heapblank{3}{3}
\heapblock{2}{4}{s_2};
\heapblock{1}{3}{s_1};
\heapblock{2}{2}{s_2};
\heapblock{1}{1}{s_1};
\end{tikzpicture}
}
\subcaptionbox{}[.3\linewidth]{
\begin{tikzpicture}[scale=0.45]
\heapblank{3}{2}
\heapblock{1}{4}{s_1};
\heapblock{2}{3}{s_2};
\heapblock{1}{2}{s_1};
\heapblock{2}{1}{s_2};
\end{tikzpicture}
}
\subcaptionbox{}[.3\linewidth]{
\begin{tikzpicture}[scale=0.45]
\heapblank{3}{2}
\heapblock{2}{3}{s_{k+1}};
\heapblock{1}{2}{s_k};
\heapblock{2}{1}{s_{k+1}};
\end{tikzpicture}
}\\
\vspace{1em}
\subcaptionbox{}[.3\linewidth]{
\begin{tikzpicture}[scale=0.45]
\heapblank{0}{2}
\heapblock{1}{3}{s_{k-1}};
\heapblock{2}{2}{s_k};
\heapblock{1}{1}{s_{k-1}};
\end{tikzpicture}
}
\subcaptionbox{}[.3\linewidth]{
\begin{tikzpicture}[scale=0.45]
\heapblank{0}{2}
\heapblock{2}{4}{s_{n+1}};
\heapblock{1}{3}{s_n};
\heapblock{2}{2}{s_{n+1}};
\heapblock{1}{1}{s_n};
\end{tikzpicture}
}
\subcaptionbox{}[.3\linewidth]{
\begin{tikzpicture}[scale=0.45]
\heapblank{0}{3}
\heapblock{1}{4}{s_n};
\heapblock{2}{3}{s_{n+1}};
\heapblock{1}{2}{s_n};
\heapblock{2}{1}{s_{n+1}};
\end{tikzpicture}
}
\caption{Impermissible convex subheaps for elements in $\FC(\C_n)$.}\label{fig:impermissible heap configs}
\end{figure}

\end{subsection}


\begin{subsection}{Hecke algebras}

Let $(W,S)$ be a Coxeter system of type $\Gamma$.  Loosely speaking, the associated Hecke algebra is an algebra with a basis indexed by the elements of $W$ and relations that deform the relations of $W$ by a parameter $q$.  More specifically, we define the \emph{Hecke algebra} of type $\Gamma$, denoted by $\H_{q}(\Gamma)$, to be the $\Z[q,q^{-1}]$-algebra with basis consisting of elements $T_{w}$, for all $w \in W(\Gamma)$, satisfying
\[
T_{s}T_{w}=\begin{cases}
T_{sw},  & \text{if } \ell(sw)>\ell(w)\\
qT_{sw}+(q-1)T_{w},  & \text{if } \ell(sw)<\ell(w),
\end{cases}
\]
where $s \in S(\Gamma)$ and $w \in W(\Gamma)$.  This is enough to compute $T_{x}T_{w}$ for arbitrary $x, w \in W(\Gamma)$.  Also, it follows from the definition that each $T_{w}$ is invertible.  It is convenient to extend the scalars of $\H_{q}(\Gamma)$ to produce an $\A$-algebra, $\H(\Gamma)=\A \otimes_{\Z[q,q^{-1}]} \H_{q}(\Gamma)$, where $\A=\Z[v,v^{-1}]$ and $v^{2}=q$.  The Laurent polynomial $v+v^{-1} \in \A$ occurs frequently and will be denoted by $\delta$.  

Since $W(\C_{n})$ is an infinite group, $\H(\C_{n})$ is an $\A$-algebra of infinite rank.  On the other hand, since $W(B_{n})$ and $W(B'_{n})$ are finite, both $\H(B_{n})$ and $\H(B'_{n})$ are of finite rank. For more on Hecke algebras, we refer the reader to~\cite[Chapter 7]{Humphreys1990}.

\end{subsection}


\begin{subsection}{Temperley--Lieb algebras}\label{subsec:TL-algebras}

Let $J(\Gamma)$ be the two-sided ideal of $\H(\Gamma)$ generated by the set $\{J_{s,t}\mid 3\leq m(s,t)<\infty\}$, where
\[
J_{s,t}=\sum_{w \in \langle s, t \rangle}T_{w}
\]
and $\langle s, t \rangle$ is the subgroup generated by $s$ and $t$.

Following Graham~\cite[Definition 6.1]{Graham1995}, we define the \emph{(generalized) Temperley--Lieb algebra}, $\TL(\Gamma)$, to be the quotient $\A$-algebra $\H(\Gamma)/J(\Gamma)$.  The image of $T_w$ under the canonical epimorphism $\H(\Gamma) \to \TL(\Gamma)$ is denoted $t_w$.  According to \cite[Theorem 6.2]{Graham1995}, the set $\{t_{w}\mid w \in \FC(\Gamma)\}$ is an $\A$-basis for $\TL(\Gamma)$.  

For our purposes, it will be more useful to work a different basis, which we define in terms of the $t$-basis. For each $s \in S(\Gamma)$, define $b_{s}=v^{-1}t_{s}+v^{-1}t_{e}$, where $e$ is the identity in $W(\Gamma)$.  If $s=s_{i}$, we will write $b_{i}$ in place of $b_{s_{i}}$.  If $w \in \FC(\Gamma)$ has reduced expression $\w=s_{x_{1}}\cdots s_{x_{r}}$, then we define 
\[
b_{\w}=b_{x_{1}}\cdots b_{x_{r}}.
\]

Note that if $\w$ and $\w'$ are two different reduced expressions for $w \in \FC(\Gamma)$, then $b_{\w}=b_{\w'}$ since $\w$ and $\w'$ are commutation equivalent and $b_{i}b_{j}=b_{j}b_{i}$ when $m(s_{i}, s_{j})=2$.  So we will write $b_{w}$ if we do not have a particular reduced expression in mind.  It is well known (and follows from~\cite[Proposition 2.4]{Green2006a}) that the set $\{b_{w}\mid w \in \FC(\Gamma)\}$ forms an $\A$-basis for $\TL(\Gamma)$.  This basis is referred to as the \emph{monomial basis}.  We let $b_{e}$ denote the identity of $\TL(\Gamma)$.

Recall that $W(\C_{n})$ contains an infinite number of FC elements, while $W(B_{n})$ and $W(B'_{n})$ contain finitely many.  Hence $\TL(\C_{n})$ is an $\A$-algebra of infinite rank while $\TL(B_{n})$ and $\TL(B'_{n})$ are of finite rank.  Note that we can have $\H(\Gamma)$ being of infinite rank while $\TL(\Gamma)$ is of finite rank.  In particular, $\TL(E_n)$ for $n\geq 9$ is finite dimensional while $\H(E_n)$ is of infinite rank.

It will be convenient for us to have a presentation for $\TL(\Gamma)$ in terms of generators and relations. The following proposition is a special case of~\cite[Proposition 2.6]{Green2006a}.  

\begin{proposition}\label{prop:affine C relations}
The algebra $\TL(\C_{n})$ is generated (as a unital algebra) by $b_{1}, b_{2}, \dots, b_{n+1}$ with defining relations
\begin{enumerate}[label=\rm{(\arabic*)}]
\item $b_{i}^{2}=\delta b_{i}$ for all $i$;
\item $b_{i}b_{j}=b_{j}b_{i}$ if $|i-j|>1$;
\item $b_{i}b_{j}b_{i}=b_{i}$ if $|i-j|=1$ and $1< i,j < n+1$;
\item $b_{i}b_{j}b_{i}b_{j}=2b_{i}b_{j}$ if $\{i,j\}=\{1,2\}$ or $\{n,n+1\}$.
\end{enumerate}
In addition, $\TL(B_{n})$ (respectively, $\TL(B'_{n})$) is generated (as a unital algebra) by $b_{1}, b_{2}, \dots, b_{n}$ (respectively, $b_{2}, b_{3}, \dots, b_{n+1}$) with the corresponding relations above.  \hfill $\qed$
\end{proposition}

It is known that we can consider $\TL(B_{n})$ and $\TL(B'_{n})$ as subalgebras of $\TL(\C_{n})$ in the obvious way. It will be useful for us to know what form an arbitrary product of monomial generators takes in $\TL(\C_{n})$.  The next lemma is similar to~\cite[Lemma 2.1.3]{Green2000}, which is a statement involving $W(B_{n})$.

\begin{lemma}\label{lem:powers of 2 and delta monomials}
If $w \in \FC(\C_{n})$ and $s \in S(\C_{n})$, then
\[
b_{s}b_{w}=2^{k}\delta^{m}b_{w'}
\]
for some $k, m \in \Z^{+}\cup \{0\}$ and $w' \in \FC(\C_{n})$.
\end{lemma}

\begin{proof}
We induct on the length of $w$.  For the base case, suppose $\ell(w)=0$.  Then for any $s \in S(\C_{n})$, we have $b_{s}b_{e}=b_{s}$, as desired.  Now, suppose $\ell(w)=p>1$.  There are three possibilities to consider.

First, if $sw$ is not reduced, then $s \in \L(w)$, and so we must be able to write $w=sv$ (reduced).  In this case, we see that $b_{s}b_{w}=b_{s}b_{s}b_{v}=\delta b_{s}b_{v}=\delta b_{w}$. Next, if $sw$ is reduced and FC, then $b_{s}b_{w}=b_{sw}$.

For the final case, suppose $sw$ is reduced but not FC.  Then $sw$ must have a reduced expression containing the subword $sts$ if $m(s,t)=3$ or $stst$ if $m(s,t)=4$.  So we must be able to write
\[
w=\begin{cases}
utsv, & \text{if } m(s,t)=3\\
utstv, & \text{if } m(s,t)=4,
\end{cases}
\]
where each product is reduced, $u, v \in \FC(\C_{n})$, and $s$ commutes with every element of $\supp(u)$, so that
\[
sw=\begin{cases}
ustsv, & \text{if } m(s,t)=3\\
uststv, & \text{if } m(s,t)=4.
\end{cases}
\]
This implies that
\[
b_{s}b_{w}=\begin{cases}
b_{u}b_{s}b_{t}b_{s}b_{v}=b_{u}b_{s}b_{v}, & \text{if } m(s,t)=3\\
b_{u}b_{s}b_{t}b_{s}b_{t}b_{v}=2b_{u}b_{s}b_{t}b_{v}, & \text{if } m(s,t)=4.
\end{cases}
\]
Note that $\ell(u)+1+\ell(v) < p$ when $m(s,t)=3$ and $\ell(u)+2+\ell(v) < p$ when $m(s,t)=4$.  So we can apply the inductive hypothesis $\ell(u)+1$ (respectively, $\ell(u)+2$) times if $m(s,t)=3$ (respectively, $m(s,t)=4$) starting with $b_{s}b_{v}$ (respectively, $b_{t}b_{v}$).  Therefore, we obtain $b_{s}b_{w}=2^{k}\delta^{m}b_{w'}$ for some $k, m \in \Z^{+}\cup \{0\}$ and $w' \in \FC(\C_{n})$.
\end{proof}

If $b_{x_{1}}, b_{x_{2}}, \dots, b_{x_{p}}$ is any collection of $p$ monomial generators, then it follows immediately from Lemma~\ref{lem:powers of 2 and delta monomials}  that $b_{x_{1}}b_{x_{2}}\cdots b_{x_{p}}=2^{k}\delta^{m} b_{w}$ for some $k, m \in \Z^{+}\cup \{0\}$ and $w \in \FC(\C_{n})$.

\end{subsection}

\end{section}


\begin{section}{Combinatorics in Coxeter groups of type affine $C$}\label{sec:combinatorics}

In this section, we explore some of the relevant combinatorics in Coxeter groups of type $\C_n$.


\begin{subsection}{type I and type II elements}

Define the following elements of $W(\C_{n})$.
\begin{enumerate}
\item If $i<j$, let
\[
\z_{i,j}=s_{i}s_{i+1}\cdots s_{j-1}s_{j}
\]
and
\[
\z_{j,i}=s_{j}s_{j-1}\cdots s_{i-1}s_{i}.
\]
We also let $\z_{i,i}=s_{i}$.

\item If $1< i \leq n+1$ and $1 < j \leq n+1$, let
\[
\z^{L,2k}_{i,j}=\z_{i,2}(\z_{1,n}\z_{n+1,2})^{k-1}\z_{1,n}\z_{n+1,j}.
\]
\item If $1< i \leq n+1$ and $1 \leq j < n+1$, let
\[
\z^{L,2k+1}_{i,j}=\z_{i,2}(\z_{1,n}\z_{n+1,2})^{k}\z_{1,j}.
\]

\item If $1\leq i < n+1$ and $1 \leq j <  n+1$, let
\[
\z^{R,2k}_{i,j}=\z_{i,n}(\z_{n+1,2}\z_{1,n})^{k-1}\z_{n+1,2}\z_{1,j}.
\]
	
\item If $1\leq i < n+1$ and $1 < j \leq  n+1$, let 
\[
\z^{R,2k+1}_{i,j}=\z_{i,n}(\z_{n+1,2}\z_{1,n})^{k}\z_{n+1,j}.
\]
\end{enumerate}
If $w \in W(\C_n)$ is equal to one of the elements in (1)--(5), then we say that $w$ is of \emph{type I}.

The notation for the type I elements looks more cumbersome than the underlying concept.  Our notation is motivated by the zigzagging shape of the corresponding heaps.  The index $i$ tells us where to start and the index $j$ tells us where to stop.  The L (respectively, R) tells us to start zigzagging to the left (respectively, right).  Also, $2k+1$ (respectively, $2k$) indicates the number of times we should encounter an end generator (i.e., $s_{1}$ or $s_{n+1}$) after the first occurrence of $s_{i}$ as we zigzag through the generators.  If $s_{i}$ is an end generator, it is not included in this count.  However, if $s_{j}$ is an end generator, it is included.  

Note that the expressions given for the type I elements are reduced. In fact, every type I element is rigid in the sense that each has a unique reduced expression. This implies that every type I element is FC and has a unique heap.  Moreover, there are an infinite number of type I elements since there is no limit to the amount of zigzagging. 

\begin{example}
If $1<i,j\leq n+1$, then $H(\z^{L,2k}_{i,j})$ is equal to the heap in Figure~\ref{fig:zigzag}, where we encounter an entry labeled by either $s_{1}$ or $s_{n+1}$ a combined total of $2k$ times if $i\neq n+1$ and $2k+1$ times if $i=n+1$.  
\end{example}

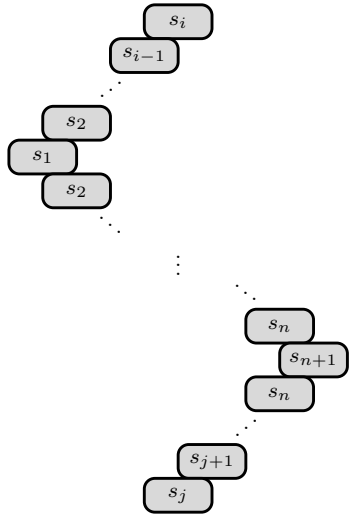
\begin{figure}[!ht]
\begin{tikzpicture}[scale=0.45]
\heapblock{5}{15}{s_i};
\heapblock{4}{14}{s_{i-1}};
\draw (3,13.2) node {\scriptsize $\iddots$};
\heapblock{2}{12}{s_2};
\heapblock{1}{11}{s_1};
\heapblock{2}{10}{s_2};
\draw (3,9.2) node {\scriptsize $\ddots$};
\draw (5,8) node {\scriptsize $\vdots$};
\draw (7,7.2) node {\scriptsize $\ddots$};
\heapblock{8}{6}{s_n};
\heapblock{9}{5}{s_{n+1}};
\heapblock{8}{4}{s_{n}};
\draw (7,3.2) node {\scriptsize $\iddots$};
\heapblock{6}{2}{s_{j+1}};
\heapblock{5}{1}{s_{j}};
\end{tikzpicture}
\caption{Example of the heap for a type I element.}\label{fig:zigzag} 
\end{figure}

For $w \in \FC(\C_{n})$, we define $n(w)$ to be the maximum integer $k$ such that $w = u x v$ (reduced), where $u, x, v \in \FC(\C_{n})$, $\ell(x)=k$, and $x$ is a product of commuting generators.  Note that $n(w)$ may be greater than the size of any row in the canonical representation of $H(w)$.  It is known that $n(w)$ is equal to the size of a maximal antichain in the heap poset for $w$~\cite[Lemma 2.9]{Shi2005b}.  

It is easily seen that if $w$ is of type I, then $w$ is FC with $n(w)=1$. Conversely, according to~\cite[Proposition 3.1.3]{Ernst2010}, if $n(w)=1$, then $w$ is of type I.

Next, define
\[
\x_{\O}=s_{1}s_{3}\cdots s_{2\lambda-1}s_{2\lambda+1},
\]
and
\[
\x_{\E}=s_{2}s_{4}\cdots s_{2\lambda-2}s_{2\lambda},
\]
where $\lambda=\lceil \frac{n-1}{2}\rceil$.  If $w \in W(\C_n)$ is equal to a finite alternating product of $\x_{\O}$ and $\x_{\E}$, then we say that $w$ is of \emph{type II}.  It is important to point out that the corresponding expressions are indeed reduced and that there are infinitely many type II elements, all of which are FC.  Moreover, if $w$ is of type II but not equal to $\x_{\E}$ when $n$ is even, then $n(w)=\lambda$~\cite[Proposition 3.2.3]{Ernst2010}. However, if $w\in\FC(\C_n)$ such that $n(w)=\lambda$, then $w$ may not be of type II. 

\end{subsection}


\begin{subsection}{Preparatory lemmas}\label{subsec:prep lemmas}  

We will use the following result~\cite[Lemma 5.2.1]{Ernst2010} to determine whether an element is of type I.

\begin{lemma}\label{lem:zigzag}
Let $w \in \FC(\C_{n})$. Suppose $w=uzv$ (reduced), where $z$ is equal to one of the following type I elements:  
\begin{enumerate}[label=\rm{(\arabic*)}]
\item $\z^{L,2}_{2,n}=s_{2}s_{1}s_{2}s_{3} \cdots s_{n-1}s_{n}s_{n+1}s_{n}$;
\item $\z^{R,2}_{n,2}=s_{n}s_{n+1}s_{n}s_{n-1}\cdots s_{3}s_{2}s_{1}s_{2}$;
\item $\z^{R,2}_{1,1}=s_{1}s_{2}\cdots s_{n}s_{n+1}s_{n}\cdots s_{2}s_{1}$;
\item $\z^{L,2}_{n+1,n+1}=s_{n+1}s_{n}\cdots s_{2}s_{1}s_{2}\cdots s_{n}s_{n+1}$.
\end{enumerate}
Then $w$ is of type I. \hfill $\qed$
\end{lemma}

The purpose of the next three lemmas (Lemmas~\ref{lem:all filled in}--\ref{lem:typeI or nothing above}) is to prove Lemma~\ref{lem:main zigzag lemma}, which plays a crucial role in Section~\ref{sec:main results}

\begin{lemma}\label{lem:all filled in}
Let $w \in \FC(\C_{n})$.  If the heap in Figure~\ref{fig:all filled in 1} is a saturated subheap of $H(w)$, where $i \neq n+1$, then the heap in Figure~\ref{fig:all filled in 2} is a convex subheap of $H(w)$, where every possible entry occurs in the region between the two occurrences of $s_1$ in Figure~\ref{fig:all filled in 1}.
\end{lemma}

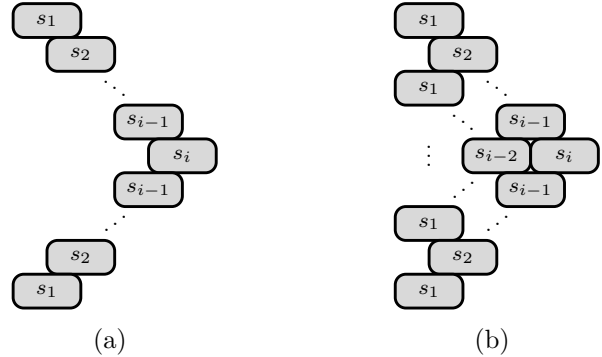
\begin{figure}[!ht]
\subcaptionbox{\label{fig:all filled in 1}}[.3\textwidth]{
\begin{tikzpicture}[scale=0.45]
\heapblock{1}{9}{s_1};
\heapblock{2}{8}{s_2};
\draw (3,7.2) node {\scriptsize $\ddots$};
\heapblock{4}{6}{s_{i-1}};
\heapblock{5}{5}{s_i};
\heapblock{4}{4}{s_{i-1}};
\draw (3,3.2) node {\scriptsize $\iddots$};
\heapblock{2}{2}{s_2};
\heapblock{1}{1}{s_1};
\end{tikzpicture}
}
\subcaptionbox{\label{fig:all filled in 2}}[.3\textwidth]{
\begin{tikzpicture}[scale=0.45]
\heapblock{1}{9}{s_1};
\heapblock{2}{8}{s_2};
\heapblock{1}{7}{s_1};
\draw (2,6.2) node {\scriptsize $\ddots$};
\draw (3,7.2) node {\scriptsize $\ddots$};
\draw (1,5.2) node {\scriptsize $\vdots$};
\heapblock{4}{6}{s_{i-1}};
\heapblock{3}{5}{s_{i-2}};
\heapblock{5}{5}{s_i};
\heapblock{4}{4}{s_{i-1}};
\draw (2,4.2) node {\scriptsize $\iddots$};
\draw (3,3.2) node {\scriptsize $\iddots$};
\heapblock{1}{3}{s_1};
\heapblock{2}{2}{s_2};
\heapblock{1}{1}{s_1};
\end{tikzpicture}
}
\caption{Subheaps corresponding to Lemma~\ref{lem:all filled in}.}\label{fig:all filled in}
\end{figure}

\begin{proof}
This follows quickly from Lemma~\ref{lem:impermissible heap configs}; all other configurations will violate $w$ being FC.
\end{proof}

As before, when representing convex subheaps of $H(w)$ for $w\in \FC(\C_n)$, we will use a rectangle with a dotted boundary to indicate the absence of an entry in the corresponding location in any representation of $H(w)$.  It is important to note that the occurrence of a such a rectangle implies that an entry from the canonical representation of $H(w)$ cannot be shifted vertically from above or below to occupy the location.

\begin{lemma}\label{lem:zigzag subword}
Let $w \in \FC(\C_{n})$.  If $1<i < n+1$ such that $H(w)$ has two consecutive occurrences of entries labeled by $s_{i}$ and there is no entry labeled by $s_{i+1}$ occurring between them, then the heap in Figure~\ref{fig:zigzag subword 1} is a convex subheap of $H(w)$.
\end{lemma}

\begin{proof}
We proceed by induction on $i$.  The base case involving $i=2$ is clear. Now, suppose $H(w)$ has two consecutive occurrences of entries labeled by $s_{i}$ and there is no entry labeled by $s_{i+1}$ occurring between them. Since there is no entry labeled by $s_{i+1}$ occurring between these entries and $w$ is FC, there must be at least two entries labeled by $s_{i-1}$ occurring between the consecutive occurrences of $s_{i}$.  For sake of a contradiction, suppose there are three or more entries in $H(w)$ labeled by $s_{i-1}$ occurring between the pair of entries labeled by $s_{i}$.  By induction, the heap in Figure~\ref{fig:zigzag subword 3} is a saturated subheap of $H(w)$.  But by Lemma~\ref{lem:all filled in}, the convex closure of the saturated subheap occurring between the top two occurrences of $s_{1}$ must be completely filled in.  This produces a convex chain that corresponds to the subword $s_{2}s_{1}s_{2}s_{1}$, which contradicts $w \in \FC(\C_{n})$.  Therefore, between the consecutive occurrences of entries labeled by $s_{i}$, there must be exactly two occurrences of an entry labeled by $s_{i-1}$.  Inductively, the heap in Figure~\ref{fig:zigzag subword 1} is a convex subheap of $H(w)$.
\end{proof}

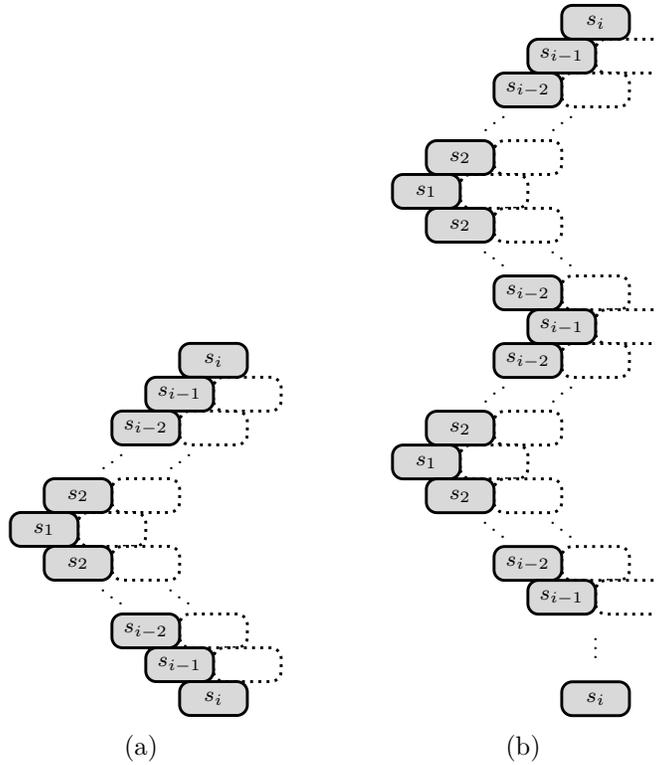
\begin{figure}[!ht]
\subcaptionbox{\label{fig:zigzag subword 1}}[.3\textwidth]{
\begin{tikzpicture}[scale=0.45]
\heapblank{6}{9}
\heapblank{4}{7}
\heapblank{3}{6}
\heapblank{4}{5}
\heapblank{6}{3}
\heapblank{7}{2}
\heapblank{7}{10}
\draw (5,8.2) node {\scriptsize $\iddots$};
\draw (5,4.2) node {\scriptsize $\ddots$};
\heapblock{6}{11}{s_{i}};
\heapblock{5}{10}{s_{i-1}};
\heapblock{4}{9}{s_{i-2}};
\draw (3,8.2) node {\scriptsize $\iddots$};
\heapblock{2}{7}{s_2};
\heapblock{1}{6}{s_1};
\heapblock{2}{5}{s_2};
\draw (3,4.2) node {\scriptsize $\ddots$};
\heapblock{4}{3}{s_{i-2}};
\heapblock{5}{2}{s_{i-1}};
\heapblock{6}{1}{s_{i}};
\end{tikzpicture}
}
\subcaptionbox{\label{fig:zigzag subword 3}}[.3\textwidth]{
\begin{tikzpicture}[scale=0.45]
\heapblank{7}{18}
\heapblank{6}{17}
\heapblank{4}{15}
\heapblank{3}{14}
\heapblank{4}{13}
\heapblank{6}{11}
\draw (5,16.2) node {\scriptsize $\iddots$};
\draw (5,12.2) node {\scriptsize $\ddots$};
\heapblank{7}{10}
\heapblank{6}{9}
\heapblank{4}{7}
\heapblank{3}{6}
\heapblank{4}{5}
\heapblank{6}{3}
\heapblank{7}{2}
\draw (5,8.2) node {\scriptsize $\iddots$};
\draw (5,4.2) node {\scriptsize $\ddots$};
\heapblock{6}{19}{s_{i}};
\heapblock{5}{18}{s_{i-1}};
\heapblock{4}{17}{s_{i-2}};
\draw (3,16.2) node {\scriptsize $\iddots$};
\heapblock{2}{15}{s_2};
\heapblock{1}{14}{s_1};
\heapblock{2}{13}{s_2};
\draw (3,12.2) node {\scriptsize $\ddots$};
\heapblock{4}{11}{s_{i-2}};
\heapblock{5}{10}{s_{i-1}};
\heapblock{4}{9}{s_{i-2}};
\draw (3,8.2) node {\scriptsize $\iddots$};
\heapblock{2}{7}{s_2};
\heapblock{1}{6}{s_1};
\heapblock{2}{5}{s_2};
\draw (3,4.2) node {\scriptsize $\ddots$};
\heapblock{4}{3}{s_{i-2}};
\heapblock{5}{2}{s_{i-1}};
\draw (6,0.7) node {\scriptsize $\vdots$};
\heapblock{6}{-1}{s_{i}};
\end{tikzpicture}
}
\caption{Subheaps corresponding to Lemma~\ref{lem:zigzag subword}.}
\end{figure}

\begin{lemma}\label{lem:typeI or nothing above}
Let $w \in \FC(\C_{n})$.  If $s_{2}s_{1}s_{2}$ is a subword of some reduced expression for $w$ and $i$ is the largest index such that the heap in Figure~\ref{fig:typeI or nothing above 1} is a saturated subheap of $H(w)$, then one or both of the following must be true:
\begin{enumerate}[label=\rm{(\arabic*)}]
\item $w$ is of type I; 
\item The subheap in Figure~\ref{fig:typeI or nothing above 2} is the northwest corner of $H(w)$, where the entry labeled by $s_{i}$ in Figure~\ref{fig:typeI or nothing above 2} is not covered.
\end{enumerate}
\end{lemma}

\begin{figure}[!ht]
\subcaptionbox{\label{fig:typeI or nothing above 1}}[.3\textwidth]{
\begin{tikzpicture}[scale=0.45]
\heapblank{3}{2}
\heapblock{6}{7}{s_{i}};
\heapblock{5}{6}{s_{i-1}};
\draw (4,5.2) node {\scriptsize $\iddots$};
\heapblock{3}{4}{s_3};
\heapblock{2}{3}{s_2};
\heapblock{1}{2}{s_1};
\heapblock{2}{1}{s_2};
\end{tikzpicture}
}
\subcaptionbox{\label{fig:typeI or nothing above 2}}[.3\textwidth]{
\begin{tikzpicture}[scale=0.45]
\heapblank{4}{7}
\heapblank{3}{6}
\heapblank{1}{4}
\draw (2,5.2) node {\scriptsize $\iddots$};
\heapblank{3}{2}
\heapblock{6}{7}{s_{i}};
\heapblock{5}{6}{s_{i-1}};
\draw (4,5.2) node {\scriptsize $\iddots$};
\heapblock{3}{4}{s_3};
\heapblock{2}{3}{s_2};
\heapblock{1}{2}{s_1};
\heapblock{2}{1}{s_2};
\end{tikzpicture}
}
\caption{Subheaps corresponding to Lemma~\ref{lem:typeI or nothing above}.}
\end{figure}
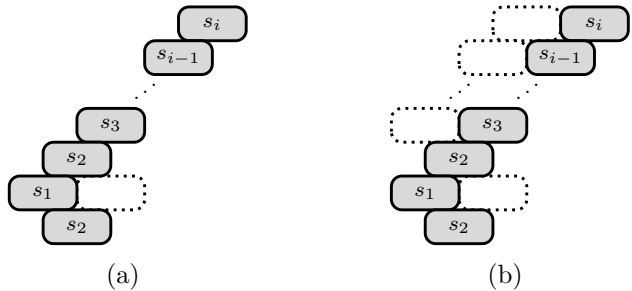

\begin{proof}
The higher entry labeled by $s_{2}$ cannot be covered by an entry labeled by $s_{1}$; otherwise, we produce one of the impermissible configurations of Lemma~\ref{lem:impermissible heap configs} and violate $w$ being FC.  Iterating, for $2\leq j\leq i-1$, each entry on the diagonal of the subheap labeled by $s_{j}$, can only be covered by an entry labeled by $s_{j+1}$.  If $i<n+1$, we are done since the entry labeled by $s_{i}$ cannot be covered by an entry labeled by $s_{i-1}$.  On the other hand, if $i=n+1$ and the entry labeled by $s_{n+1}$ at the top of the diagonal in the subheap is covered by an entry labeled by $s_{n}$, then by Lemma~\ref{lem:zigzag}, $w$ is of type I.
\end{proof}

The previous lemma has versions corresponding to the southwest, northeast, and southeast corners of $H(w)$.  As stated earlier, the purpose of the previous three lemmas was to aid in the proof of the next important lemma.  

\begin{lemma}\label{lem:main zigzag lemma}
Let $w \in \FC(\C_{n})$.  If $1<i < n+1$ such that $H(w)$ has two consecutive occurrences of entries labeled by $s_{i}$ and there is no entry labeled by $s_{i+1}$ occurring between them, then either $w$ is of type I or the heap in Figure~\ref{fig:main zigzag lemma 1} is a convex subheap of $H(w)$ and there are no other occurrences of entries labeled by $s_{1}, s_{2}, \dots, s_{i}$ in $H(w)$.
\end{lemma}

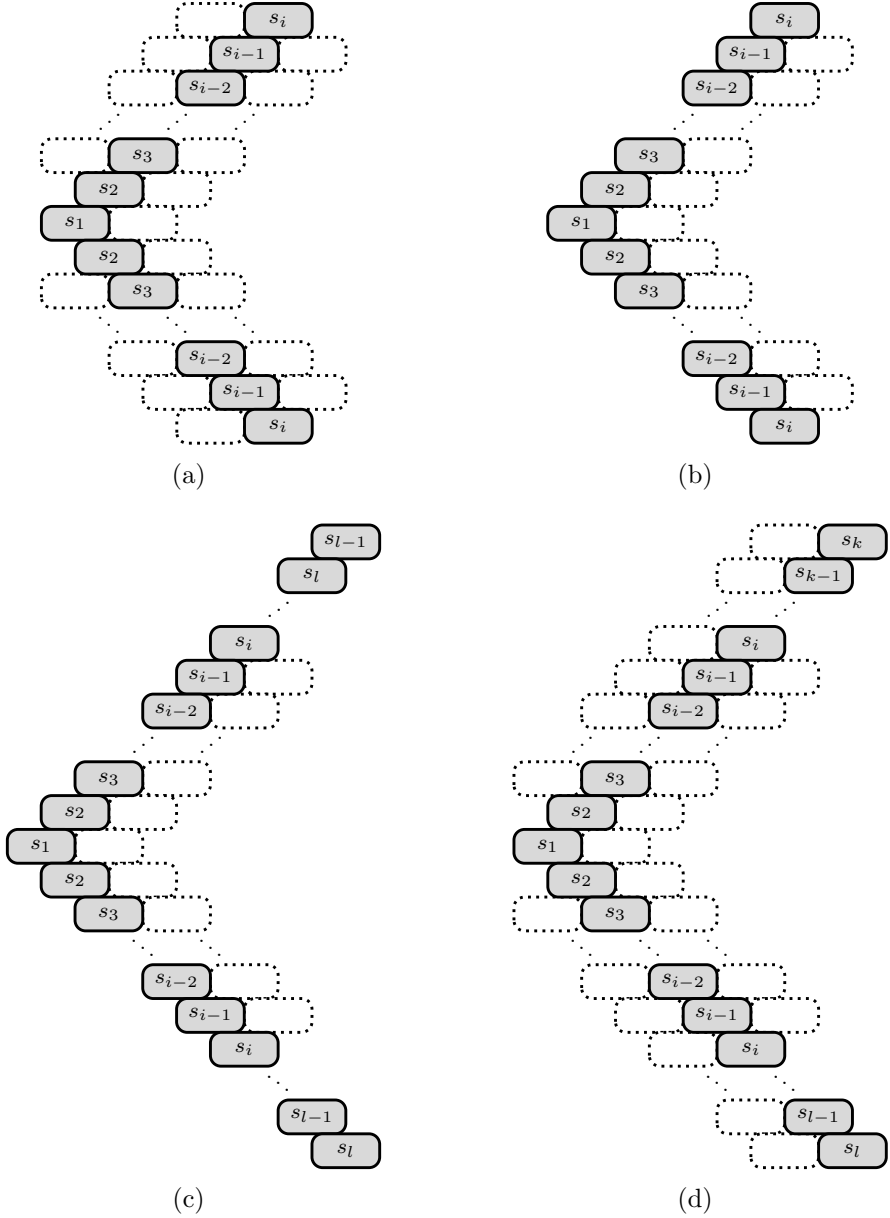
\begin{figure}[!ht]
\subcaptionbox{\label{fig:main zigzag lemma 1}}[.4\textwidth]{
\begin{tikzpicture}[scale=0.45]
\heapblank{5}{13}
\heapblank{4}{12}
\heapblank{3}{11}
\draw (2,10.2) node {\scriptsize $\iddots$};
\heapblank{1}{9}
\heapblank{1}{5}
\draw (2,4.2) node {\scriptsize $\ddots$};
\heapblank{3}{3}
\heapblank{4}{2}
\heapblank{5}{1}
\heapblank{8}{12}
\heapblank{7}{11}
\draw (6,10.2) node {\scriptsize $\iddots$};
\heapblank{5}{9}
\heapblank{4}{8}
\heapblank{3}{7}
\heapblank{4}{6}
\heapblank{5}{5}
\draw (6,4.2) node {\scriptsize $\ddots$};
\heapblank{7}{3}
\heapblank{8}{2}
\heapblock{7}{13}{s_{i}};
\heapblock{6}{12}{s_{i-1}};
\heapblock{5}{11}{s_{i-2}};
\draw (4,10.2) node {\scriptsize $\iddots$};
\heapblock{3}{9}{s_3};
\heapblock{2}{8}{s_2};
\heapblock{1}{7}{s_1};
\heapblock{2}{6}{s_2};
\heapblock{3}{5}{s_3};
\draw (4,4.2) node {\scriptsize $\ddots$};
\heapblock{5}{3}{s_{i-2}};
\heapblock{6}{2}{s_{i-1}};
\heapblock{7}{1}{s_{i}};
\end{tikzpicture}
}
\subcaptionbox{\label{fig:main zigzag lemma 2}}[.4\textwidth]{
\begin{tikzpicture}[scale=0.45]
\heapblank{8}{12}
\heapblank{7}{11}
\draw (6,10.2) node {\scriptsize $\iddots$};
\heapblank{5}{9}
\heapblank{4}{8}
\heapblank{3}{7}
\heapblank{4}{6}
\heapblank{5}{5}
\draw (6,4.2) node {\scriptsize $\ddots$};
\heapblank{7}{3}
\heapblank{8}{2}
\heapblock{7}{13}{s_{i}};
\heapblock{6}{12}{s_{i-1}};
\heapblock{5}{11}{s_{i-2}};
\draw (4,10.2) node {\scriptsize $\iddots$};
\heapblock{3}{9}{s_3};
\heapblock{2}{8}{s_2};
\heapblock{1}{7}{s_1};
\heapblock{2}{6}{s_2};
\heapblock{3}{5}{s_3};
\draw (4,4.2) node {\scriptsize $\ddots$};
\heapblock{5}{3}{s_{i-2}};
\heapblock{6}{2}{s_{i-1}};
\heapblock{7}{1}{s_{i}};
\end{tikzpicture}
}\\
\vspace{1em}
\subcaptionbox{\label{fig:main zigzag lemma 3}}[.4\textwidth]{
\begin{tikzpicture}[scale=0.45]
\heapblank{8}{12}
\heapblank{7}{11}
\draw (6,10.2) node {\scriptsize $\iddots$};
\heapblank{5}{9}
\heapblank{4}{8}
\heapblank{3}{7}
\heapblank{4}{6}
\heapblank{5}{5}
\draw (6,4.2) node {\scriptsize $\ddots$};
\heapblank{7}{3}
\heapblank{8}{2}
\heapblock{10}{16}{s_{l-1}};
\heapblock{9}{15}{s_{l}};
\draw (8,14.2) node {\scriptsize $\iddots$};
\heapblock{7}{13}{s_{i}};
\heapblock{6}{12}{s_{i-1}};
\heapblock{5}{11}{s_{i-2}};
\draw (4,10.2) node {\scriptsize $\iddots$};
\heapblock{3}{9}{s_3};
\heapblock{2}{8}{s_2};
\heapblock{1}{7}{s_1};
\heapblock{2}{6}{s_2};
\heapblock{3}{5}{s_3};
\draw (4,4.2) node {\scriptsize $\ddots$};
\heapblock{5}{3}{s_{i-2}};
\heapblock{6}{2}{s_{i-1}};
\heapblock{7}{1}{s_{i}};
\draw (8,0.2) node {\scriptsize $\ddots$};
\heapblock{9}{-1}{s_{l-1}};
\heapblock{10}{-2}{s_{l}};
\end{tikzpicture}
}
\subcaptionbox{\label{fig:main zigzag lemma 4}}[.4\textwidth]{
\begin{tikzpicture}[scale=0.45]
\heapblank{8}{16}
\heapblank{7}{15}
\draw (6,14.2) node {\scriptsize $\iddots$};
\heapblank{5}{13}
\heapblank{4}{12}
\heapblank{3}{11}
\draw (2,10.2) node {\scriptsize $\iddots$};
\heapblank{1}{9}
\heapblank{1}{5}
\draw (2,4.2) node {\scriptsize $\ddots$};
\heapblank{3}{3}
\heapblank{4}{2}
\heapblank{5}{1}
\heapblank{8}{12}
\heapblank{7}{11}
\draw (6,10.2) node {\scriptsize $\iddots$};
\heapblank{5}{9}
\heapblank{4}{8}
\heapblank{3}{7}
\heapblank{4}{6}
\heapblank{5}{5}
\draw (6,4.2) node {\scriptsize $\ddots$};
\heapblank{7}{3}
\heapblank{8}{2}
\draw (6,0.2) node {\scriptsize $\ddots$};
\heapblank{7}{-1}
\heapblank{8}{-2}
\heapblock{10}{16}{s_{k}};
\heapblock{9}{15}{s_{k-1}};
\draw (8,14.2) node {\scriptsize $\iddots$};
\heapblock{7}{13}{s_{i}};
\heapblock{6}{12}{s_{i-1}};
\heapblock{5}{11}{s_{i-2}};
\draw (4,10.2) node {\scriptsize $\iddots$};
\heapblock{3}{9}{s_3};
\heapblock{2}{8}{s_2};
\heapblock{1}{7}{s_1};
\heapblock{2}{6}{s_2};
\heapblock{3}{5}{s_3};
\draw (4,4.2) node {\scriptsize $\ddots$};
\heapblock{5}{3}{s_{i-2}};
\heapblock{6}{2}{s_{i-1}};
\heapblock{7}{1}{s_{i}};
\draw (8,0.2) node {\scriptsize $\ddots$};
\heapblock{9}{-1}{s_{l-1}};
\heapblock{10}{-2}{s_{l}};
\end{tikzpicture}
}
\caption{Subheaps corresponding to Lemma~\ref{lem:main zigzag lemma}.}
\end{figure}

\begin{proof}
Choose the largest index $i$ with $1<i<n+1$ such that $s_{i+1}$ does not occur between two consecutive occurrences of $s_{i}$ in $w$.  By Lemma~\ref{lem:zigzag subword}, the heap in Figure~\ref{fig:main zigzag lemma 2} is a convex subheap of $H(w)$.  Let $k$ be largest index with $i\leq k \leq n+1$ such that each entry labeled by $s_{j}$ on the upper diagonal in the subheap of Figure~\ref{fig:main zigzag lemma 1} covers an entry labeled by $s_{j-1}$ for $j\leq k$.  Similarly, let $l$ be the largest index with $i\leq l \leq n+1$ such that each entry on the lower diagonal labeled by $s_{j}$ is covered by an entry labeled by $s_{j-1}$ for $j\leq l$.  Then the heap in Figure~\ref{fig:main zigzag lemma 3} is a convex subheap of $H(w)$.  By the northwest and southwest versions of Lemma~\ref{lem:typeI or nothing above}, the heap in Figure~\ref{fig:main zigzag lemma 4} must be a convex subheap of $H(w)$.  If neither of $k$ nor $l$ are equal to $n+1$, then we are done since the entry labeled by $s_{k}$ (respectively, $s_{l}$) cannot be covered by (respectively, cover) an entry labeled by $s_{k-1}$ (respectively, $s_{l-1}$); otherwise, we contradict $w \in \FC(\C_{n})$. If $k=n+1$ and the entry labeled by $s_{n+1}$ is covered, it must be covered by an entry labeled by $s_{n}$. But then $w$ would then be of type I by Lemma~\ref{lem:zigzag}. The case involving $l=n+1$ follows by a similar argument.
\end{proof}

Note that all of the previous lemmas of this section have right-handed versions where $s_{1}$ and $s_{2}$ are replaced with $s_{n+1}$ and $s_{n}$, respectively.

\end{subsection}


\begin{subsection}{Weak star reductions and non-cancellable elements}\label{subsec:weak star}

The notion of a star operation was originally defined by Kazhdan and Lusztig in~\cite[\textsection 4.1]{Kazhdan1979} for simply-laced Coxeter systems (i.e., $m(s,t)\leq3$ for all $s,t \in S$) and was later generalized to arbitrary Coxeter systems in~\cite[\textsection 10.2]{Lusztig1985}.  If $I=\{s,t\}$ is a pair of noncommuting generators for $W$ of type $\Gamma$, then $I$ induces four partially defined maps from $W$ to itself, known as star operations. A star operation, when it is defined, respects the partition $W = \FC(\Gamma)\ \dot{\cup}\  (W \setminus \FC(\Gamma) )$ of the Coxeter group, and increases or decreases the length of the element to which it is applied by 1.  For our purposes, it is enough to define a weaker notion of a star operation that decreases length by 1, and so we will not develop the full generality.  

We now introduce the concept of weak star reducible, which is related to Fan's notion of cancellable in~\cite{Fan1997}.   Suppose $(W,S)$ is a Coxeter system of type $\Gamma$.  If $w \in \FC(\Gamma)$, then $w$ is \emph{left weak star reducible by $s$ with respect to $t$} to $sw$ if (i) $s\in \L(w)$, (ii) $t \in \L(sw)$ with $m(s,t) \geq 3$, and (iii) $tw \notin \FC(\Gamma)$.  Observe that Condition~(i) implies that $sw$ has length strictly smaller than $w$ while Condition~(iii) implies that $tw$ has length strictly larger than $w$.  We analogously define \emph{right weak star reducible by $s$ with respect to $t$}.  If $w$ is left (respectively, right) weak star reducible by $s$ with respect to $t$, then we refer to $w\mapsto sw$ (respectively, $w\mapsto ws$) as a \emph{left} (respectively, \emph{right}) \emph{weak star reduction}. If $w$ is either left or right weak star reducible by some $s\in S$, we say that $w$ is \emph{weak star reducible}.  Otherwise, we say that $w\in \FC(\Gamma)$ is \emph{non-cancellable}~\cite{Ernst2010}.  

The non-cancellable elements of a Coxeter group $W$ are intimately related to the two-sided cells of the generalized Temperley--Lieb algebra associated to $W$.  The connection between the non-cancellable elements and the two-sided cells has been examined for types $E_n$ and $\widetilde{A}_n$ in~\cite{Fan1997} and~\cite{Fan1999}, respectively.

\begin{example}
Suppose $\w=s_{1}s_{2}s_{1}$ and $\w'=s_{1}s_{2}$ are reduced expressions for $w,w' \in \FC(\C_{n})$, respectively.  We see that $w$ is left (respectively, right) weak star reducible by $s_{1}$ with respect to $s_{2}$ to $s_2 s_1$ (respectively, $s_1 s_2$), and so $w$ is not non-cancellable.  However, $w'$ is non-cancellable.
\end{example}

If $w\in \FC(\Gamma)$ and $s \in \L(w)$ (respectively, $\R(w)$), it is clear that $sw$ (respectively, $ws$) is still FC.  This implies that if $w \in \FC(\Gamma)$ is left or right weak star reducible to $u$, then $u$ is also FC. If $w \in \FC(\C_{n})$, then $w$ is left weak star reducible by $s$ with respect to $t$ if and only if $w=stv$ (reduced) when $m(s,t)=3$, or $w=stsv$ (reduced) when $m(s,t)=4$.   Note that this characterization applies to $\FC(B_{n})$ and $\FC(B'_{n})$, as well.  In terms of heaps, if $\w=s_{x_1}\cdots s_{x_r}$ is a reduced expression for $w \in \FC(\C_{n})$, then $w$ is left weak star reducible by $s$ with respect to $t$ if and only if (i) there is an entry in $H(\w)$ labeled by $s$ that is not covered by any other entry, and (ii) the heap $H(t\w)$ contains one of the convex subheaps of Lemma~\ref{lem:impermissible heap configs}.  Of course, we have an analogous statement for right weak star reducible.

The main result in~\cite{Ernst2010} is the classification of the non-cancellable elements in Coxeter groups of types $B_n$ and $\C_n$.  For type $B_n$, $w$ is non-cancellable if and only if $w$ is equal to either a product of commuting generators, $s_{1}s_{2}u$, or $s_{2}s_{1}u$, where $u$ is a product of commuting generators with $s_{1}, s_{2}, s_{3} \notin \supp(u)$.  We have an analogous statement for type $B'_n$, where $s_{1}$ and $s_{2}$ are replaced with $s_{n+1}$ and $s_{n}$, respectively.  The type $\C_n$ non-cancellable elements are comprised of reduced products consisting of a type $B_n$ non-cancellable element times a type $B'_n$ non-cancellable element with non-overlapping supports, the type I elements having left and right descent sets equal to one of the end generators, and all of the type II elements. 

With respect to weak star reductions, computation involving monomial basis elements is ``well-behaved", as the next remark illustrates.

\begin{remark}\label{rem:monomial weak star reductions}
If $w \in \FC(\C_{n})$ is left weak star reducible by $s$ with respect to $t$, then $w=stv$ (reduced) when $m(s,t)=3$ and $w=stsv$ (reduced) when $m(s,t)=4$.  In this case, we have
\[
b_{t}b_{w}=\begin{cases}
b_{tv},   & \text{if } m(s,t)=3\\
2b_{tsv},   & \text{if } m(s,t)=4.
\end{cases}
\]
It is important to note that $\ell(tv)=\ell(w)-1$ when $m(s,t)=3$ and $\ell(tsv)=\ell(w)-1$ when $m(s,t)=4$.  We have a similar characterization for right weak star reducibility.  
\end{remark}

It is tempting to think that if $b_{w}$ is a monomial basis element such that $b_{t}b_{w}=2^{c}b_{y}$, where $c\in \{0,1\}$ and $\ell(y)<\ell(w)$, then $w$ is weak star reducible by some $s$ with respect to $t$, where $m(s,t)\geq 3$.  However, this is not true.  For example, let $w=s_{1}s_{2}s_{3}s_{4} \in \FC(\C_{n})$ with $n \geq 3$, so that $m(s_{2},s_{3})=3$, and let $t=s_{3}$.  Then
{\allowdisplaybreaks
\begin{align*}
b_{t}b_{w} &= b_{3}b_{1}b_{2}b_{3}b_{4} \\
&= b_{1}b_{3}b_{2}b_{3}b_{4}\\
&= b_{1}b_{3}b_{4} \\
&= b_{s_{1}s_{3}s_{4}}.
\end{align*}}%
We see that $\ell(s_{1}s_{3}s_{4}) < \ell(w)$, but $w$ is not left weak star reducible by $s_{3}$ (or any generator).  

The next lemma is useful for reversing the multiplication of monomials corresponding to weak star reductions.

\begin{lemma}\label{lem:weak star reverse}
Let $w \in \FC(\C_{n})$. If $w$ is left weak star reducible by $s$ with respect to $t$, then
\[
b_{s}b_{t}b_{w}=\begin{cases}
b_{w},   & \text{if } m(s,t)=3\\
2b_{w},   & \text{if } m(s,t)=4.
\end{cases}
\]
We have an analogous statement if $w$ is right weak star reducible by $s$ with respect to $t$.
\end{lemma}

\begin{proof}
We can write $w=stv$ (reduced) when $m(s,t)=3$ or $w=stsv$ (reduced) when $m(s,t)=4$, so that
\[
b_{t}b_{w}=\begin{cases}
b_{tv},   & \text{if } m(s,t)=3\\
2b_{tsv},   & \text{if } m(s,t)=4
\end{cases}
\]
by Remark~\ref{rem:monomial weak star reductions}. Therefore, we have
{\allowdisplaybreaks
\begin{align*}
b_{s}b_{t}b_{w}=&\begin{cases}
b_{s}b_{tv},   & \text{if } m(s,t)=3\\
2b_{s}b_{tsv},   & \text{if } m(s,t)=4,
\end{cases}\\
=& \begin{cases}
b_{stv},   & \text{if } m(s,t)=3\\
2b_{stsv},   & \text{if } m(s,t)=4,
\end{cases}\\
=& \begin{cases}
b_{w},   & \text{if } m(s,t)=3\\
2b_{w},   & \text{if } m(s,t)=4.
\end{cases}
\end{align*}}%
\end{proof}

\end{subsection}

\end{section}


\begin{section}{Diagram algebras}\label{sec:diagram algebras}

This section summarizes Sections~3--5 of~\cite{Ernst2012} and its goal is to familiarize the reader with the necessary background on diagram algebras.  We refer the reader to~\cite{Ernst2008,Ernst2012} for additional details.  Our diagram algebras possess many of the same features as those already appearing in the literature, however the typical developments are too restrictive to accomplish the task of finding a faithful diagrammatic representation of the infinite dimensional Temperley--Lieb algebra of type $\C_n$.  Yet, our approach is modeled after~\cite{Green2003,Green2012,Jones1999,Martin2007}.


\begin{subsection}{Undecorated diagrams}

The \emph{standard $k$-box} is a rectangle with $2k$ marks points, called \emph{nodes} (or \emph{vertices}) labeled as in Figure~\ref{fig:k-box}.  We will refer to the top of the rectangle as the \emph{north face} and the bottom as the \emph{south face}.  Sometimes, it will be useful for us to think of the standard $k$-box as being embedded in the plane.  In this case, we put the lower left corner of the rectangle at the origin such that each node $i$ (respectively, $i'$) is located at the point $(i,1)$ (respectively, $(i,0)$).

\begin{figure}[!ht]
\centering
\begin{tikzpicture}[scale=.75]
\draw[gray,thick] (0,0) rectangle (6,1.5);
\foreach \x in {1,2,5} \filldraw (\x,0) circle (1pt);
\foreach \x in {1,2,5} \filldraw (\x,1.5) circle (1pt);
\draw (1,1.5) node[above]{\tiny $1$};
\draw (2,1.5) node[above]{\tiny $2$};
\draw (5,1.5) node[above]{\tiny $k$};
\draw (1,0) node[below]{\tiny $1'$};
\draw (2,0) node[below]{\tiny $2'$};
\draw (5,0) node[below]{\tiny $k'$};
\draw (3.5,1.5) node[above]{\tiny $\cdots$};
\draw (3.5,0) node[below]{\tiny $\cdots$};
\end{tikzpicture}
\caption{Standard $k$-box.}\label{fig:k-box}
\end{figure}
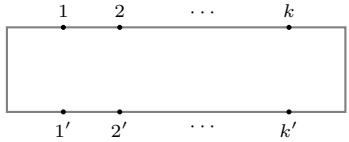

A \emph{concrete pseudo $k$-diagram} consists of a finite number of disjoint curves (planar), called \emph{edges}, embedded in the standard $k$-box.  Edges may be closed (isotopic to circles), but not if their endpoints coincide with the nodes of the box.  The nodes of the box are the endpoints of curves, which meet the box transversely.  Otherwise, the curves are disjoint from the box.  We define an equivalence relation on the set of concrete pseudo $k$-diagrams.  Two concrete pseudo $k$-diagrams are (\emph{isotopically}) \emph{equivalent} if one concrete diagram can be obtained from the other by isotopically deforming the edges such that any intermediate diagram is also a concrete pseudo $k$-diagram.  A \emph{pseudo $k$-diagram} is defined to be an equivalence class of equivalent concrete pseudo $k$-diagrams.  We denote the set of pseudo $k$-diagrams by $T_{k}(\emptyset)$.

When representing a pseudo $k$-diagram with a drawing, we pick an arbitrary concrete representative among a continuum of equivalent choices.  When no confusion can arise, we will not make a distinction between a concrete pseudo $k$-diagram and the equivalence class that it represents.  

We will refer to a closed curve occurring in the pseudo $k$-diagram as a \emph{loop edge}, or simply a \emph{loop}.  The diagram in Figure~\ref{fig:equivalent pseudo diagrams} has a single loop.  Note that we used the word ``pseudo'' in our definition to emphasize that we allow loops to appear in our diagrams.  Most examples of diagram algebras occurring in the literature ``scale away'' loops that appear.  There are loops in the diagram algebra that we are interested in preserving, so as to obtain infinitely many diagrams.  The presence of $\emptyset$ in the notation $T_{k}(\emptyset)$ is to emphasize that the edges of the diagrams are undecorated.

Let $d$ be a diagram.  If $d$ has an edge $e$ that joins node $i$ in the north face to node $j'$ in the south face, then $e$ is called a \emph{propagating edge from $i$ to $j'$}.  Propagating edges are often referred to as ``through strings'' in the literature.  If a propagating edge joins $i$ to $i'$, then we will call it a \emph{vertical propagating edge}.  If an edge is not propagating, loop edge or otherwise, it will be called \emph{non-propagating}.

\begin{example}\label{ex:equivalent pseudo diagrams}
The two concrete pseudo $5$-diagrams in Figure~\ref{fig:equivalent pseudo diagrams} are equivalent since each diagram can be obtained from the other by isotopically deforming the edges. Either diagram may be used to represent the equivalence class of the corresponding pseudo $5$-diagram. In this case, the diagram contains one loop and a single propagating edge.
\end{example}

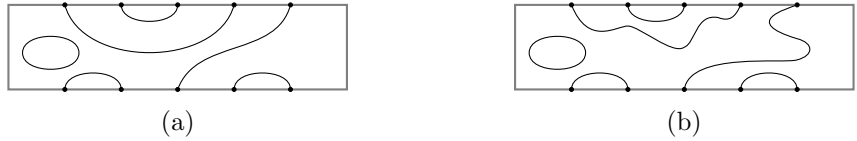
\begin{figure}[!h]
\centering
\begin{subfigure}[b]{0.4\linewidth}
\centering
\begin{tikzpicture}[scale=.75]
\draw[gray,thick] (0,0) rectangle (6,1.5);
\foreach \x in {1,2,3,4,5} \filldraw (\x,0) circle (1pt);
\foreach \x in {1,2,3,4,5} \filldraw (\x,1.5) circle (1pt);
\draw[out=-75,in=-105] (1,1.5) to (4,1.5);
\draw[out=-90,in=-90] (2,1.5) to (3,1.5);
\draw[out=90,in=90] (1,0) to (2,0);
\draw[out=90,in=90] (4,0) to (5,0);
\draw (3,0) to [out=70,in=-105] (5,1.5);
\lp{.25}{.65};
\end{tikzpicture}
\caption{}
\end{subfigure}
\begin{subfigure}[b]{0.4\linewidth}
\centering
\begin{tikzpicture}[scale=.75]
\draw[gray,thick] (0,0) rectangle (6,1.5);
\foreach \x in {1,2,3,4,5} \filldraw (\x,1.5) circle (1pt);
\foreach \x in {1,2,3,4,5} \filldraw (\x,0) circle (1pt);
\draw[out=-90,in=-90] (2,1.5) to (3,1.5);
\draw[out=90,in=90] (1,0) to (2,0);
\draw[out=90,in=90] (4,0) to (5,0);
\draw[yscale=.75] (1,2)  .. controls (1.3,1) and (1.8,1.5)  .. (2,1.5) .. controls (2.2,1.5) and (2.8,.8) .. (3,1) .. controls (3.2,1.1) and (3.2,1.8) .. (3.6,1.7) .. controls (3.8,1.6) and (3.9,1.7) .. (4,2);
\draw[yscale=.75] (5,2) .. controls (4.8,1.9) and (4,1.6) .. (5,1.2) .. controls (5.3,1.1) and (5.3,0.8) .. (5,0.7) .. controls (4.8,0.6) and (3.1,0.9) .. (3,0);
\lp{.25}{.65};
\end{tikzpicture}
\caption{}
\end{subfigure}
\caption{Isotopically equivalent concrete pseudo $5$-diagrams.}\label{fig:equivalent pseudo diagrams}
\end{figure}

If a diagram $d$ has at least one propagating edge, then we say that $d$ is \emph{dammed}.  If, on the other hand, $d$ has no propagating edges (which can only happen if $k$ is even), then we say that $d$ is \emph{undammed}.  Note that the number of non-propagating edges in the north face of a diagram must be equal to the number of non-propagating edges in the south face.  We define the function $\a: T_{k}(\emptyset) \to \Z^{+}\cup \{0\}$ via
\[
\a(d)=\text{ number of non-propagating edges in the north face of } d.
\]
There is only one diagram with $\a$-value $0$ having no loops; namely the diagram $d_{e}$ that appears in Figure~\ref{fig:a-value0}.  The maximum value that $\a(d)$ can take is $\lfloor k/2 \rfloor$.  In particular, if $k$ is even, then the maximum value that $\a(d)$ can take is $k/2$, i.e., $d$ is undammed.  On the other hand, if $\a(d)=\lfloor k/2 \rfloor$ while $k$ is odd, then $d$ has a unique propagating edge.

\begin{figure}[!ht]
\centering
\begin{tikzpicture}[scale=.75]
\draw[gray,thick] (0,0) rectangle (6,1.5);
\foreach \x in {1,2,5} \filldraw (\x,0) circle (1pt);
\foreach \x in {1,2,5} \filldraw (\x,1.5) circle (1pt);
\foreach \x in {1,2,5} \draw (\x,0) to (\x,1.5);
\draw (1,1.5) node[above]{\tiny $1$};
\draw (2,1.5) node[above]{\tiny $2$};
\draw (5,1.5) node[above]{\tiny $k$};
\draw (1,0) node[below]{\tiny $1'$};
\draw (2,0) node[below]{\tiny $2'$};
\draw (5,0) node[below]{\tiny $k'$};
\draw (3.5,1.5) node[above]{\tiny $\cdots$};
\draw (3.5,0) node[below]{\tiny $\cdots$};
\end{tikzpicture}
\caption{Only diagram having $\a$-value 0 and no loops.}\label{fig:a-value0}
\end{figure}
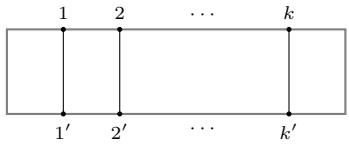

We wish to define an associative algebra that has the pseudo $k$-diagrams as a basis. Let $R$ be a commutative ring with $1$.  The associative algebra $\P_{k}(\emptyset)$ over $R$ is the free $R$-module having $T_{k}(\emptyset)$ as a basis, with multiplication defined as follows.  If $d, d' \in T_{k}(\emptyset)$, the product $d'd$ is the element of $T_{k}(\emptyset)$ obtained by placing $d'$ on top of $d$, so that node $i'$ of $d'$ coincides with node $i$ of $d$, rescaling vertically by a factor of $1/2$ and then applying the appropriate translation to recover a standard $k$-box.  For a proof that this procedure does in fact define an associative algebra see~\cite{Green2003,Jones1999}.

The (ordinary) Temperley--Lieb diagram algebra (see~\cite{Green1998a, Green2003, Jones1999, Penrose1971}) can be easily defined in terms of this formalism. Let $\DTL(A_{n})$ be the associative $\Z[\delta]$-algebra equal to the quotient of $\P_{n+1}(\emptyset)$ by the relation depicted in Figure~\ref{fig:loop}.

\begin{figure}[!ht]
\centering
$\tikz[baseline=-0.5ex,scale=.7]{\lp{1}{0};}  = \delta$
\caption{Defining relation of $\DTL(A_{n})$.}\label{fig:loop}
\end{figure}

It is well known that $\DTL(A_{n})$ is the free $\Z[\delta]$-module with basis given by the elements of $T_{n+1}(\emptyset)$ having no loops. The multiplication is inherited from the multiplication on $\P_{n+1}(\emptyset)$ except we multiply by a factor of $\delta$ for each resulting loop and then discard the loop.  We will refer to $\DTL(A_{n})$ as the \emph{type $A$ Temperley--Lieb diagram algebra}.

As $\Z[\delta]$-algebras, the Temperley--Lieb algebra $\TL(A_{n})$ that was briefly discussed in Section~\ref{sec:intro} is isomorphic to $\DTL(A_{n})$.  Moreover, each loop-free diagram from $T_{n+1}(\emptyset)$ corresponds to a unique monomial basis element of $\TL(A_{n})$.  For more details, see~\cite{Kauffman1987,Penrose1971}.

\end{subsection}


\begin{subsection}{Decorated diagrams}

We wish to adorn the edges of a diagram with elements from an associative algebra having a basis containing $1$.  First, we need to develop some terminology and lay out a few restrictions on how we decorate our diagrams.

Let $\Omega=\{\bcirc, \btri, \wcirc, \wtri\}$ and consider the free monoid $\Omega^{*}$.  We will use the elements of $\Omega$ to adorn the edges of a diagram and we will refer to each element of $\Omega$ as a \emph{decoration}.  In particular, $\bcirc$ and $\btri$ are called \emph{closed decorations}, while $\wcirc$ and $\wtri$ are called \emph{open decorations}.   Let $\mathbf{b}=x_{1}x_{2}\cdots x_{r}$ be a finite sequence of decorations in $\Omega^{*}$.  We say that $x_{i}$ and $x_{j}$ are \emph{adjacent} in $\mathbf{b}$ if $|i-j|=1$ and we will refer to $\mathbf{b}$ as a \emph{block} of decorations of \emph{width} $r$.  Note that a block of width $1$ is a just a single decoration.  The string $\bcirc~\bcirc~\btri~\wcirc~\bcirc~\wtri~\bcirc$ is an example of a block of width 7 from $\Omega^*$.

We have several restrictions for how we allow the edges of a diagram to be decorated, which we will now outline.  Let $d$ be a fixed concrete pseudo $k$-diagram and let $e$ be an edge of $d$.

\begin{enumerate}[label=\rm{(D0)}]

\item \label{D0} If $\a(d)=0$, then $e$ is undecorated.

\end{enumerate}
In particular, the unique diagram $d_{e}$ with $\a$-value 0 and no loops is undecorated.

Subject to some restrictions, if $\a(d)>0$, we may adorn $e$ with a finite sequence of blocks of decorations $\mathbf{b}_{1}, \dots, \mathbf{b}_{m}$ such that adjacency of blocks and decorations of each block is preserved as we travel along $e$.  

If $e$ is a non-loop edge, the convention we adopt is that the decorations of the block are placed so that we can read off the sequence of decorations from left to right as we traverse $e$ from $i$ to $j'$ if $e$ is propagating, or from $i$ to $j$ (respectively, $i'$ to $j'$) with $i < j$ (respectively, $i' < j'$) if $e$ is non-propagating.

If $e$ is a loop edge, reading the corresponding sequence of decorations depends on an arbitrary choice of starting point and direction round the loop. We say two sequences of blocks are \emph{loop equivalent} if one can be changed to the other or its opposite by any cyclic permutation. Note that loop equivalence is an equivalence relation on the set of sequences of blocks.  So the sequence of blocks on a loop is only defined up to loop equivalence.  That is, if we adorn a loop edge with a sequence of blocks of decorations, we only require that adjacency be preserved.

Each decoration $x_{i}$ on $e$  has coordinates in the $xy$-plane.  In particular, each decoration has an associated $y$-value, which we will call its \emph{vertical position}.  

If $\a(d)\neq 0$, then we also require the following:

\begin{enumerate}[label=\rm{(D\arabic*)}]

\item \label{D1} All decorated edges can be deformed so as to take closed decorations to the left wall of the diagram and open decorations to the right wall simultaneously without crossing any other edges.

\item \label{D2} If $e$ is non-propagating (loop edge or otherwise), then we allow adjacent blocks on $e$ to be conjoined to form larger blocks.

\item \label{D3} If $\a(d)>1$ and $e$ is propagating, then as in~\ref{D2}, we allow adjacent blocks on $e$ to be conjoined to form larger blocks.

\item \label{unusual} If $\a(d)=1$ and $e$ is propagating, then we allow $e$ to be decorated subject to the following constraints.

\begin{enumerate}
\item All decorations occurring on propagating edges must have vertical position lower (respectively, higher) than the vertical positions of decorations occurring on the (unique) non-propagating edge in the north face (respectively, south face) of $d$.

\item If $\mathbf{b}$ is a block of decorations occurring on $e$, then no other decorations occurring on any other propagating edges may have vertical position in the range of vertical positions that $\mathbf{b}$ occupies.

\item If $\mathbf{b}_{i}$ and $\mathbf{b}_{i+1}$ are two adjacent blocks occurring on $e$, then they may be conjoined to form a larger block only if the previous requirements are not violated.
\end{enumerate}
\end{enumerate}
A \emph{concrete LR-decorated pseudo $k$-diagram} is any concrete $k$-diagram decorated by elements of $\Omega$ that satisfies conditions~\ref{D0}--\ref{unusual}.

Requirement~\ref{D1} is related to the concept of ``exposed'' that appears in context of the Temperley--Lieb algebra of type $B_n$~\cite{Green1998a,Green2001,Green2003}.  The general idea is to mimic what happens in the type $B_n$ case on both the east and west ends of the diagrams.  Note that~\ref{unusual} is an unusual requirement for decorated diagrams.  We require this feature to ensure faithfulness of our diagrammatic representation on the monomial basis elements of $\TL(\C_{n})$ indexed by the type I elements. We call a block \emph{maximal} if its width cannot be increased by conjoining it with another block without violating~\ref{unusual}.

\begin{example}\label{ex:decorated diagrams}
Here are a few examples.
\begin{enumerate}
\item \label{ex:decorated diagram1} The diagram in Figure~\ref{fig:LR-decorated1} is an example of a concrete LR-decorated pseudo $5$-diagram.  In this diagram, there are no restrictions on the relative vertical position of decorations since the $\a$-value is greater than 1.  The decorations on the unique propagating edge can be conjoined to form a maximal block of width 4.

\item \label{ex:decorated diagram2} The diagram in Figure~\ref{fig:LR-decorated2} is another example of a concrete LR-decorated pseudo $5$-diagram, but with $\a$-value 1.  We use the horizontal dotted lines to indicate that the three closed decorations on the leftmost propagating edge are in three distinct blocks.  We cannot conjoin these three decorations to form a single block because there are decorations on the rightmost propagating edge occupying vertical positions between them.  Similarly, the open decorations on the rightmost propagating edge form two distinct blocks that may not be conjoined.

\item \label{ex:decorated diagram3} Lastly, the diagram in Figure~\ref{fig:LR-decorated3} is an example of a concrete LR-decorated pseudo $6$-diagram with maximal $\a$-value and no propagating edges.
\end{enumerate}
\end{example}

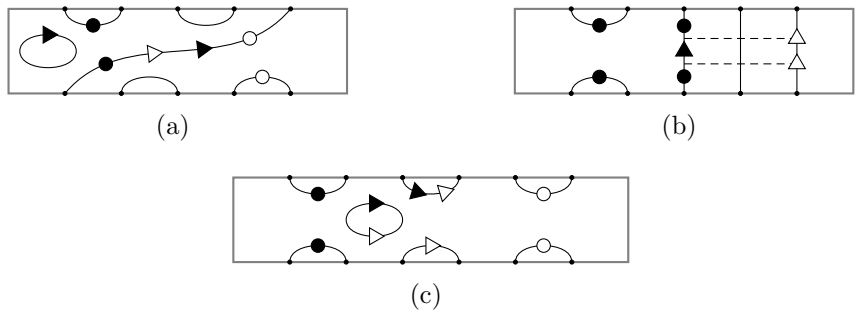
\begin{figure}[!ht]
\subcaptionbox{\label{fig:LR-decorated1}}[.4\linewidth]{
\begin{tikzpicture}[scale=.75]
\draw[gray,thick] (0,0) rectangle (6,1.5);
\foreach \x in {1,2,3,4,5} \filldraw (\x,0) circle (1pt);
\foreach \x in {1,2,3,4,5} \filldraw (\x,1.5) circle (1pt);
\draw[out=-90,in=-90] (1,1.5) to node[blackcirc, pos=0.5]{} (2,1.5);
\draw[out=-90,in=-90] (3,1.5) to (4,1.5); 
\draw[out=90,in=90] (2,0) to (3,0);
\draw[out=90,in=90] (4,0) to 
node[whitecirc, pos=0.5]{}
(5,0);
\draw[out=50,in=-130] (1,0) to
	node[blackcirc, pos=0.2]{}
	node[whitetri, pos=0.4]{}
	node[blacktri, pos=0.6]{}
	node[whitecirc, pos=0.8]{}
	(5,1.5);
\blacktrilp{.2}{.75};
\end{tikzpicture}}
\subcaptionbox{\label{fig:LR-decorated2}}[.4\linewidth]{
\begin{tikzpicture}[scale=.75]
\draw[gray,thick] (0,0) rectangle (6,1.5);
\foreach \x in {1,2,3,4,5} \filldraw (\x,0) circle (1pt);
\foreach \x in {1,2,3,4,5} \filldraw (\x,1.5) circle (1pt);
\draw[densely dashed] (3,.525) to (5,.525);
\draw[densely dashed] (3,.975) to (5,.975);
\draw[out=-90,in=-90] (1,1.5) to node[blackcirc, pos=0.5]{} (2,1.5);
\draw[out=90,in=90] (1,0) to node[blackcirc, pos=0.5]{} (2,0);
\draw (3,0) to node[blackcirc, pos=0.2]{} node[blacktri, pos=0.5]{} node[blackcirc, pos=0.8]{}(3,1.5);
\draw (4,0) to (4,1.5);
\draw (5,0) to node[whitetri, pos=0.35]{} node[whitetri, pos=0.65]{} (5,1.5);
\end{tikzpicture}}\\
\vspace{1em}
\subcaptionbox{\label{fig:LR-decorated3}}[.6\linewidth]{
\centering
\begin{tikzpicture}[scale=.75]
\draw[gray,thick] (0,0) rectangle (7,1.5);
\foreach \x in {1,2,3,4,5,6} \filldraw (\x,0) circle (1pt);
\foreach \x in {1,2,3,4,5,6} \filldraw (\x,1.5) circle (1pt);
\draw[out=-90,in=-90] (1,1.5) to node[blackcirc, pos=0.5]{} (2,1.5);
\draw[out=-90,in=-90] (3,1.5) to node[blacktri, pos=0.33]{} node[whitetri, pos=0.67]{} (4,1.5);
\draw[out=-90,in=-90] (5,1.5) to node[whitecirc, pos=0.5]{} (6,1.5);
\draw[out=90,in=90] (1,0) to node[blackcirc, pos=0.5]{} (2,0);
\draw[out=90,in=90] (3,0) to node[whitetri, pos=0.5]{} (4,0);
\draw[out=90,in=90] (5,0) to node[whitecirc, pos=0.5]{} (6,0);
\blackwhitetrilp{2}{.75};
\end{tikzpicture}}
\caption{Examples of concrete LR-decorated pseudo diagrams.}
\end{figure}

Note that an isotopy of a concrete LR-decorated pseudo $k$-diagram $d$ that preserves the faces of the standard $k$-box may not preserve the relative vertical position of the decorations even if it is mapping $d$ to an equivalent diagram.  The only time equivalence is an issue is when $\a(d)=1$.  In this case, we wish to preserve the relative vertical position of the blocks.  We define two concrete pseudo LR-decorated $k$-diagrams to be \emph{$\Omega$-equivalent} if we can isotopically deform one diagram into the other such that any intermediate diagram is also a concrete LR-decorated pseudo $k$-diagram.  Note that we do allow decorations from the same maximal block to pass each other's vertical position (while maintaining adjacency).  

An \emph{LR-decorated pseudo $k$-diagram} is defined to be an equivalence class of $\Omega$-equivalent concrete LR-decorated pseudo $k$-diagrams.  We denote the set of LR-decorated diagrams by $T_{k}^{LR}(\Omega)$. As with pseudo $k$-diagrams, when representing an LR-decorated pseudo $k$-diagram with a drawing, we pick an arbitrary concrete representative among a continuum of equivalent choices.  When no confusion will arise, we will not make a distinction between a concrete LR-decorated pseudo $k$-diagram and the equivalence class that it represents. 

\begin{remark}\label{rem:LR-decorated}
We make several observations.
\begin{enumerate}
\item The set of LR-decorated diagrams $T_{k}^{LR}(\Omega)$ is infinite since there is no limit to the number of loops that may appear.

\item \label{rem:closed left open right} If $d$ is an undammed LR-decorated diagram, then all closed decorations occurring on an edge connecting nodes in the north face (respectively, south face) of $d$ must occur before all of the open decorations occurring on the same edge as we travel the edge from the left node to the right node. Otherwise, we would not be able to simultaneously deform decorated edges to the left and right.  Furthermore, if an edge joining nodes in the north face of $d$ is adorned with an open (respectively, closed) decoration, then no non-propagating edge occurring to the right (respectively, left) in the north face may be adorned with closed (respectively, open) decorations.  We have an analogous statement for non-propagating edges in the south face.

\item \label{rem:undammed dec loops} Loops can only be decorated by both types of decorations if $d$ is undammed.  Again, we would not be able to simultaneously deform decorated edges to the left and right, otherwise.

\item \label{rem:both decs one prop} If $d$ is a dammed LR-decorated diagram, then closed decorations (respectively, open decorations) only occur to the left (respectively, right) of and possibly on the leftmost (respectively, rightmost) propagating edge.  The only way a propagating edge can have decorations of both types is if there is a single propagating edge, which can only happen if $k$ is odd.
\end{enumerate}
\end{remark}

\begin{example}
The diagram of Figure~\ref{fig:LR-decorated3} is an example that illustrates Conditions (\ref{rem:closed left open right}) and (\ref{rem:undammed dec loops}) of Remark~\ref{rem:LR-decorated}, while the diagram of Figure~\ref{fig:LR-decorated1} illustrates Condition~(\ref{rem:both decs one prop}).
\end{example}

We define $\P_{k}^{LR}(\Omega)$ to be the free $\Z[\delta]$-module having the LR-decorated pseudo $k$-diagrams $T_{k}^{LR}(\Omega)$ as a basis.  We define multiplication in $\P_{k}^{LR}(\Omega)$ by defining multiplication in the case where $d$ and $d'$ are basis elements, and then extend bilinearly.  To calculate the product $d'd$, concatenate $d'$ and $d$.  While maintaining $\Omega$-equivalence, conjoin adjacent blocks.  According to the discussion in Section~3.2 of~\cite{Ernst2012}, $\P_{k}^{LR}(\Omega)$ is a well-defined infinite dimensional associative $\Z[\delta]$-algebra.

Our immediate goal is to define a quotient of $\P_{k}^{LR}(\Omega)$ having relations that are determined by applying local combinatorial rules to the diagrams. 

Let $R=\Z[\delta]$ and define the algebra $\V$ to be the quotient of $R\Omega^{*}$ by the following relations:
\begin{enumerate}
\item $\bcirc~\bcirc~=\btri$;
\item $\bcirc~\btri=\btri~\bcirc~=~2~\bcirc$;
\item $\wcirc~\wcirc~=~\wtri$;
\item $\wcirc\wtri~=~\wtri\wcirc~=~2~\wcirc$.
\end{enumerate}
The algebra $\V$ is associative and has a basis consisting of the identity and all finite alternating products of open and closed decorations.  For example, in $\V$ we have
\[
\bcirc~\bcirc~\wcirc~\bcirc~\wcirc~\wcirc~\bcirc~=\btri~\wcirc~\bcirc~\wtri~\bcirc,
\]
where the expression on the right is a basis element, while the expression on the left is a block of width 7, but not a basis element.  We will refer to $\V$ as our \emph{decoration algebra}.

The point is that there is no interaction between open and closed symbols.  It turns out that if $\delta=1$, the algebra $\V$ is equal to the free product of two rank 3 Verlinde algebras.  For more details, see Chapter 7 of the author's PhD thesis~\cite{Ernst2008}.

Let $\widehat{\P}_{k}^{LR}(\Omega)$ be the associative $\Z[\delta]$-algebra equal to the quotient of $\P_{k}^{LR}(\Omega)$ by the relations depicted in Figure~\ref{fig:defining relations}, where the decorations on the edges represent adjacent decorations of the same block. Note that with the exception of the relations involving loops, multiplication in $\widehat{\P}_{k}^{LR}(\Omega)$ is inherited from the relations of the decoration algebra $\V$.   Also, observe that all of the relations are local in the sense that a single reduction only involves a single edge.  As a consequence of the relations in Figure~\ref{fig:defining relations}, we also have the relations of Figure~\ref{fig:additional relations}.

\begin{figure}[!ht]
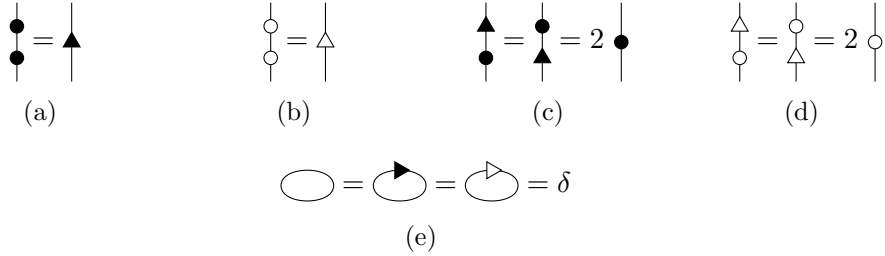

\subcaptionbox{}[.15\linewidth]{
\centering
$\tikz[baseline=-0.5ex,scale=.7]{\draw (0,-.75) to node[blackcirc, pos=0.3]{} node[blackcirc, pos=0.7]{} (0,.75);} = \tikz[baseline=-0.5ex,scale=.7]{\draw (0,-.75) to node[blacktri, pos=0.5]{} (0,.75);}$
}
\qquad
\subcaptionbox{}[.15\linewidth]{
\centering
$\tikz[baseline=-0.5ex,scale=.7]{\draw (0,-.75) to node[whitecirc, pos=0.3]{} node[whitecirc, pos=0.7]{} (0,.75);} = \tikz[baseline=-0.5ex,scale=.7]{\draw (0,-.75) to node[whitetri, pos=0.5]{} (0,.75);}$
}
\qquad
\subcaptionbox{}[.15\linewidth]{
\centering
$\tikz[baseline=-0.5ex,scale=.7]{\draw (0,-.75) to node[blackcirc, pos=0.3]{} node[blacktri, pos=0.7]{} (0,.75);} = \tikz[baseline=-0.5ex,scale=.7]{\draw (0,-.75) to node[blacktri, pos=0.3]{} node[blackcirc, pos=0.7]{} (0,.75);} = 2\ \tikz[baseline=-0.5ex,scale=.7]{\draw (0,-.75) to node[blackcirc, pos=0.5]{} (0,.75);}$
}
\qquad
\subcaptionbox{}[.15\linewidth]{
\centering
$\tikz[baseline=-0.5ex,scale=.7]{\draw (0,-.75) to node[whitecirc, pos=0.3]{} node[whitetri, pos=0.7]{} (0,.75);} = \tikz[baseline=-0.5ex,scale=.7]{\draw (0,-.75) to node[whitetri, pos=0.3]{} node[whitecirc, pos=0.7]{} (0,.75);} = 2\ \tikz[baseline=-0.5ex,scale=.7]{\draw (0,-.75) to node[whitecirc, pos=0.5]{} (0,.75);}$
}\\
\vspace{1em}
\subcaptionbox{}{
\centering
$\tikz[baseline=-0.5ex,scale=.7]{\lp{1}{0};} = \tikz[baseline=-0.5ex,scale=.7]{\blacktrilp{1}{0};} = \tikz[baseline=-0.5ex,scale=.7]{\whitetrilp{1}{0};}  = \delta$
}
\caption{Defining relations of $\widehat{\P}_{k}^{LR}(\Omega)$.}\label{fig:defining relations}
\end{figure}

\begin{figure}[!ht]
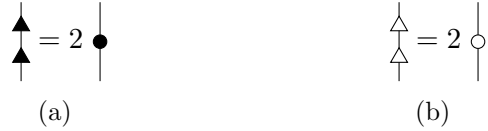

\subcaptionbox{}[.25\linewidth]{
\centering
$\tikz[baseline=-0.5ex,scale=.7]{\draw (0,-.75) to node[blacktri, pos=0.3]{} node[blacktri, pos=0.7]{} (0,.75);} = 2\ \tikz[baseline=-0.5ex,scale=.7]{\draw (0,-.75) to node[blackcirc, pos=0.5]{} (0,.75);}$
}
\qquad
\subcaptionbox{}[.25\linewidth]{
\centering
$\tikz[baseline=-0.5ex,scale=.7]{\draw (0,-.75) to node[whitetri, pos=0.3]{} node[whitetri, pos=0.7]{} (0,.75);} = 2\ \tikz[baseline=-0.5ex,scale=.7]{\draw (0,-.75) to node[whitecirc, pos=0.5]{} (0,.75);}$
}
\caption{Additional relations of $\widehat{\P}_{k}^{LR}(\Omega)$.}\label{fig:additional relations}
\end{figure}

\begin{example}
Figure~\ref{fig:ex mult1} depicts multiplication of three diagrams in $\widehat{\P}_{6}^{LR}(\Omega)$ and Figure~\ref{fig:ex mult2} shows an example where each of the diagrams and their product have $\a$-value 1.  Again, we use the dotted line to emphasize that the two closed decorations on the leftmost propagating edge belong to distinct blocks.
\end{example}

\begin{figure}[!ht]
\centering
$\begin{tikzpicture}[baseline=-0.5ex,scale=.75]
\draw[gray,thick] (0,-3) rectangle (7,-1);
\foreach \x in {1,2,3,4,5,6} \filldraw (\x,-3) circle (1pt);
\foreach \x in {1,2,3,4,5,6} \filldraw (\x,-1) circle (1pt);
\draw[out=-90,in=-90] (1,-1) to node[blackcirc, pos=0.5]{} (2,-1);
\draw[out=-90,in=-90] (3,-1) to node[whitetri, pos=0.5]{} (4,-1);
\draw[out=-90,in=-90] (5,-1) to node[whitecirc, pos=0.5]{} (6,-1);
\draw[out=90,in=90] (1,-3) to node[blackcirc, pos=0.5]{} (2,-3);
\draw[out=90,in=90] (3,-3) to (4,-3);
\draw[out=90,in=90] (5,-3) to node[whitecirc, pos=0.5]{} (6,-3);

\draw[gray,thick] (0,-1) rectangle (7,1);
\foreach \x in {1,2,3,4,5,6} \filldraw (\x,-1) circle (1pt);
\foreach \x in {1,2,3,4,5,6} \filldraw (\x,1) circle (1pt);
\draw (1,-1) to (1,1);
\draw (6,-1) to (6,1);
\draw[out=-90,in=-90] (2,1) to (3,1);
\draw[out=-90,in=-90] (4,1) to (5,1);
\draw[out=90,in=90] (2,-1) to (3,-1);
\draw[out=90,in=90] (4,-1) to (5,-1);

\draw[gray,thick] (0,1) rectangle (7,3);
\foreach \x in {1,2,3,4,5,6} \filldraw (\x,1) circle (1pt);
\foreach \x in {1,2,3,4,5,6} \filldraw (\x,3) circle (1pt);
\draw[out=-90,in=-90] (1,3) to node[blackcirc, pos=0.5]{} (2,3);
\draw[out=-90,in=-90] (3,3) to (4,3);
\draw[out=-90,in=-90] (5,3) to node[whitecirc, pos=0.5]{} (6,3);
\draw[out=90,in=90] (1,1) to node[blackcirc, pos=0.5]{} (2,1);
\draw[out=90,in=90] (3,1) to (4,1);
\draw[out=90,in=90] (5,1) to node[whitecirc, pos=0.5]{} (6,1);
\end{tikzpicture}
\ =\ 2\ \ \begin{tikzpicture}[baseline=-0.5ex,scale=.75]
\draw[gray,thick] (0,-1) rectangle (7,1);
\foreach \x in {1,2,3,4,5,6} \filldraw (\x,-1) circle (1pt);
\foreach \x in {1,2,3,4,5,6} \filldraw (\x,1) circle (1pt);
\draw[out=-90,in=-90] (1,1) to node[blackcirc, pos=0.5]{} (2,1);
\draw[out=-90,in=-90] (3,1) to (4,1);
\draw[out=-90,in=-90] (5,1) to node[whitecirc, pos=0.5]{} (6,1);
\draw[out=90,in=90] (1,-1) to node[blackcirc, pos=0.5]{} (2,-1);
\draw[out=90,in=90] (3,-1) to (4,-1);
\draw[out=90,in=90] (5,-1) to node[whitecirc, pos=0.5]{} (6,-1);
\blackwhitetrilp{2}{0};
\end{tikzpicture}$
\caption{Example of multiplication in $\widehat{\P}_{6}^{LR}(\Omega)$.}\label{fig:ex mult1}
\end{figure}
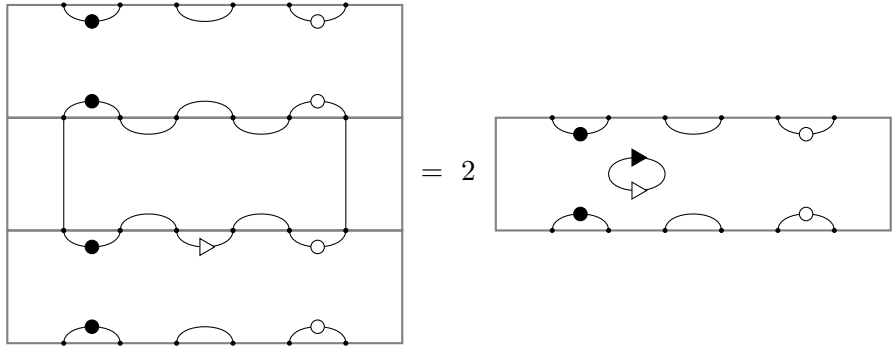 

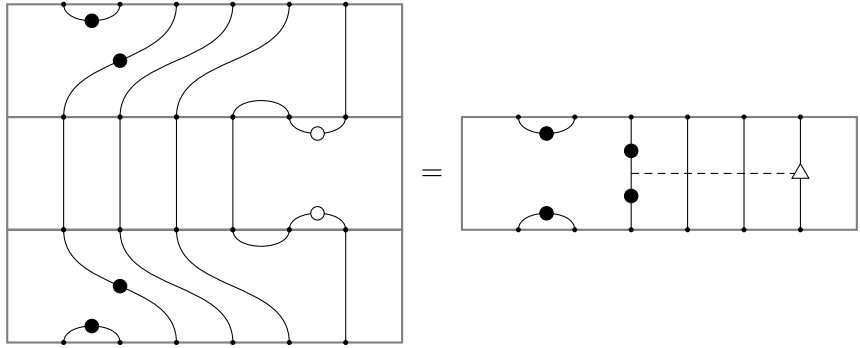
\begin{figure}[!ht]
\centering
$\begin{tikzpicture}[baseline=-0.5ex,scale=.75]
\draw[gray,thick] (0,-3) rectangle (7,-1);
\foreach \x in {1,2,3,4,5,6} \filldraw (\x,-3) circle (1pt);
\foreach \x in {1,2,3,4,5,6} \filldraw (\x,-1) circle (1pt);
\draw[out=-90,in=90] (1,-1) to node[blackcirc, pos=0.5]{} (3,-3);
\draw[out=-90,in=90] (2,-1) to (4,-3);
\draw[out=-90,in=90] (3,-1) to (5,-3);
\draw[out=-90,in=-90] (4,-1) to (5,-1);
\draw (6,-1) to (6,-3);
\draw[out=90,in=90] (1,-3) to node[blackcirc, pos=0.5]{} (2,-3);

\draw[gray,thick] (0,-1) rectangle (7,1);
\foreach \x in {1,2,3,4,5,6} \filldraw (\x,-1) circle (1pt);
\foreach \x in {1,2,3,4,5,6} \filldraw (\x,1) circle (1pt);
\draw (1,-1) to (1,1);
\draw (2,-1) to (2,1);
\draw (3,-1) to (3,1);
\draw (4,-1) to (4,1);
\draw[out=-90,in=-90] (5,1) to node[whitecirc, pos=0.5]{} (6,1);
\draw[out=90,in=90] (5,-1) to node[whitecirc, pos=0.5]{} (6,-1);

\draw[gray,thick] (0,1) rectangle (7,3);
\foreach \x in {1,2,3,4,5,6} \filldraw (\x,1) circle (1pt);
\foreach \x in {1,2,3,4,5,6} \filldraw (\x,3) circle (1pt);
\draw[out=-90,in=-90] (1,3) to node[blackcirc, pos=0.5]{} (2,3);
\draw[out=90,in=90] (4,1) to (5,1);
\draw[out=90,in=-90] (1,1) to node[blackcirc, pos=0.5]{} (3,3);
\draw[out=90,in=-90] (2,1) to (4,3);
\draw[out=90,in=-90] (3,1) to (5,3);
\draw (6,1) to (6,3);
\end{tikzpicture}
\ =\ \begin{tikzpicture}[baseline=-0.5ex,scale=.75]
\draw[gray,thick] (0,-1) rectangle (7,1);
\foreach \x in {1,2,3,4,5,6} \filldraw (\x,-1) circle (1pt);
\foreach \x in {1,2,3,4,5,6} \filldraw (\x,1) circle (1pt);
\draw[densely dashed] (3,0) to (6,0);
\draw[out=-90,in=-90] (1,1) to node[blackcirc, pos=0.5]{} (2,1);
\draw[out=90,in=90] (1,-1) to node[blackcirc, pos=0.5]{} (2,-1);
\draw (3,-1) to node[blackcirc, pos=0.3]{} node[blackcirc, pos=0.7]{}(3,1);
\draw (4,-1) to (4,1);
\draw (5,-1) to (5,1);
\draw (6,-1) to node[whitetri, pos=0.5]{} (6,1);
\end{tikzpicture}$
\caption{Example of multiplication in $\widehat{\P}_{6}^{LR}(\Omega)$ with diagrams having $\a$-value 1.}\label{fig:ex mult2}
\end{figure}

Define the \emph{simple diagrams} $d_{1}, d_{2}, \dots, d_{n+1}$ as shown in Figure~\ref{fig:simple diagrams}.  Note that each of the simple diagrams is a basis element of $\widehat{\P}_{n+2}^{LR}(\Omega)$.   Let $\DTL(\C_n)$ be the $\Z[\delta]$-subalgebra of $\widehat{\P}_{n+2}^{LR}(\Omega)$ generated (as a unital algebra) by $d_{1}, d_{2}, \dots, d_{n+1}$ with multiplication inherited from $\widehat{\P}_{n+2}^{LR}(\Omega)$. Note that $\DTL(\C_n)$ was denoted by $\D_n$ in~\cite{Ernst2012}. As we shall see in Theorem~\ref{thm:main result}, the diagram algebra $\DTL(\C_n)$ is a faithful representation of $\TL(\C_n)$.

\begin{figure}[!ht]
\begin{align*}
d_{1} & = \begin{tikzpicture}[baseline=-0.5ex,scale=.75]
\draw[gray,thick] (0,-1) rectangle (9,1);
\foreach \x in {1,2,3,8} \filldraw (\x,-1) circle (1pt);
\foreach \x in {1,2,3,8} \filldraw (\x,1) circle (1pt);
\node[above] at (1,1) {\tiny $\phantom{+}1\phantom{+}$};
\node[above] at (2,1) {\tiny $\phantom{+}2\phantom{+}$};
\node[above] at (3,1) {\tiny $\phantom{+}3\phantom{+}$};
\node[above] at (8,1) {\tiny $n+2$};
\draw (3,-1) to (3,1);
\node at (5.5,0) {\tiny $\cdots$};
\draw (8,-1) to (8,1);
\draw[out=-90,in=-90] (1,1) to node[blackcirc, pos=0.5]{} (2,1);
\draw[out=90,in=90] (1,-1) to node[blackcirc, pos=0.5]{} (2,-1);
\end{tikzpicture}\\ 
& \vdots \\
d_{i} & = \begin{tikzpicture}[baseline=-0.5ex,scale=.75]
\draw[gray,thick] (0,-1) rectangle (9,1);
\foreach \x in {1,3,4,5,6,8} \filldraw (\x,-1) circle (1pt);
\foreach \x in {1,3,4,5,6,8} \filldraw (\x,1) circle (1pt);
\node[above] at (1,1) {\tiny $\phantom{+}1\phantom{+}$};
\node[above] at (3,1) {\tiny $i-1$};
\node[above] at (4,1) {\tiny $\phantom{+}i\phantom{+}$};
\node[above] at (5,1) {\tiny $i+1$};
\node[above] at (6,1) {\tiny $i+2$};
\node[above] at (8,1) {\tiny $n+2$};
\draw (1,-1) to (1,1);
\node at (2,0) {\tiny $\cdots$};
\node at (7,0) {\tiny $\cdots$};
\draw (3,-1) to (3,1);
\draw (6,-1) to (6,1);
\draw (8,-1) to (8,1);
\draw[out=-90,in=-90] (4,1) to (5,1);
\draw[out=90,in=90] (4,-1) to (5,-1);
\end{tikzpicture}\\
& \vdots \\
d_{n+1} & = \begin{tikzpicture}[baseline=-0.5ex,scale=.75]
\draw[gray,thick] (0,-1) rectangle (9,1);
\foreach \x in {1,6,7,8} \filldraw (\x,-1) circle (1pt);
\foreach \x in {1,6,7,8} \filldraw (\x,1) circle (1pt);
\node[above] at (1,1) {\tiny $\phantom{+}1\phantom{+}$};
\node[above] at (6,1) {\tiny $\phantom{+}n-1\phantom{+}$};
\node[above] at (7,1) {\tiny $\phantom{+}n\phantom{+}$};
\node[above] at (8,1) {\tiny $n+2$};
\draw (1,-1) to (1,1);
\node at (3.5,0) {\tiny $\cdots$};
\draw (6,-1) to (6,1);
\draw[out=-90,in=-90] (7,1) to node[whitecirc, pos=0.5]{} (8,1);
\draw[out=90,in=90] (7,-1) to node[whitecirc, pos=0.5]{} (8,-1);
\end{tikzpicture}
\end{align*}
\caption{Simple diagrams.}\label{fig:simple diagrams}
\end{figure}
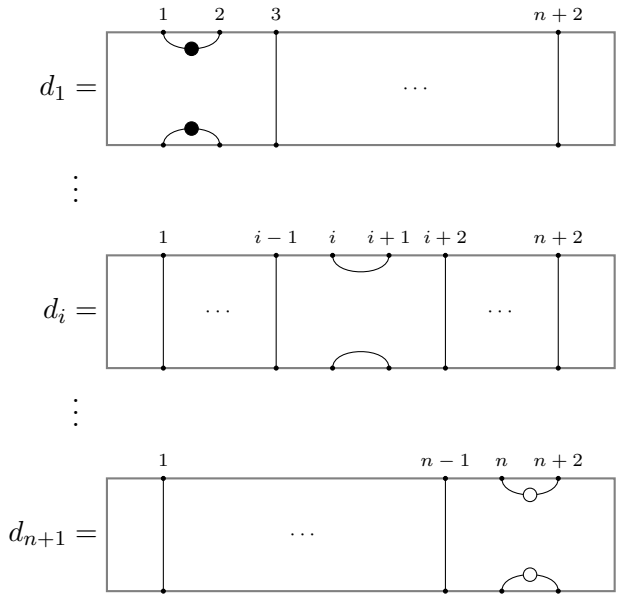

\end{subsection}


\begin{subsection}{Admissible diagrams}\label{subsec:admissible}  

We will now review the description of the basis for $\DTL(\C_n)$ that was introduced in~\cite{Ernst2012}. Let $d$ be an LR-decorated diagram.  Then we say that $d$ is \emph{$\C$-admissible}, or simply \emph{admissible}, if the following axioms are satisfied.
\begin{enumerate}[label=\rm{(C\arabic*)}]
\item \label{C1} The only loops that may appear are equivalent to the one in Figure~\ref{fig:permissible loop}.
\begin{figure}[!ht]
\centering
\begin{tikzpicture}[baseline=-0.5ex]
\blackwhitetrilp{0}{0};
\end{tikzpicture}
\caption{Only allowable loop in $\C$-admissible diagrams.}\label{fig:permissible loop}
\end{figure}
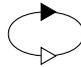

\item \label{C2} If $d$ is undammed (which can only happen if $n$ is even), then the (non-propagating) edges joining nodes $1$ and $1'$ (respectively, nodes $n+2$ and $(n+2)'$) must be decorated with a $\bcirc$ (respectively, $\wcirc$).  Furthermore, these are the only $\bcirc$ (respectively, $\wcirc$) decorations that may occur on $d$ and must be the first (respectively, last) decorations on their respective edges.

\item \label{C3} Assume $d$ has exactly one propagating edge $e$ (which can only happen if $n$ is odd). Then $e$ may be decorated by an alternating sequence of $\btri$ and $\wtri$ decorations.  If $e$ is decorated by both open and closed decorations and is connected to node 1 (respectively, $1'$), then the first (respectively, last) decoration occurring on $e$ must be a $\bcirc$.  Similarly, if $e$ is connected to node $n+2$ (respectively, $(n+2)'$), then the first (respectively, last) decoration occurring on $e$ must be a $\wcirc$.  If $e$ joins $1$ to $1'$ (respectively, $n+2$ to $(n+2)'$) and is decorated by a single decoration, then $e$ is decorated by a single $\btri$ (respectively, $\wtri$).  Furthermore, if there is a non-propagating edge connected to $1$ or $1'$ (respectively, $n+2$ or $(n+2)'$) it must be decorated only by a single $\bcirc$ (respectively, $\wcirc$).  Finally, no other $\bcirc$ or $\wcirc$ decorations appear on $d$.

\item \label{C4} Assume $d$ is dammed with $\a(d)>1$ and has more than one propagating edge.  If there is a propagating edge joining $1$ to $1'$ (respectively, $n+2$ to $(n+2)'$), then it is decorated by a single $\btri$ (respectively, $\wtri$).  Otherwise, both edges leaving either of $1$ or $1'$ (respectively, $n+2$ or $(n+2)'$) are each decorated by a single $\bcirc$ (respectively, $\wcirc$) and there are no other $\bcirc$ or $\wcirc$ decorations appearing on $d$.

\item \label{C5} If $\a(d)=1$, then the western end of $d$ is equal to one of the diagrams in Figure~\ref{fig:western C5}, where $\mathbf{B}$ represents a sequence of blocks (possibly empty) such that each block is a single $\btri$ and the diagram in Figure~\ref{fig:western C5b} can only occur if $d$ is not decorated by any open decorations.  Also, the occurrences of the $\bcirc$ decorations occurring on the propagating edge have the highest (respectively, lowest) relative vertical position of all decorations occurring on any propagating edge.  In particular, if the sequence of blocks $\mathbf{B}$ in Figure~\ref{fig:western C5d} (respectively, Figure~\ref{fig:western C5e}) is empty, then the $\bcirc$ decoration has the highest (respectively, lowest) relative vertical position among all decorations occurring on propagating edges.  We have an analogous requirement for the eastern end of $e$, where the closed decorations are replaced with open decorations.  Furthermore, if there is a non-propagating edge connected to $1$ or $1'$ (respectively, $n+2$ or $(n+2)'$) it must be decorated only by a single $\bcirc$ (respectively, $\wcirc$).  Finally, no other $\bcirc$ or $\wcirc$ decorations appear on $d$.
\end{enumerate}

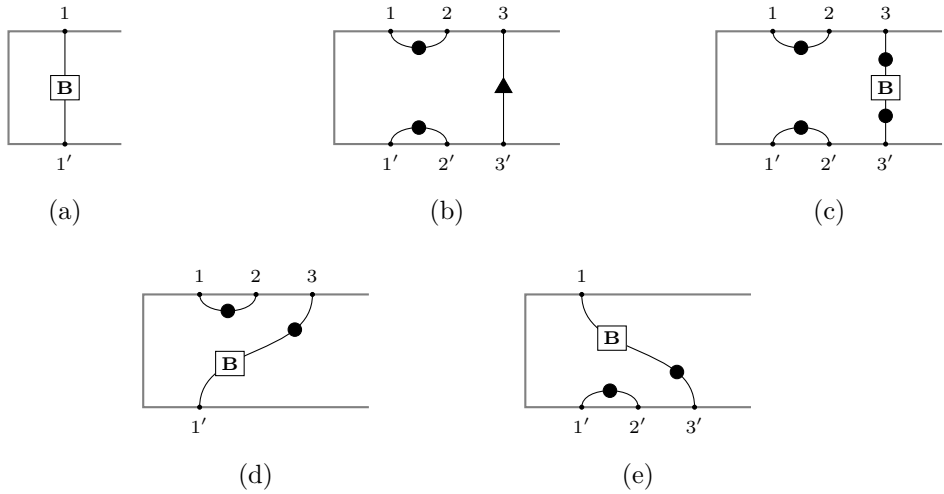
\begin{figure}[!ht]
\subcaptionbox{}[.3\linewidth]{\begin{tikzpicture}[baseline=-0.5ex,scale=.75]
\draw[gray,thick] (2,-1) to (0,-1) to (0,1) to (2,1);
\foreach \x in {1} \filldraw (\x,-1) circle (1pt);
\foreach \x in {1} \filldraw (\x,1) circle (1pt);
\node[above] at (1,1) {\tiny $\phantom{+}1\phantom{+}$};
\node[below] at (1,-1) {\tiny $\phantom{+}1'\phantom{+}$};
\draw (1,-1) to node[rectangle, draw=black, inner sep=2.2pt, fill=white, pos=0.5]{\tiny $\mathbf{B}$} (1,1);
\end{tikzpicture}
}
\subcaptionbox{\label{fig:western C5b}}[.3\linewidth]{\begin{tikzpicture}[baseline=-0.5ex,scale=.75]
\draw[gray,thick] (4,-1) to (0,-1) to (0,1) to (4,1);
\foreach \x in {1,2,3} \filldraw (\x,-1) circle (1pt);
\foreach \x in {1,2,3} \filldraw (\x,1) circle (1pt);
\node[above] at (1,1) {\tiny $\phantom{+}1\phantom{+}$};
\node[below] at (1,-1) {\tiny $\phantom{+}1'\phantom{+}$};
\node[above] at (2,1) {\tiny $\phantom{+}2\phantom{+}$};
\node[below] at (2,-1) {\tiny $\phantom{+}2'\phantom{+}$};
\node[above] at (3,1) {\tiny $\phantom{+}3\phantom{+}$};
\node[below] at (3,-1) {\tiny $\phantom{+}3'\phantom{+}$};
\draw (3,-1) to node[blacktri, pos=0.5]{} (3,1);
\draw[out=-90,in=-90] (1,1) to node[blackcirc, pos=0.5]{} (2,1);
\draw[out=90,in=90] (1,-1) to node[blackcirc, pos=0.5]{} (2,-1);
\end{tikzpicture}
}
\subcaptionbox{}[.3\linewidth]{\begin{tikzpicture}[baseline=-0.5ex,scale=.75]
\draw[gray,thick] (4,-1) to (0,-1) to (0,1) to (4,1);
\foreach \x in {1,2,3} \filldraw (\x,-1) circle (1pt);
\foreach \x in {1,2,3} \filldraw (\x,1) circle (1pt);
\node[above] at (1,1) {\tiny $\phantom{+}1\phantom{+}$};
\node[below] at (1,-1) {\tiny $\phantom{+}1'\phantom{+}$};
\node[above] at (2,1) {\tiny $\phantom{+}2\phantom{+}$};
\node[below] at (2,-1) {\tiny $\phantom{+}2'\phantom{+}$};
\node[above] at (3,1) {\tiny $\phantom{+}3\phantom{+}$};
\node[below] at (3,-1) {\tiny $\phantom{+}3'\phantom{+}$};
\draw (3,-1) to node[blackcirc, pos=0.25]{} node[rectangle, draw=black, fill=white, inner sep=2.2pt, pos=0.5]{\tiny $\mathbf{B}$} node[blackcirc, pos=0.75]{} (3,1);
\draw[out=-90,in=-90] (1,1) to node[blackcirc, pos=0.5]{} (2,1);
\draw[out=90,in=90] (1,-1) to node[blackcirc, pos=0.5]{} (2,-1);
\end{tikzpicture}
}\\
\vspace{1em}
\subcaptionbox{\label{fig:western C5d}}[.3\linewidth]{\begin{tikzpicture}[baseline=-0.5ex,scale=.75]
\draw[gray,thick] (4,-1) to (0,-1) to (0,1) to (4,1);
\foreach \x in {1} \filldraw (\x,-1) circle (1pt);
\foreach \x in {1,2,3} \filldraw (\x,1) circle (1pt);
\node[above] at (1,1) {\tiny $\phantom{+}1\phantom{+}$};
\node[below] at (1,-1) {\tiny $\phantom{+}1'\phantom{+}$};
\node[above] at (2,1) {\tiny $\phantom{+}2\phantom{+}$};
\node[above] at (3,1) {\tiny $\phantom{+}3\phantom{+}$};
\draw[out=90,in=-90] (1,-1) to node[rectangle, inner sep=2.2pt, draw=black, fill=white, pos=0.34]{\tiny $\mathbf{B}$} node[blackcirc, pos=0.75]{} (3,1);
\draw[out=-90,in=-90] (1,1) to node[blackcirc, pos=0.5]{} (2,1);
\end{tikzpicture}
}
\subcaptionbox{\label{fig:western C5e}}[.3\linewidth]{\begin{tikzpicture}[baseline=-0.5ex,scale=.75]
\draw[gray,thick] (4,-1) to (0,-1) to (0,1) to (4,1);
\foreach \x in {1,2,3} \filldraw (\x,-1) circle (1pt);
\foreach \x in {1} \filldraw (\x,1) circle (1pt);
\node[above] at (1,1) {\tiny $\phantom{+}1\phantom{+}$};
\node[below] at (1,-1) {\tiny $\phantom{+}1'\phantom{+}$};
\node[below] at (2,-1) {\tiny $\phantom{+}2'\phantom{+}$};
\node[below] at (3,-1) {\tiny $\phantom{+}3'\phantom{+}$};
\draw[out=-90,in=90] (1,1) to node[rectangle, draw=black, inner sep=2.2pt,fill=white, pos=0.34]{\tiny $\mathbf{B}$} node[blackcirc, pos=0.75]{} (3,-1);
\draw[out=90,in=90] (1,-1) to node[blackcirc, pos=0.5]{} (2,-1);
\end{tikzpicture}
}
\caption{Western end of diagrams exhibiting axiom~\ref{C5}.}\label{fig:western C5}
\end{figure}

Our definition of $\C$-admissible is motivated by the definition of $B$-admissible given by Green in~\cite[Definition 2.2.4]{Green2001} for diagrams in the context of type $B_n$.  Since the Coxeter graph of type $\C_n$ is type $B$ at ``both ends'', the general idea is to build the axioms of $B$-admissible into our definition of $\C$-admissible on the left and right sides of our diagrams.

Note that the only time an admissible diagram $d$ can have an edge adorned with both open and closed decorations is if $d$ is undammed (which only happens when $n$ is even) or if $d$ has a single propagating edge (which only happens when $n$ is odd).  The diagrams in Figure~\ref{fig:LR-decorated1} and Figure~\ref{fig:LR-decorated3} demonstrate this phenomenon (see Example~\ref{ex:decorated diagrams}).

If $d$ is an admissible diagram with $\a(d)=1$, then the restrictions on the relative vertical position of decorations on propagating edges along with axiom~\ref{C5} imply that the relative vertical positions of closed decorations on the leftmost propagating edge and open decorations on the rightmost propagating edge must alternate.  In particular, the number of closed decorations occurring on the leftmost propagating edge differs from the number of open decorations occurring on the rightmost propagating edge by at most 1.  For example, if $d$ is the diagram in Figure~\ref{fig:example C5}, where $\mathbf{B}_1$ on the leftmost propagating edge consists of $k$ blocks, each consisting of a single $\btri$ decoration, then $\mathbf{B}_2$ must consist of $k$ $\wtri$ decorations, as well. 

\begin{figure}[!ht]
\begin{tikzpicture}[baseline=-0.5ex,scale=.75]
\draw[gray,thick] (0,-1) rectangle (9,1);
\foreach \x in {1,2,5,6,7,8} \filldraw (\x,-1) circle (1pt);
\foreach \x in {1,2,3,4,7,8} \filldraw (\x,1) circle (1pt);
\draw[out=90,in=-90] (1,-1) to node[rectangle, draw=black, inner sep=2.2pt,fill=white, pos=0.4]{\tiny $\mathbf{B}_1$} node[blackcirc, pos=0.75]{} (3,1);
\draw[out=90,in=-90] (2,-1) to (4,1);
\draw[out=90,in=-90] (5,-1) to (7,1);
\draw[out=90,in=-90] (6,-1) to node[whitecirc, pos=0.25]{} node[rectangle, draw=black, inner sep=2.2pt,fill=white, pos=0.6]{\tiny $\mathbf{B}_2$} (8,1);
\node at (4.5,0) {\tiny $\cdots$};
\draw[out=-90,in=-90] (1,1) to node[blackcirc, pos=0.5]{} (2,1);
\draw[out=90,in=90] (7,-1) to node[whitecirc, pos=0.5]{} (8,-1);
\end{tikzpicture}
\caption{Example of a diagram exhibiting axiom~\ref{C5}.}\label{fig:example C5}
\end{figure}
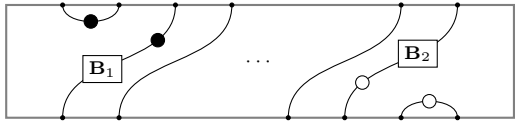

The next proposition is one of the main results of~\cite{Ernst2012}.

\begin{proposition}\label{prop:admissibles basis for D_n}
The set of admissible $(n+2)$-diagrams is a basis for $\DTL(\C_n)$. \qed
\end{proposition}

\end{subsection}

\end{section}


\begin{section}{Main results}\label{sec:main results}

This section concludes with a proof that $\TL(\C_{n})$ and $\DTL(\C_n)$ are isomorphic as $\Z[\delta]$-algebras under the correspondence induced by $b_{i} \mapsto d_{i}$.  Moreover, we show that the set of admissible diagrams corresponds to the monomial basis of $\TL(\C_{n})$.  


\begin{subsection}{A surjective homomorphism}\label{subsec:hom}

As in~\cite{Ernst2012}, define $\theta: \TL(\C_{n}) \to \DTL(\C_n)$ to be the algebra homomorphism determined by $\theta(b_{i})=d_{i}$.  The next result is~\cite[Proposition 4.1.3]{Ernst2012}.

\begin{proposition}\label{prop:surjective homomorphism}
The map $\theta$ is a surjective algebra homomorphism. \qed
\end{proposition}

To demonstrate that $\DTL(\C_n)$ is a faithful representation of $\TL(\C_n)$, it remains to show that $\theta$ is injective, which we show in Theorem~\ref{thm:main result}.

\begin{lemma}\label{lem:powers of 2 and delta for images of monomials}
If $w \in \FC(\C_{n})$, then $\theta(b_{w})=2^{k}\delta^{m}d$, where $k,m \in \Z^{+}\cup\{0\}$ and $d$ is an admissible diagram.
\end{lemma}

\begin{proof}
Suppose $\w=s_{i_{1}}\cdots s_{i_{r}}$ is a reduced expression for $w \in \FC(\C_{n})$. By repeated applications of~\cite[Proposition~5.5.1]{Ernst2012}, we have
\[
\theta(b_{w}) = \theta(b_{i_{1}}) \cdots \theta(b_{i_{r}}) = d_{i_{1}} \cdots d_{i_{r}} = 2^{k}\delta^{m}d
\]
for some $k,m \in \Z^{+}\cup\{0\}$ and admissible diagram $d$.  
\end{proof}

Since $\theta$ is well-defined, $k$, $m$, and $d$ do not depend on the choice of reduced expression for $w$ that we start with.  We will denote the diagram $d$ from Lemma~\ref{lem:powers of 2 and delta for images of monomials} by $d_{w}$.  That is, if $w \in \FC$, then $d_{w}$ is the admissible diagram satisfying $\theta(b_{w})=2^{k}\delta^{m}d_{w}$.

If $d$ is an admissible diagram, then we say that a non-propagat\-ing edge joining $i$ to $i+1$ (respectively, $i'$ to $(i+1)'$) is \emph{simple} if it is identical to the edge joining $i$ to $i+1$ (respectively, $i'$ to $(i+1)'$) in the simple diagram $d_{i}$.  That is, an edge is simple if it joins adjacent vertices in the north face (respectively, south face) and is undecorated, except when one of the vertices is 1 or $1'$ (respectively, $n+2$ or $(n+2)'$), in which case it is decorated by only a single $\bcirc$ (respectively, $\wcirc$).  

Let $d$ be an admissible diagram.  Since $\theta$ is surjective and $\DTL(\C_n)$ is generated by the simple diagrams, there exists $w \in \FC(\C_{n})$ such that $\theta(b_{w})=d_{w}=d$.  Suppose $\w=s_{i_{1}}\cdots s_{i_{r}}$ is a reduced expression for $w$.  Then $d=d_{i_{1}}\cdots d_{i_{r}}$.   For each $d_{i_{j}}$ fix a concrete representative that has straight propagating edges and no unnecessary ``wiggling'' of the simple non-propagating edges.  Now, consider the concrete diagram that results from stacking the concrete simple diagrams $d_{i_{1}},\dots, d_{i_{r}}$, rescaling to recover the standard $(n+2)$-box, but not deforming any of the edges or applying any relations among the decorations.  We will refer to this concrete diagram as the \emph{concrete simple representation of $d_{\w}$} (which does depend on $\w$).  Since $w$ is FC and vertical equivalence respects commutation, given two different reduced expressions $\w$ and $\w'$ for $w$, the concrete simple representations $d_{\w}$ and $d_{\w'}$ will be vertically equivalent.  We define the vertical equivalence class of concrete simple representations to be the \emph{simple representation of $d_{w}$}.

\begin{example}\label{ex:vert equiv}
If $w=\z_{1,1}^{R,1}$, then the diagram in Figure~\ref{fig:ex vert equiv} is vertically equivalent to the simple representation of $d_{w}$, where the vertical dashed lines in the diagram indicate that the two curves are part of the same generator.
\end{example}

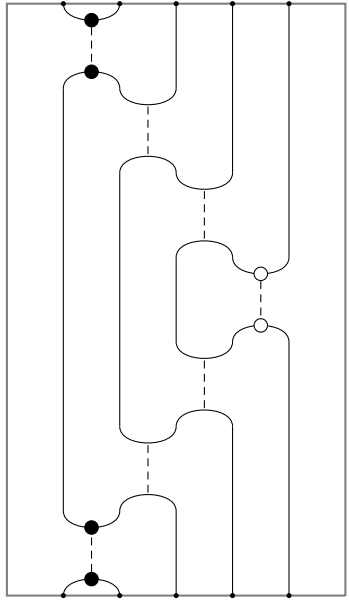
\begin{figure}[!ht]
\begin{tikzpicture}[baseline=-0.5ex,scale=.75]
\draw[gray,thick] (6,-5.25) to (0,-5.25) to (0,5.25) to (6,5.25) to (6,-5.25);
\foreach \x in {1,2,3,4,5} \filldraw (\x,-5.25) circle (1pt);
\foreach \x in {1,2,3,4,5} \filldraw (\x,5.25) circle (1pt);
\draw[out=-90,in=-90] (1,5.25) to node[blackcirc, pos=0.5](A){} (2,5.25);
\draw[out=90,in=90] (1,-5.25) to node[blackcirc, pos=0.5](B){} (2,-5.25);
\draw[out=90,in=90] (1,3.75) to node[blackcirc, pos=0.5](C){} (2,3.75);
\draw[out=-90,in=-90] (1,-3.75) to node[blackcirc, pos=0.5](D){} (2,-3.75);

\draw[out=-90,in=-90] (2,3.75) to node[inner sep=0pt](E){} (3,3.75);
\draw[out=90,in=90] (2,-3.75) to node[inner sep=0pt](F){} (3,-3.75);
\draw[out=90,in=90] (2,2.25) to node[inner sep=0pt](G){} (3,2.25);
\draw[out=-90,in=-90] (2,-2.25) to node[inner sep=0pt](H){} (3,-2.25);

\draw[out=-90,in=-90] (3,2.25) to node[inner sep=0pt](I){} (4,2.25);
\draw[out=90,in=90] (3,-2.25) to node[inner sep=0pt](J){} (4,-2.25);
\draw[out=90,in=90] (3,.75) to node[inner sep=0pt](K){} (4,.75);
\draw[out=-90,in=-90] (3,-.75) to node[inner sep=0pt](L){} (4,-.75);

\draw[out=-90,in=-90] (4,.75) to node[whitecirc, pos=0.5](M){} (5,.75);
\draw[out=90,in=90] (4,-.75) to node[whitecirc, pos=0.5](N){} (5,-.75);

\draw (1,3.75) to (1,-3.75);
\draw (2,2.25) to (2,-2.25);
\draw (3,3.75) to (3,5.25);
\draw (3,-3.75) to (3,-5.25);
\draw (3,-.75) to (3,.75);
\draw (4,2.25) to (4,5.25);
\draw (4,-2.25) to (4,-5.25);
\draw (5,.75) to (5,5.25);
\draw (5,-.75) to (5,-5.25);

\draw[densely dashed] (A) -- (C);
\draw[densely dashed] (B) -- (D);
\draw[densely dashed] (E) -- (G);
\draw[densely dashed] (F) -- (H);
\draw[densely dashed] (I) -- (K);
\draw[densely dashed] (J) -- (L);
\draw[densely dashed] (M) -- (N);
\end{tikzpicture}
\caption{Diagram corresponding to Example~\ref{ex:vert equiv}.}\label{fig:ex vert equiv}
\end{figure}

\begin{lemma}\label{lem:image zigzag}
Let $w \in \FC(\C_{n})$ be of type I.  Then 
\begin{enumerate}[label=\rm{(\arabic*)}]
\item $\theta(b_{w})=d_{w}$ with $\a(d_{w})=n(w)=1$;
\item If $w'$ is also of type I with $w \neq w'$, then $d_{w} \neq d_{w'}$;
\item If $s_{i} \in \L(w)$ (respectively, $\R(w)$), then there is a simple edge joining $i$ to $i+1$ (respectively, $i'$ to $(i+1)'$).
\end{enumerate}
\end{lemma}

\begin{proof}
This lemma follows easily from the definition of $\theta$.
\end{proof}

\begin{lemma}\label{lem:image wsrm}
Let $w\in \FC(\C_{n})$ be a non-cancellable element that is not of type I.  Then 
\begin{enumerate}[label=\rm{(\arabic*)}]
\item $\theta(b_{w})=d_{w}$ with $\a(d_{w})=n(w)>1$;
\item If $w'$ is also a non-cancellable element that is not of type I with $w \neq w'$, then $d_{w} \neq d_{w'}$;
\item If $s_{i} \in \L(w)$ (respectively, $\R(w)$), then there is a simple edge joining $i$ to $i+1$ (respectively, $i'$ to $(i+1)'$).
\end{enumerate}
\end{lemma}

\begin{proof}
The result follows from the definition of $\theta$ and the classification of the type $\C_n$ non-cancellable elements in~\cite[Theorem 5.1.1]{Ernst2010} together with~\cite[Proposition 3.1.3]{Ernst2010}.
\end{proof}

The upshot of the previous two lemmas is that the image of a monomial indexed by any type I element or any non-cancellable element is a single admissible diagram (i.e., there are no powers of 2 or $\delta$). 

\end{subsection}


\begin{subsection}{Additional preparatory lemmas}\label{subsec:additional prep lemmas}

Our immediate goal is to show that $\theta(b_{w})=d_{w}$ for any $w \in \FC(\C_{n})$ (see Proposition~\ref{prop:monomials map to single diagrams}).  To accomplish this task, we require a few additional lemmas.  We will state an ``if and only if'' version of the following lemma later (see Lemma~\ref{lem:diagram descent set}).

\begin{lemma}\label{lem:simple edge in N face}
Let $w \in \FC(\C_{n})$.  If $s_{i} \in \L(w)$ (respectively, $\R(w)$), then there is a simple edge joining node $i$ to node $i+1$ (respectively, node $i'$ to node $(i+1)'$) in the north (respectively, south) face of $d_{w}$.
\end{lemma} 

\begin{proof}
Assume $s_{i} \in \L(w)$.  Then we can write $w=s_{i}v$ (reduced).  By Lemma~\ref{lem:powers of 2 and delta for images of monomials} applied to $\theta(b_{v})$, we must have
\[
\theta(b_{w})=\theta(b_{i})\theta(b_{v})=2^{k}\delta^{m}d_{i}d_{v}.
\]
This implies that we obtain $d_{w}$ by concatenating $d_{i}$ on top of $d_{v}$.  It follows that we must have a simple edge joining $i$ to $i+1$ in the north face of $d_{w}$.  The proof that $s_{i} \in \R(w)$ implies that there is a simple edge joining $i'$ to $(i+1)'$ in the south face of $d_w$ is similar.
\end{proof}

\begin{lemma}\label{lem:weak star preserve a-value}
Let $w \in \FC(\C_{n})$. If $w$ is left weak star reducible by $s$ with respect to $t$ to $v$, then $\a(d_{w})=\a(d_{v})$.
\end{lemma}

\begin{proof}
By Remark~\ref{rem:monomial weak star reductions}, we have $b_{t}b_{w}=2^{c}b_{v}$, where $c \in \{0,1\}$ and $c=1$ if and only if $m(s,t)=4$.  This implies that
\[
\theta(b_{t}b_{w})=\theta(2^{c}b_{v})=2^{c}2^{k}\delta^{m}b_{v},
\]
where $k,m \in \Z^{+}\cup\{0\}$.  But, on the other hand, we have
\[
\theta(b_{t}b_{w})=\theta(b_{t})\theta(b_{w})=2^{k'}\delta^{m'}d_{t}d_{w},
\]
where $k',m' \in \Z^{+}\cup\{0\}$.  It follows that $2^{c}2^{k}\delta^{m}d_{v}=2^{k'}\delta^{m'}d_{t}d_{w}$, and so $d_t d_w$ is a scalar multiple of $d_{v}$. Suppose $s=s_{i}$.  By Lemma~\ref{lem:simple edge in N face}, there must be a simple edge joining $i$ to $i+1$ in north face of $d_{w}$. Without loss of generality, suppose $t=s_{i+1}$. Then the edge configuration at nodes $i$, $i+1$, and $i+2$ of $d_{w}$ is depicted in Figure~\ref{fig:weak star preserve a-value}, where $\mathbf{b}=\bcirc$ if and only if $i=1$ and is trivial otherwise.  The edge leaving node $i+2$ may be decorated and may be propagating or non-propagating.  Multiplying $d_{w}$ on the left by $d_{t}=d_{i+1}$ results in a diagram that has the same $\a$-value as $d_{w}$.  Therefore, $\a(d_{v})=\a(d_{w})$, as desired.
\end{proof}

\begin{figure}[!ht]
\begin{tikzpicture}[baseline=-0.5ex,scale=.8]
\draw[gray,thick] (0,1) to (4,1);
\foreach \x in {1,2,3} \filldraw (\x,1) circle (1pt);
\node[above] at (1,1) {\tiny $\phantom{+}i\phantom{+}$};
\node[above] at (2,1) {\tiny $i+1$};
\node[above] at (3,1) {\tiny $i+2$};
\draw[out=-90,in=90,-stealth] (3,1) to (2,0);
\draw[out=-90,in=-90] (1,1) to node[rectangle, inner sep=2.2pt, draw=black, fill=white, pos=0.5]{\tiny $\mathbf{b}$} (2,1);
\end{tikzpicture}
\caption{Portion of the diagram corresponding to the proof of Lemma~\ref{lem:weak star preserve a-value}.}\label{fig:weak star preserve a-value}
\end{figure}
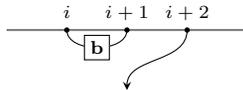

Lemma~\ref{lem:weak star preserve a-value} has an analogous statement involving right weak star reductions.

\begin{lemma}\label{lem:a=1 implies type I}
Let $w \in \FC(\C_{n})$.  Then $\a(d_{w})=1$ if and only if $w$ is of type I.
\end{lemma}

\begin{proof}
First, suppose $\a(d_{w})=1$.  If $w$ is non-cancellable, then by Lemmas~\ref{lem:image zigzag} and~\ref{lem:image wsrm}, $w$ must be of type I.  On the other hand, assume $w$ is not non-cancellable.  Then there exists a sequence of weak star reductions that reduces $w$ to a non-cancellable element.  By Lemma~\ref{lem:weak star preserve a-value}, each diagram corresponding to the elements of this sequence has the same $\a$-value as $d_{w}$.  Since the sequence of weak star reductions terminates at a non-cancellable element and the diagram corresponding to this element has $\a$-value 1, the non-cancellable element must also have $n$-value 1 by Lemmas~\ref{lem:image zigzag} and~\ref{lem:image wsrm}. Since weak star reductions are a special case of ordinary star reductions, it follows from Lemma~2.9 in~\cite{Shi2005b} that any sequence of weak star reductions applied to $w$ will preserve the $n$-value. It follows that $n(w)=1$, and so $w$ is of type I.  Conversely, if $w$ is of type I, then according to Lemma~\ref{lem:image zigzag}, $\a(d_{w})=1$.
\end{proof}

\begin{lemma}\label{lem:diagram version main zigzag lemma}
Let $w \in \FC(\C_{n})$.  If $1<i < n+1$ such that $H(w)$ has two consecutive occurrences of entries labeled by $s_{i}$ and there is no entry labeled by $s_{i+1}$ occurring between them, then one or both of the following must be true about $d_{w}$:
\begin{enumerate}[label=\rm{(\arabic*)}]
\item The western end of the simple representation of $d_{w}$ is vertically equivalent to the diagram in Figure~\ref{fig:diagram version main zigzag lemma}, where the vertical dashed lines in the diagram indicate that the two curves are part of the same generator $d_{j}$ and the free horizontal arrow indicates a continuation of the pattern of the same shape.  Furthermore, there are no other occurrences of the generators $d_{1}, \dots, d_{i}$ in the simple representation of $d_{w}$;
\item $\a(d_{w})=1$.
\end{enumerate}
\end{lemma}

\begin{figure}[!ht]
\begin{tikzpicture}[baseline=-0.5ex,scale=.75]
\draw[gray,thick] (9,-8.25) to (0,-8.25) to (0,8.25) to (9,8.25);
\foreach \x in {1,2,5,6,7,8} \filldraw (\x,8.25) circle (1pt);
\foreach \x in {1,2,5,6,7,8} \filldraw (\x,-8.25) circle (1pt);
\draw[out=-90,in=-90] (1,.75) to node[blackcirc, pos=0.5](A){} (2,.75);
\draw[out=90,in=90] (1,-.75) to node[blackcirc, pos=0.5](B){} (2,-.75);

\draw[out=90,in=90] (2,.75) to node[inner sep=0pt](C){} (3,.75);
\draw[out=-90,in=-90] (2,2.25) to node[inner sep=0pt](D){} (3,2.25);

\draw[out=-90,in=-90] (2,-.75) to node[inner sep=0pt](E){} (3,-.75);
\draw[out=90,in=90] (2,-2.25) to node[inner sep=0pt](F){} (3,-2.25);

\draw[out=90,in=90] (3,2.25) to node[inner sep=0pt](G){} (4,2.25);
\draw[out=-90,in=-90] (3,-2.25) to node[inner sep=0pt](H){} (4,-2.25);

\draw[out=90,in=90] (5,-3.75) to node[inner sep=0pt](I){} (6,-3.75);
\draw[out=-90,in=-90] (5,3.75) to node[inner sep=0pt](J){} (6,3.75);

\draw[out=90,in=90] (6,3.75) to node[inner sep=0pt](L){} (7,3.75);
\draw[out=-90,in=-90] (6,5.25) to node[inner sep=0pt](K){} (7,5.25);

\draw[out=90,in=90] (7,5.25) to node[inner sep=0pt](M){} (8,5.25);
\draw[out=-90,in=-90,-stealth] (7,6.75) to node[inner sep=0pt](N){} (8,6.75);

\draw[out=-90,in=-90] (6,-3.75) to node[inner sep=0pt](O){} (7,-3.75);
\draw[out=90,in=90] (6,-5.25) to node[inner sep=0pt](P){} (7,-5.25);

\draw[out=-90,in=-90] (7,-5.25) to node[inner sep=0pt](Q){} (8,-5.25);
\draw[out=90,in=90,-stealth] (7,-6.75) to node[inner sep=0pt](R){} (8,-6.75);

\node at (5,0) {$\longrightarrow$};
\node[above] at (3.5,8.25) {\tiny $\cdots$};

\node[above] at (1,8.25) {\tiny $\phantom{+}1\phantom{+}$};
\node[above] at (2,8.25) {\tiny $\phantom{+}2\phantom{+}$};
\node[above] at (5,8.25) {\tiny $\phantom{+}i-2\phantom{+}$};
\node[above] at (6,8.25) {\tiny $\phantom{+}i-1\phantom{+}$};
\node[above] at (7,8.25) {\tiny $\phantom{+}i\phantom{+}$};
\node[above] at (8,8.25) {\tiny $\phantom{+}i+1\phantom{+}$};

\draw (3,.75) to (3,-.75);
\draw (4,2.25) to (4,-2.25);
\draw (1,.75) to (1,8.25);
\draw (1,-.75) to (1,-8.25);
\draw (2,2.25) to (2,8.25);
\draw (2,-2.25) to (2,-8.25);
\draw (5,3.75) to (5,8.25);
\draw (5,-3.75) to (5,-8.25);
\draw (6,5.25) to (6,8.25);
\draw (6,-5.25) to (6,-8.25);
\draw (7,6.75) to (7,8.25);
\draw (7,-6.75) to (7,-8.25);
\draw (7,-3.75) to (7,3.75);
\draw (8,-5.25) to (8,5.25);

\draw[densely dashed] (A) -- (B);
\draw[densely dashed] (C) -- (D);
\draw[densely dashed] (E) -- (F);
\draw[densely dashed] (K) -- (L);
\draw[densely dashed] (M) -- (N);
\draw[densely dashed] (O) -- (P);
\draw[densely dashed] (Q) -- (R);
\end{tikzpicture}
\caption{Portion of the diagram corresponding to Lemma~\ref{lem:diagram version main zigzag lemma}.}\label{fig:diagram version main zigzag lemma}
\end{figure}
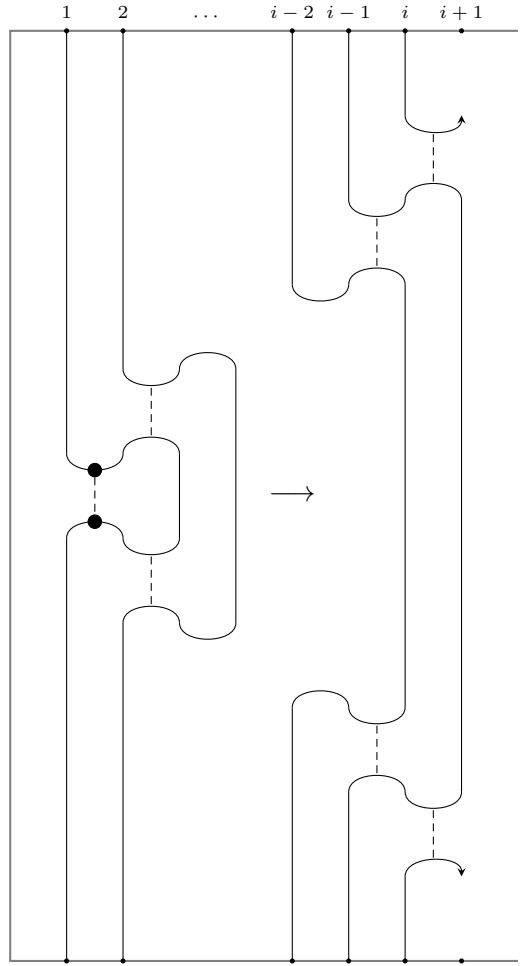

\begin{proof}
This follows immediately from the definition of the map $\theta$ together with Lemma~\ref{lem:main zigzag lemma}.
\end{proof}

\begin{lemma}\label{lem:direction change}
If $w \in \FC(\C_{n})$, then the only way that an edge of the simple representation of $d_{w}$ may change direction from right to left is if a convex subset of the simple representation of $d_{w}$ is vertically equivalent to one of the diagrams in Figure~\ref{fig:direction change}, where the vertical dashed lines in the diagram indicate that the two curves are part of the same generator $d_{j}$ and the arrows indicate a continuation of the pattern of the same shape.
\end{lemma}

\begin{figure}[!ht]
\subcaptionbox{\label{fig:direction change 1}}[.49\linewidth]{
\begin{tikzpicture}[baseline=-0.5ex,scale=.75]
\draw[gray,thick] (0,-6.75) to (0,6.75);
\draw[gray,thick,dashed] (9,-6.75) to (0,-6.75);
\draw[gray,thick,dashed] (0,6.75) to (9,6.75);
\draw[out=-90,in=-90] (1,.75) to node[blackcirc, pos=0.5](A){} (2,.75);
\draw[out=90,in=90] (1,-.75) to node[blackcirc, pos=0.5](B){} (2,-.75);

\draw[out=90,in=90] (2,.75) to node[inner sep=0pt](C){} (3,.75);
\draw[out=-90,in=-90] (2,2.25) to node[inner sep=0pt](D){} (3,2.25);

\draw[out=-90,in=-90] (2,-.75) to node[inner sep=0pt](E){} (3,-.75);
\draw[out=90,in=90] (2,-2.25) to node[inner sep=0pt](F){} (3,-2.25);

\draw[out=90,in=90] (3,2.25) to node[inner sep=0pt](G){} (4,2.25);
\draw[out=-90,in=-90] (3,-2.25) to node[inner sep=0pt](H){} (4,-2.25);

\draw[out=90,in=90] (5,-3.75) to node[inner sep=0pt](I){} (6,-3.75);
\draw[out=-90,in=-90] (5,3.75) to node[inner sep=0pt](J){} (6,3.75);

\draw[out=90,in=90] (6,3.75) to node[inner sep=0pt](L){} (7,3.75);
\draw[out=-90,in=-90] (6,5.25) to node[inner sep=0pt](K){} (7,5.25);

\draw[out=90,in=90] (7,5.25) to node[inner sep=0pt](M){} (8,5.25);
\draw[out=-90,in=-90] (7,6.75) to node[inner sep=0pt](N){} (8,6.75);

\draw[out=-90,in=-90] (6,-3.75) to node[inner sep=0pt](O){} (7,-3.75);
\draw[out=90,in=90] (6,-5.25) to node[inner sep=0pt](P){} (7,-5.25);

\draw[out=-90,in=-90] (7,-5.25) to node[inner sep=0pt](Q){} (8,-5.25);
\draw[out=90,in=90] (7,-6.75) to node[inner sep=0pt](R){} (8,-6.75);

\node at (5,0) {$\longrightarrow$};
\node[above] at (3.5,6.75) {\tiny $\cdots$};

\node[above] at (1,6.75) {\tiny $\phantom{+}1\phantom{+}$};
\node[above] at (2,6.75) {\tiny $\phantom{+}2\phantom{+}$};
\node[above] at (5,6.75) {\tiny $\phantom{+}i-2\phantom{+}$};
\node[above] at (6,6.75) {\tiny $\phantom{+}i-1\phantom{+}$};
\node[above] at (7,6.75) {\tiny $\phantom{+}i\phantom{+}$};
\node[above] at (8,6.75) {\tiny $\phantom{+}i+1\phantom{+}$};

\draw (3,.75) to (3,-.75);
\draw (4,2.25) to (4,-2.25);
\draw (1,.75) to (1,6.75);
\draw (1,-.75) to (1,-6.75);
\draw (2,2.25) to (2,6.75);
\draw (2,-2.25) to (2,-6.75);
\draw (5,3.75) to (5,6.75);
\draw (5,-3.75) to (5,-6.75);
\draw (6,5.25) to (6,6.75);
\draw (6,-5.25) to (6,-6.75);
\draw (7,6.75) to (7,6.75);
\draw (7,-6.75) to (7,-6.75);
\draw (7,-3.75) to (7,3.75);
\draw (8,-5.25) to (8,5.25);

\draw[densely dashed] (A) -- (B);
\draw[densely dashed] (C) -- (D);
\draw[densely dashed] (E) -- (F);
\draw[densely dashed] (K) -- (L);
\draw[densely dashed] (M) -- (N);
\draw[densely dashed] (O) -- (P);
\draw[densely dashed] (Q) -- (R);
\end{tikzpicture}}
\subcaptionbox{\label{fig:direction change 2}}[.49\linewidth]{
\begin{tikzpicture}[baseline=-0.5ex,scale=.75]
\draw[gray,thick] (0,-6.75) to (0,6.75);
\draw[gray,thick] (9,-6.75) to (9,6.75);
\draw[gray,thick,dashed] (9,-6.75) to (0,-6.75);
\draw[gray,thick,dashed] (0,6.75) to (9,6.75);
\draw[out=-90,in=-90] (1,.75) to node[blackcirc, pos=0.5](A){} (2,.75);
\draw[out=90,in=90] (1,-.75) to node[blackcirc, pos=0.5](B){} (2,-.75);

\draw[out=90,in=90] (2,.75) to node[inner sep=0pt](C){} (3,.75);
\draw[out=-90,in=-90] (2,2.25) to node[inner sep=0pt](D){} (3,2.25);

\draw[out=-90,in=-90] (2,-.75) to node[inner sep=0pt](E){} (3,-.75);
\draw[out=90,in=90] (2,-2.25) to node[inner sep=0pt](F){} (3,-2.25);

\draw[out=90,in=90] (3,2.25) to node[inner sep=0pt](G){} (4,2.25);
\draw[out=-90,in=-90] (3,-2.25) to node[inner sep=0pt](H){} (4,-2.25);

\draw[out=90,in=90] (5,-3.75) to node[inner sep=0pt](I){} (6,-3.75);
\draw[out=-90,in=-90] (5,3.75) to node[inner sep=0pt](J){} (6,3.75);

\draw[out=90,in=90] (6,3.75) to node[inner sep=0pt](L){} (7,3.75);
\draw[out=-90,in=-90] (6,5.25) to node[inner sep=0pt](K){} (7,5.25);

\draw[out=90,in=90] (7,5.25) to node[whitecirc, pos=0.5](M){} (8,5.25);
\draw[out=-90,in=-90] (7,6.75) to node[whitecirc, pos=0.5](N){} (8,6.75);

\draw[out=-90,in=-90] (6,-3.75) to node[inner sep=0pt](O){} (7,-3.75);
\draw[out=90,in=90] (6,-5.25) to node[inner sep=0pt](P){} (7,-5.25);

\draw[out=-90,in=-90] (7,-5.25) to node[whitecirc, pos=0.5](Q){} (8,-5.25);
\draw[out=90,in=90] (7,-6.75) to node[whitecirc, pos=0.5](R){} (8,-6.75);

\node at (5,0) {$\longrightarrow$};
\node[above] at (3.5,6.75) {\tiny $\cdots$};

\node[above] at (1,6.75) {\tiny $\phantom{+}1\phantom{+}$};
\node[above] at (2,6.75) {\tiny $\phantom{+}2\phantom{+}$};
\node[above] at (5,6.75) {\tiny $\phantom{+}n-1\phantom{+}$};
\node[above] at (6,6.75) {\tiny $\phantom{+}n\phantom{+}$};
\node[above] at (7,6.75) {\tiny $\phantom{+}n+1\phantom{+}$};
\node[above] at (8,6.75) {\tiny $\phantom{+}n+2\phantom{+}$};

\draw (3,.75) to (3,-.75);
\draw (4,2.25) to (4,-2.25);
\draw (1,.75) to (1,6.75);
\draw (1,-.75) to (1,-6.75);
\draw (2,2.25) to (2,6.75);
\draw (2,-2.25) to (2,-6.75);
\draw (5,3.75) to (5,6.75);
\draw (5,-3.75) to (5,-6.75);
\draw (6,5.25) to (6,6.75);
\draw (6,-5.25) to (6,-6.75);
\draw (7,6.75) to (7,6.75);
\draw (7,-6.75) to (7,-6.75);
\draw (7,-3.75) to (7,3.75);
\draw (8,-5.25) to (8,5.25);

\draw[densely dashed] (A) -- (B);
\draw[densely dashed] (C) -- (D);
\draw[densely dashed] (E) -- (F);
\draw[densely dashed] (K) -- (L);
\draw[densely dashed] (M) -- (N);
\draw[densely dashed] (O) -- (P);
\draw[densely dashed] (Q) -- (R);
\end{tikzpicture}}\\
\vspace{1em}
\subcaptionbox{\label{fig:direction change 3}}{
\begin{tikzpicture}[scale=.75]
\draw[gray,thick] (9,2.75) to (9,6.75) to (5,6.75);
\foreach \x in {6,7,8} \filldraw (\x,6.75) circle (1pt);
\draw[out=90,in=90] (6,3.75) to node[inner sep=0pt](L){} (7,3.75);
\draw[out=-90,in=-90] (6,5.25) to node[inner sep=0pt](K){} (7,5.25);
\draw[out=90,in=90] (7,5.25) to node[whitecirc, pos=0.5](M){} (8,5.25);
\draw[out=-90,in=-90] (7,3.75) to node[whitecirc, pos=0.5](Q){} (8,3.75);
\node[above] at (6,6.75) {\tiny $\phantom{+}n\phantom{+}$};
\node[above] at (7,6.75) {\tiny $\phantom{+}n+1\phantom{+}$};
\node[above] at (8,6.75) {\tiny $\phantom{+}n+2\phantom{+}$};
\draw (8,3.75) to (8,5.25);
\draw[densely dashed] (K) -- (L);
\end{tikzpicture}}
\caption{Portions of the diagrams corresponding to Lemma~\ref{lem:direction change}.}\label{fig:direction change}
\end{figure}
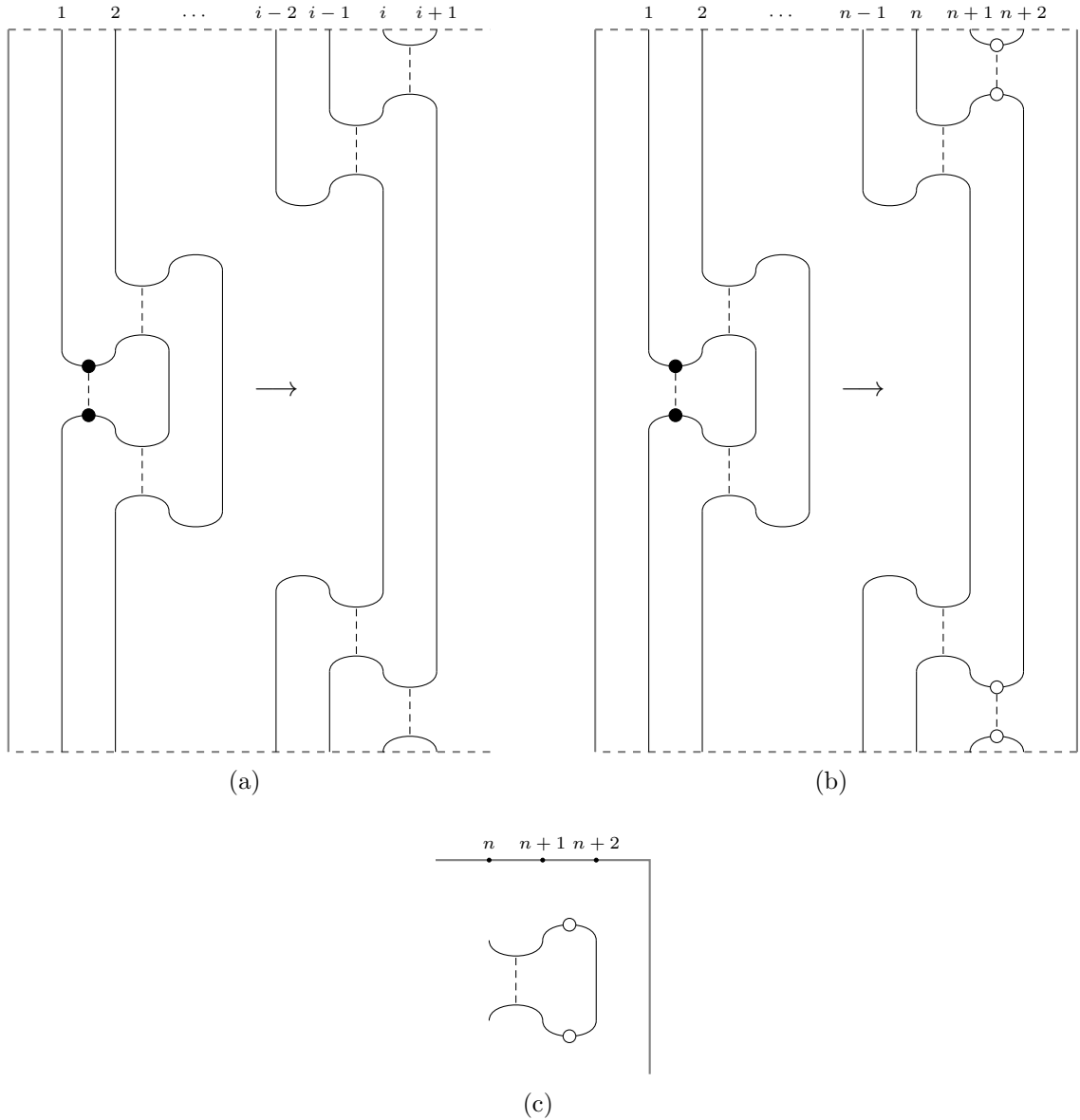

\begin{proof}
An edge changing direction from right to left directly below node $i+1$ indicates that there are two consecutive occurrences of the simple diagram $d_{i}$ not having an occurrence of $d_{i+1}$ between them. Note that this forces $i>1$.  If $i<n+1$, then by Lemma~\ref{lem:diagram version main zigzag lemma}, we must have the diagram in Figure~\ref{fig:direction change 1}.  On the other hand, suppose $i=n+1$.  In this case, there are at most two occurrences of $d_{n}$ occurring between the two consecutive occurrences of $d_{n+1}$ by Lemma~\ref{lem:main zigzag lemma}. If there are two occurrences of $d_{n}$, then applying Lemma~\ref{lem:diagram version main zigzag lemma} to the two consecutive occurrences of $d_{n}$ forces us to have the diagram in Figure~\ref{fig:direction change 2}.  Finally, if there is a single occurrence of $d_{n}$ between the two consecutive occurrences of $d_{n+1}$, then we must have the diagram in Figure~\ref{fig:direction change 3}.
\end{proof}

\begin{remark}\label{rem:direction change}
If the diagram in Figure~\ref{fig:direction change 2} occurs, then $w$ must be of type I by Lemma~\ref{lem:zigzag}.  
\end{remark}

\begin{lemma}\label{lem:diagram descent set}
Let $w \in \FC(\C_{n})$.  Then $s_{i} \in \L(w)$ (respectively, $\R(w)$) if and only if there is a simple edge joining $i$ to $i+1$ (respectively, $i'$ to $(i+1)'$) in $d_{w}$.
\end{lemma}

\begin{proof}
The forward direction is Lemma~\ref{lem:simple edge in N face}. For the converse, suppose there is a simple edge joining node $i$ to node $i+1$ in the north face of $d_{w}$.  We need to show that $s_{i} \in \L(w)$.  By Lemma~\ref{lem:a=1 implies type I}, if $\a(d)=1$, then $w$ is of type I.  In this case, $s_{i} \in \L(w)$ by Lemma~\ref{lem:image zigzag}.  Now, suppose $\a(d)>1$.  Consider the simple representation for $d_{w}$ and let $e$ be the edge joining $i$ to $i+1$.  For sake of a contradiction, suppose  that $s_{i} \notin \L(w)$.  Then either 
\begin{enumerate}[label=\rm{(\alph*)}]
\item It is not the case that the end of $e$ leaving node $i+1$ encounters the northernmost occurrence of $d_{i}$ before any other generator; or
\item It is not the case that the end of $e$ leaving node $i$ encounters the northernmost occurrence of $d_{i}$ before any other generator.
\end{enumerate}
Note that since $e$ crosses the midline between nodes $i$ and $i+1$, it must encounter $d_{i}$ at some stage.  We consider three distinct cases.

Case~(1).  Assume $i \notin \{1, 2, n, n+1\}$.  Then $e$ is undecorated.  We deal with Condition~(a) from above; Condition~(b) has a similar argument.  Since the curve must eventually encounter $d_{i}$, the edge $e$ must change direction from right to left.  Then we must be in one of the three situations of Lemma~\ref{lem:direction change}.  But since $e$ is undecorated and we are assuming that $\a(d_{w})>1$, by Remark~\ref{rem:direction change}, Figure~\ref{fig:diagram descent set 2} depicts the only possibility for the edge leaving node $i+1$ in the simple representation for $d_{w}$, where $\mathbf{b}=\wcirc$ if $i=n-1$ and is trivial otherwise. Then it must be the case that $s_{i+2}$ does not occur between two consecutive occurrences of $s_{i+1}$ in reduced expressions for $w$.  By Lemma~\ref{lem:diagram version main zigzag lemma}, $d_{i}$ cannot occur again in $d_{w}$, which contradicts the edge $e$ joining node $i$ to node $i+1$.

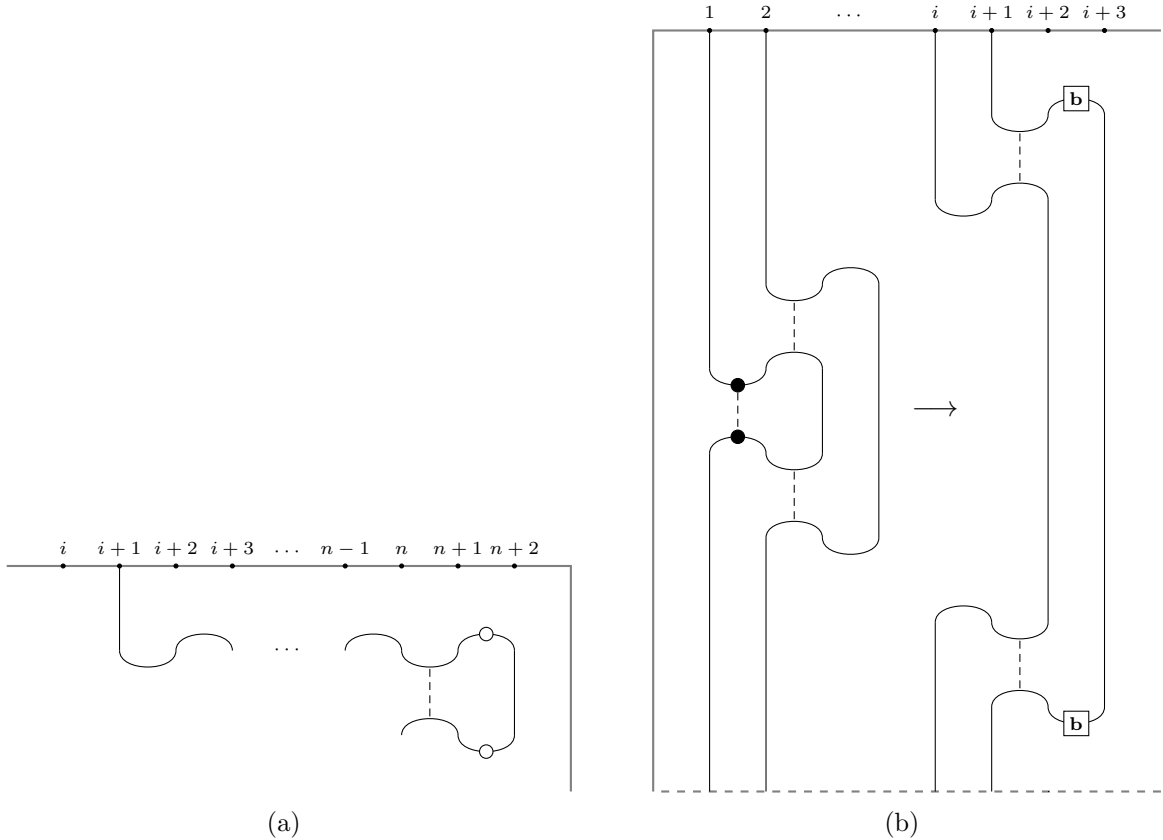
\begin{figure}[!ht]
\centering
\subcaptionbox{\label{fig:diagram descent set 1}}[.49\linewidth]{
\begin{tikzpicture}[scale=.75]
\draw[gray,thick] (9,2.75) to (9,6.75) to (-1,6.75);
\foreach \x in {0,1,2,3,5,6,7,8} \filldraw (\x,6.75) circle (1pt);
\draw[out=90,in=90] (6,3.75) to node[inner sep=0pt](A){} (7,3.75);
\draw[out=-90,in=-90] (6,5.25) to node[inner sep=0pt](B){} (7,5.25);
\draw[out=-90,in=-90] (1,5.25) to node[inner sep=0pt](G){} (2,5.25);
\draw[out=90,in=90] (5,5.25) to node[inner sep=0pt](E){} (6,5.25);
\draw[out=90,in=90] (2,5.25) to node[inner sep=0pt](F){} (3,5.25);

\draw[out=90,in=90] (7,5.25) to node[whitecirc, pos=0.5](C){} (8,5.25);
\draw[out=-90,in=-90] (7,3.75) to node[whitecirc, pos=0.5](D){} (8,3.75);
\node[above] at (4,6.75) {\tiny $\cdots$};
\node[above] at (4,5) {\tiny $\cdots$};
\node[above] at (0,6.75) {\tiny $\phantom{+}i\phantom{+}$};
\node[above] at (1,6.75) {\tiny $\phantom{+}i+1\phantom{+}$};
\node[above] at (2,6.75) {\tiny $\phantom{+}i+2\phantom{+}$};
\node[above] at (3,6.75) {\tiny $\phantom{+}i+3\phantom{+}$};
\node[above] at (5,6.75) {\tiny $\phantom{+}n-1\phantom{+}$};
\node[above] at (6,6.75) {\tiny $\phantom{+}n\phantom{+}$};
\node[above] at (7,6.75) {\tiny $\phantom{+}n+1\phantom{+}$};
\node[above] at (8,6.75) {\tiny $\phantom{+}n+2\phantom{+}$};
\draw (8,3.75) to (8,5.25);
\draw (1,6.75) to (1,5.25);
\draw[densely dashed] (B) -- (A);
\end{tikzpicture}
}
\subcaptionbox{\label{fig:diagram descent set 2}}[.49\linewidth]{
\begin{tikzpicture}[baseline=-0.5ex,scale=.75]
\draw[gray,thick] (0,-6.75) to (0,6.75) to (9,6.75);
\draw[gray,thick,dashed] (9,-6.75) to (0,-6.75);
\foreach \x in {1,2,5,6,7,8} \filldraw (\x,6.75) circle (1pt);
\draw[out=-90,in=-90] (1,.75) to node[blackcirc, pos=0.5](A){} (2,.75);
\draw[out=90,in=90] (1,-.75) to node[blackcirc, pos=0.5](B){} (2,-.75);

\draw[out=90,in=90] (2,.75) to node[inner sep=0pt](C){} (3,.75);
\draw[out=-90,in=-90] (2,2.25) to node[inner sep=0pt](D){} (3,2.25);

\draw[out=-90,in=-90] (2,-.75) to node[inner sep=0pt](E){} (3,-.75);
\draw[out=90,in=90] (2,-2.25) to node[inner sep=0pt](F){} (3,-2.25);

\draw[out=90,in=90] (3,2.25) to node[inner sep=0pt](G){} (4,2.25);
\draw[out=-90,in=-90] (3,-2.25) to node[inner sep=0pt](H){} (4,-2.25);

\draw[out=90,in=90] (5,-3.75) to node[inner sep=0pt](I){} (6,-3.75);
\draw[out=-90,in=-90] (5,3.75) to node[inner sep=0pt](J){} (6,3.75);

\draw[out=90,in=90] (6,3.75) to node[inner sep=0pt](L){} (7,3.75);
\draw[out=-90,in=-90] (6,5.25) to node[inner sep=0pt](K){} (7,5.25);

\draw[out=90,in=90] (7,5.25) to node[rectangle, inner sep=2.2pt, draw=black, fill=white, pos=0.5](M){\tiny $\mathbf{b}$} (8,5.25);

\draw[out=-90,in=-90] (6,-3.75) to node[inner sep=0pt](O){} (7,-3.75);
\draw[out=90,in=90] (6,-5.25) to node[inner sep=0pt](P){} (7,-5.25);

\draw[out=-90,in=-90] (7,-5.25) to node[rectangle, inner sep=2.2pt, draw=black, fill=white, pos=0.5](Q){\tiny $\mathbf{b}$} (8,-5.25);

\node at (5,0) {$\longrightarrow$};
\node[above] at (3.5,6.75) {\tiny $\cdots$};

\node[above] at (1,6.75) {\tiny $\phantom{+}1\phantom{+}$};
\node[above] at (2,6.75) {\tiny $\phantom{+}2\phantom{+}$};
\node[above] at (5,6.75) {\tiny $\phantom{+}i\phantom{+}$};
\node[above] at (6,6.75) {\tiny $\phantom{+}i+1\phantom{+}$};
\node[above] at (7,6.75) {\tiny $\phantom{+}i+2\phantom{+}$};
\node[above] at (8,6.75) {\tiny $\phantom{+}i+3\phantom{+}$};

\draw (3,.75) to (3,-.75);
\draw (4,2.25) to (4,-2.25);
\draw (1,.75) to (1,6.75);
\draw (1,-.75) to (1,-6.75);
\draw (2,2.25) to (2,6.75);
\draw (2,-2.25) to (2,-6.75);
\draw (5,3.75) to (5,6.75);
\draw (5,-3.75) to (5,-6.75);
\draw (6,5.25) to (6,6.75);
\draw (6,-5.25) to (6,-6.75);
\draw (7,6.75) to (7,6.75);
\draw (7,-6.75) to (7,-6.75);
\draw (7,-3.75) to (7,3.75);
\draw (8,-5.25) to (8,5.25);

\draw[densely dashed] (A) -- (B);
\draw[densely dashed] (C) -- (D);
\draw[densely dashed] (E) -- (F);
\draw[densely dashed] (K) -- (L);
\draw[densely dashed] (O) -- (P);
\end{tikzpicture}
}
\caption{Portions of the diagrams corresponding to the proof of Lemma~\ref{lem:diagram descent set}.}\label{fig:diagram descent set}
\end{figure}

Case~(2). Next, assume $i \in \{2, n\}$.  As in Case~(1), $e$ must be undecorated.  Without loss of generality, assume $i=2$.  If Condition~(a) occurs, then we can apply the argument in Case~(1).  If Condition~(b) occurs, then the end of $e$ leaving node $i=2$ must immediately encounter $d_{1}$.  This contradicts $e$ being undecorated since there is no sequence of relations that can completely remove decorations from a non-loop edge.

Case~(3). Lastly, suppose $i \in \{1, n+1\}$. Without loss of generality, assume $i=1$. In this case, $e$ must be decorated by a single $\bcirc$.  This forces us to be in the situation described in Condition~(a) above.  Since $e$ must eventually encounter $d_{1}$, the edge must change direction from right to left. The only two possibilities for the edge leaving node $2$ in the simple representation for $d_{w}$ are depicted in Figure~\ref{fig:diagram descent set}, where $i+1=2$ and $\mathbf{b}$ is trivial. Either way, we arrive at a contradiction similar to Case~(1).

The proof that $s_{i} \in \R(w)$ if and only if there is a simple edge joining $i'$ to $(i+1)'$ is similar to the above.
\end{proof}

\end{subsection}


\begin{subsection}{Proof of injectivity}\label{subsec:injectivity}

The lemmas of the previous section allow us to prove that each monomial basis element maps to a single admissible diagram, as the next proposition illustrates.

\begin{proposition}\label{prop:monomials map to single diagrams}
If $w \in \FC(\C_{n})$, then $\theta(b_{w})=d_{w}$.
\end{proposition}

\begin{proof}
If $w \in \FC(\C_{n})$, then there exists a sequence (possibly trivial) of left and right weak star reductions that reduce $w$ to a non-cancellable element.  We induct on the number of steps in this sequence.  The base case is handled by Lemmas~\ref{lem:image zigzag} and~\ref{lem:image wsrm}. For the inductive step, it is sufficient to assume that $w$ is not of type I by Lemma~\ref{lem:image zigzag}. Without loss of generality, suppose $w$ is left weak star reducible by $s$ with respect to $t$.  We can write: (1) $w=stv$ (reduced) if $m(s,t)=3$ or (2) $w=stsv$ (reduced) if $m(s,t)=4$.  In either case, choose $s$ and $t$ such that $sw$ requires the least number of steps to reduce to a non-cancellable element. We now consider possibilities (1) and (2) separately.

Case~(1).  Assume $m(s,t)=3$.  Without loss of generality, suppose  $s=s_{i}$ and $t=s_{i+1}$ with $1<i<n$.  This implies that
{\allowdisplaybreaks
\begin{align*}
\theta(b_{w})&=\theta(b_{i}b_{s_{i+1}v})\\
&=\theta(b_{i})\theta(b_{s_{i+1}v})\\
&=d_{i} d_{s_{i+1}v}\\
&= \begin{tikzpicture}[scale=.75,baseline=-5.5ex]
\draw[gray,thick] (0,-1) rectangle (10,1);
\draw[gray,thick] (0,-3) rectangle (10,-1);
\foreach \x in {1,3,4,5,6,7,9} \filldraw (\x,-1) circle (1pt);
\foreach \x in {1,3,4,5,6,7,9} \filldraw (\x,1) circle (1pt);
\draw (1,-1) to (1,1);
\draw (3,-1) to (3,1);
\draw (6,-1) to (6,1);
\draw (7,-1) to (7,1);
\draw (9,-1) to (9,1);
\node at (2,0) {\tiny $\cdots$};
\node at (8,0) {\tiny $\cdots$};
\node at (5,-2) {$d_{s_{i+1}v}$};
\node[above] at (4,1) {\tiny $\phantom{+}i\phantom{+}$};
\node[above] at (5,1) {\tiny $\phantom{+}i+1\phantom{+}$};
\draw[out=-90,in=-90] (4,1) to (5,1);
\draw[out=90,in=90] (4,-1) to (5,-1);
\draw[out=-90,in=-90] (5,-1) to (6,-1);
\end{tikzpicture}\ ,
\end{align*}}%
where we are applying the induction hypothesis to $\theta(b_{s_{i+1}v})$ and we are using Lemma~\ref{lem:diagram descent set} to draw the bottom diagram in the last line.  By inspecting the product $d_{i}d_{s_{i+1}v}$, we see that there are no loops and no new relations to apply since $d_{s_{i+1}v}$ is admissible.  Therefore, $\theta(b_{w})=d_{w}$.

Case~(2).  Assume $m(s,t)=4$.  Without loss of generality, suppose $\{s,t\}=\{s_{1},s_{2}\}$.  Since $w=stsv$ (reduced) and $w$ is FC, neither $s$ nor $t$ are in $\L(v)$. Also, since $w$ is left weak star reducible by $s$ with respect to $t$ to $tsv$, we have $\theta(b_{tsv})=d_{tsv}$ by induction.  By Lemma~\ref{lem:powers of 2 and delta for images of monomials}, there exists $k', m' \in \Z^{+}\cup\{0\}$ such that $\theta(b_{v})=2^{k'}\delta^{m'}d_{v}$.  But then
{\allowdisplaybreaks
\begin{align*}
d_{tsv}&=\theta(b_{tsv})\\
&=\theta(b_{t}b_{s}b_{v})\\
&=\theta(b_{t})\theta(b_{s})\theta(b_{v})\\
&=d_{t}d_{s}2^{k'}\delta^{m'}d_{v}\\
&=2^{k'}\delta^{m'}d_{t}d_{s}d_{v}.
\end{align*}}%
Since $d_{tsv}$ is an admissible diagram, we must have $k'=0$ and $m'=0$, and so $\theta(b_{v})=d_{v}$. A similar argument shows that $\theta(b_{sv})=d_{sv}$.  Next, we consider two possible subcases: (a) $s=s_{1}$, $t=s_{2}$ and (b) $s=s_{2}$, $t=s_{1}$.  

(a)  Assume $s=s_{1}$ and $t=s_{2}$.  By Lemma~\ref{lem:powers of 2 and delta for images of monomials}, we see that
{\allowdisplaybreaks
\begin{align*}
2^{k}\delta^{m}d_{w}&=\theta(b_{w})\\
&=\theta(b_{1}b_{2}b_{1}b_{v})\\
&=\theta(b_{1})\theta(b_{2})\theta(b_{1})\theta(b_{v})\\
&=d_{1}d_{2}d_{1}d_{v}\\
&= \begin{tikzpicture}[scale=.75,baseline=-5.5ex]
\draw[gray,thick] (0,-1) rectangle (10,1);
\draw[gray,thick] (0,-3) rectangle (10,-1);
\foreach \x in {1,2,3,4,9} \filldraw (\x,-1) circle (1pt);
\foreach \x in {1,2,3,4,9} \filldraw (\x,1) circle (1pt);
\draw (3,-1) to node[blacktri, pos=0.5]{} (3,1);
\draw (4,-1) to (4,1);
\draw (9,-1) to (9,1);
\node at (6.5,0) {\tiny $\cdots$};
\node at (5,-2) {$d_{v}$};
\draw[out=-90,in=-90] (1,1) to node[blackcirc, pos=0.5]{} (2,1);
\draw[out=90,in=90] (1,-1) to node[blackcirc, pos=0.5]{} (2,-1);
\end{tikzpicture}
\end{align*}}%
for some $k, m \in \Z^{+}\cup\{0\}$. Since $s_{1} \notin \L(v)$, there cannot be a simple edge joining node 1 to node 2 in the north face of $d_{v}$ by Lemma~\ref{lem:diagram descent set}.  This implies that $m=0$.  Moreover, the edge leaving node 3 of $d_{v}$ is not exposed to the west, and so it cannot be decorated with a closed symbol.  Since $s_{2} \notin \L(v)$, there cannot be a simple edge joining node 2 to node 3 in $d_{v}$.  In order for $d_{1}d_{2}d_{1}d_{v}$ to be equal to a non-zero power of 2 times $d_{w}$, the edge leaving node 1 in $d_{v}$ must be decorated with a closed decoration.  Since we do not have a simple edge between nodes 2 and 3 in $d_v$, the first  decoration on the edge leaving node 1 in $d_{v}$ cannot be a $\bcirc$. If the first decoration occurring on the edge leaving node 1 in $d_{v}$ is a $\btri$, then $d_{1}d_{v}=2 d'$ for some admissible diagram $d'$.  But this contradicts $d_{1}d_{v}=\theta(b_{1})\theta(b_{v})=\theta(b_{sv})=d_{sv}$. So, $k=0$ as expected.

(b) Now, suppose $s=s_{2}$ and $t=s_{1}$.  In this case, we see that
{\allowdisplaybreaks
\begin{align*}
2^{k}\delta^{m}d_{w}&=\theta(b_{w})\\
&=\theta(b_{2}b_{1}b_{2}b_{v})\\
&=\theta(b_{2})\theta(b_{1})\theta(b_{2})\theta(b_{v})\\
&=d_{2}d_{1}d_{2}d_{v}\\
&= 
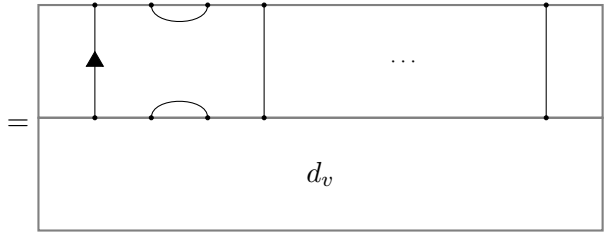
\begin{tikzpicture}[scale=.75,baseline=-5.5ex]
\draw[gray,thick] (0,-1) rectangle (10,1);
\draw[gray,thick] (0,-3) rectangle (10,-1);
\foreach \x in {1,2,3,4,9} \filldraw (\x,-1) circle (1pt);
\foreach \x in {1,2,3,4,9} \filldraw (\x,1) circle (1pt);
\draw (1,-1) to node[blacktri, pos=0.5]{} (1,1);
\draw (4,-1) to (4,1);
\draw (9,-1) to (9,1);
\node at (6.5,0) {\tiny $\cdots$};
\node at (5,-2) {$d_{v}$};
\draw[out=-90,in=-90] (2,1) to (3,1);
\draw[out=90,in=90] (2,-1) to (3,-1);
\end{tikzpicture}\ . 
\end{align*}}%
Since $s_{2} \notin \L(v)$, there cannot be a simple edge joining node 2 to node 3 in the north face of $d_{v}$ by Lemma~\ref{lem:diagram descent set}.  This implies that there can be no loops in the product in the last line above, and so $m=0$.  For sake of a contradiction, suppose $k=0$. Then the edge leaving node 1 in $d_{v}$ must be decorated with a closed decoration. This implies that if $d_{v}$ is written as a product of simple diagrams, there must be at least one occurrence of $d_{1}$. Then we must have $s_{1} \in \supp(v)$, which implies that $tsv=s_{1}s_{2}v$ contains at least two occurrences of $s_{1}$. Consider the top two occurrences of $s_{1}$ in the canonical representation of $H(tsv)=H(s_{1}s_{2}v)$.  Since $w=s_{2}s_{1}s_{2}v$ (reduced) and $w$ is FC, there must be an entry in $H(v)$ labeled by $s_{2}$ that covers the highest occurrence of $s_{1}$.  By a right-handed version of Lemma~\ref{lem:main zigzag lemma}, $\z_{1,1}^{R,1}$ must be a subword of some reduced expression for $w$. But then $w$ must be of type I by Lemma~\ref{lem:zigzag}.  This contradicts our earlier assumption that $w$ is not of type I.  Therefore, the edge leaving node 1 in $d_{v}$ does not carry any closed decorations.  Hence $k=0$.
\end{proof}

The next lemma will be useful for simplifying the argument in the proof of Theorem~\ref{thm:main result}.

\begin{lemma}\label{lem:nugget}
Let $w, w' \in \FC(\C_{n})$. If $d_{w}=d_{w'}$ and $w'$ is left weak star reducible by $s$ with respect to $t$, then
\[
b_{t}b_{w}=\begin{cases}
b_{w''}, & \text{if } m(s,t)=3\\
2b_{w''}, & \text{if } m(s,t)=4,
\end{cases}
\]
for some $w'' \in \FC(\C_{n})$.
\end{lemma}

\begin{proof}
Since $w'$ is left weak star reducible by $s$ with respect to $t$, we can write
\[
w'=\begin{cases}
stv', & \text{if } m(s,t)=3\\
stsv', & \text{if } m(s,t)=4,
\end{cases}
\]
where each product is reduced. Proposition~\ref{prop:monomials map to single diagrams} together Remark~\ref{rem:monomial weak star reductions} implies that
{\allowdisplaybreaks
\begin{align*}
\theta(b_{t}b_{w})&=d_{t}d_{w}\\
&=d_{t}d_{w'}\\
&=\theta(b_{t}b_{w'})\\
&=\begin{cases}
\theta(b_{tv'}), & \text{if } m(s,t)=3\\
2\theta(b_{tsv'}), & \text{if } m(s,t)=4
\end{cases}\\
&=\begin{cases}
d_{tv'}, & \text{if } m(s,t)=3\\
2d_{tsv'}, & \text{if } m(s,t)=4.
\end{cases}
\end{align*}}%
Therefore, it must be the case that
\[
b_{t}b_{w}=\begin{cases}
b_{w''}, & \text{if } m(s,t)=3\\
2b_{w''}, & \text{if } m(s,t)=4
\end{cases}
\]
for some $w'' \in \FC(\C_{n})$.
\end{proof}

Note that Lemma~\ref{lem:nugget} has an analogous statement involving right weak star reductions.

\begin{theorem}\label{thm:main result}
The map $\theta$ of Proposition~\ref{prop:surjective homomorphism} is an isomorphism of $\TL(\C_{n})$ and $\DTL(\C_n)$.  Moreover, each admissible diagram corresponds to a unique monomial basis element.
\end{theorem}

\begin{proof}
According to Proposition~\ref{prop:surjective homomorphism}, $\theta$ is a surjective homomorphism.  Also, by Proposition~\ref{prop:monomials map to single diagrams}, the image of each monomial basis element is a single admissible diagram.  It remains to show that $\theta$ is injective.  For sake of a contradiction, suppose $\theta(b_{w})=\theta(b_{w'})$ for some $w, w' \in \FC(\C_{n})$ with $w \neq w'$. By Lemma~\ref{lem:diagram descent set}, $\L(w)=\L(w')$ and $\R(w)=\R(w')$ since $d_{w}=d_{w'}$.

By Lemmas~\ref{lem:image zigzag} and \ref{lem:a=1 implies type I}, if either of $w$ or $w'$ is of type I, then it must be the case that both are of type I. But then by Lemma~\ref{lem:image zigzag}, we must have $w=w'$, and so neither of $w$ nor $w'$ is of type I.  

Now, suppose neither of $w$ nor $w'$ is non-cancellable.  Then there exists a sequence of left and right weak star reductions that reduces $w'$ to a non-cancellable element, say $w''$.  By Remark~\ref{rem:monomial weak star reductions}, there exists $t'_1, \ldots, t'_l$ and $t''_1, \ldots, t''_{r}$ from $S$ such that
\begin{equation}\label{eqn1}
b_{t'_{l}}\cdots b_{t'_{1}}b_{w'}b_{t''_{1}}\cdots b_{t''_{r}}=2^{k}b_{w''},
\end{equation}
where $k\geq 0$.  By Lemma~\ref{lem:nugget} and its analogous right-handed version, it follows that
\begin{equation}\label{eqn2}
b_{t'_{l}}\cdots b_{t'_{1}}b_{w}b_{t''_{1}}\cdots b_{t''_{r}}=2^{k}b_{w'''},
\end{equation}
for some $w''' \in \FC(\C_{n})$. By applying $\theta$ to Equations~\ref{eqn1} and \ref{eqn2}, we see that $\theta(b_{w''})=\theta(b_{w'''})$ since $\theta(b_{w'})=\theta(b_{w})$. Lemma~\ref{lem:weak star reverse} implies that there exists $k' \geq 0$ such that
\[
b_{s'_{1}}\cdots b_{s'_{l}}b_{t'_{l}}\cdots b_{t'_{1}}b_{w'}b_{t''_{1}}\cdots b_{t''_{r}}b_{s''_{r}}\cdots b_{s''_{1}}=2^{k'}b_{w'},
\]
which implies that
\[
b_{s'_{1}}\cdots b_{s'_{l}}b_{t'_{l}}\cdots b_{t'_{1}}b_{w}b_{t''_{1}}\cdots b_{t''_{r}}b_{s''_{l}}\cdots b_{s''_{1}}=2^{k'}b_{w}.
\]
That is, we can reverse the sequences that reduced $b_{w'}$ (respectively, $b_{w}$) to a multiple of $b_{w''}$ (respectively, $b_{w'''}$).  Hence we may simplify our argument and assume that at least one of $w$ or $w'$ is non-cancellable.  

Without loss of generality, suppose $w$ is non-cancellable.  If $w'$ is also non-cancellable, then we must have $w=w'$ since monomials indexed by distinct non-cancellable elements map to distinct diagrams by Lemmas~\ref{lem:image zigzag} and~\ref{lem:image wsrm}.  So, $w'$ is not non-cancellable. 

Again, without loss of generality, suppose $w'$ is left weak star reducible by $s$ with respect to $t$, so that we may write 
\[
w'=\begin{cases}
stv', & \text{if } m(s,t)=3\\
stsv', & \text{if } m(s,t)=4,
\end{cases}
\]
where each product is reduced.  By Remark~\ref{rem:monomial weak star reductions}, we see that
\[
b_{t}b_{w'}=\begin{cases}
b_{tv'}, & \text{if } m(s,t)=3\\
2b_{tsv'}, & \text{if } m(s,t)=4.
\end{cases}
\]
Note that since $\L(w)=\L(w')$ and $s \in \L(w')$, we have $s \in \L(w)$.  Since $w$ is non-cancellable, $tw$ is reduced and FC.  Then $b_{t}b_{w}=b_{tw}$, which shows that $m(s,t) \neq 4$ by Lemma~\ref{lem:nugget}.  So we must have $m(s,t)=3$.  

Without loss of generality, suppose $s=s_{i+1}$ and $t=s_{i}$ with $2<i<n+1$, so that $w'=s_{i+1}s_{i}v'$ (reduced).  By Lemma~\ref{lem:diagram descent set}, $d_{w'}$, and hence $d_{w}$, has a simple edge joining node $i+1$ to node $i+2$. Moreover,   $s_{i-1} \notin \L(w')=\L(w)$ since $w'=s_{i+1}s_{i}v'$ and $w'$ is FC. Thus, $w$ cannot be of type II since $s_{i+1} \in \L(w)$ while $s_{i-1} \notin \L(w)$.

Since $w$ is non-cancellable but neither of type I nor type II, it follows from the classification in~\cite[Theorem 5.1.1]{Ernst2010} that $w$ can be written as a reduced product of a type $B$ non-cancellable element times a type $B'$ non-cancellable element with non-overlapping supports.  This implies that $w$ contains a single occurrence of $s_{i+1}$ and no occurrences of $s_{i}$.  Then $d_{w}$, and hence $d_{w'}$, can be drawn so that no edges intersect the vertical midline between nodes $i$ and $i+1$.  Furthermore, there are no closed (respectively, open) decorations occurring to the right (respectively, left) of this line.  However, we see that
\[
d_{w'} = d_{i+1}d_{i}d_{v'} = \begin{tikzpicture}[scale=.75,baseline=-5.5ex]
\draw[gray,thick] (0,-1) rectangle (10,1);
\draw[gray,thick] (0,-3) rectangle (10,-1);
\foreach \x in {1,3,4,5,6,7,9} \filldraw (\x,-1) circle (1pt);
\foreach \x in {1,3,4,5,6,7,9} \filldraw (\x,1) circle (1pt);
\draw (1,-1) to (1,1);
\draw (3,-1) to (3,1);
\draw (7,-1) to (7,1);
\draw (9,-1) to (9,1);
\draw[out=-90,in=90] (4,1) to (6,-1);
\node at (2,0) {\tiny $\cdots$};
\node at (8,0) {\tiny $\cdots$};
\node at (5,-2) {$d_{v'}$};
\node[above] at (4,1) {\tiny $\phantom{+}i\phantom{+}$};
\node[above] at (5,1) {\tiny $\phantom{+}i+1\phantom{+}$};
\node[above] at (6,1) {\tiny $\phantom{+}i+2\phantom{+}$};
\draw[out=-90,in=-90] (5,1) to (6,1);
\draw[out=90,in=90] (4,-1) to (5,-1);
\end{tikzpicture}\ . 
\]
This implies that the edge leaving node $i$ in the simple representation of $d_{w'}$ must change direction from right to left.  By Lemma~\ref{lem:direction change} and Remark~\ref{rem:direction change}, the  simple representation of $d_{w'}$ must be vertically equivalent to one of the diagrams in Figure~\ref{fig:main result}.  But we cannot have the diagram in Figure~\ref{fig:main result 1} since then we would have open decorations occurring to the left of the vertical midline between nodes $i$ and $i+1$.  So we must have the diagram in Figure~\ref{fig:main result 2}, which puts us in the situation of Lemma~\ref{lem:diagram version main zigzag lemma}.  Since $w$ is not of type I, there are no other occurrences of the generators $d_{1}, \dots, d_{i}$ in the simple representation of $d_{w'}$.  This implies that $d_{w}$ has a propagating edge connecting node $1$ to node $1'$ that is labeled by a single $\btri$.  By inspecting the images of monomials indexed by non-type I and non-type II non-cancellable elements~\cite[Theorem 5.1.1]{Ernst2010}, we see that none of them have this configuration.  Therefore, we have a contradiction, and hence $\theta$ is injective.
\end{proof}

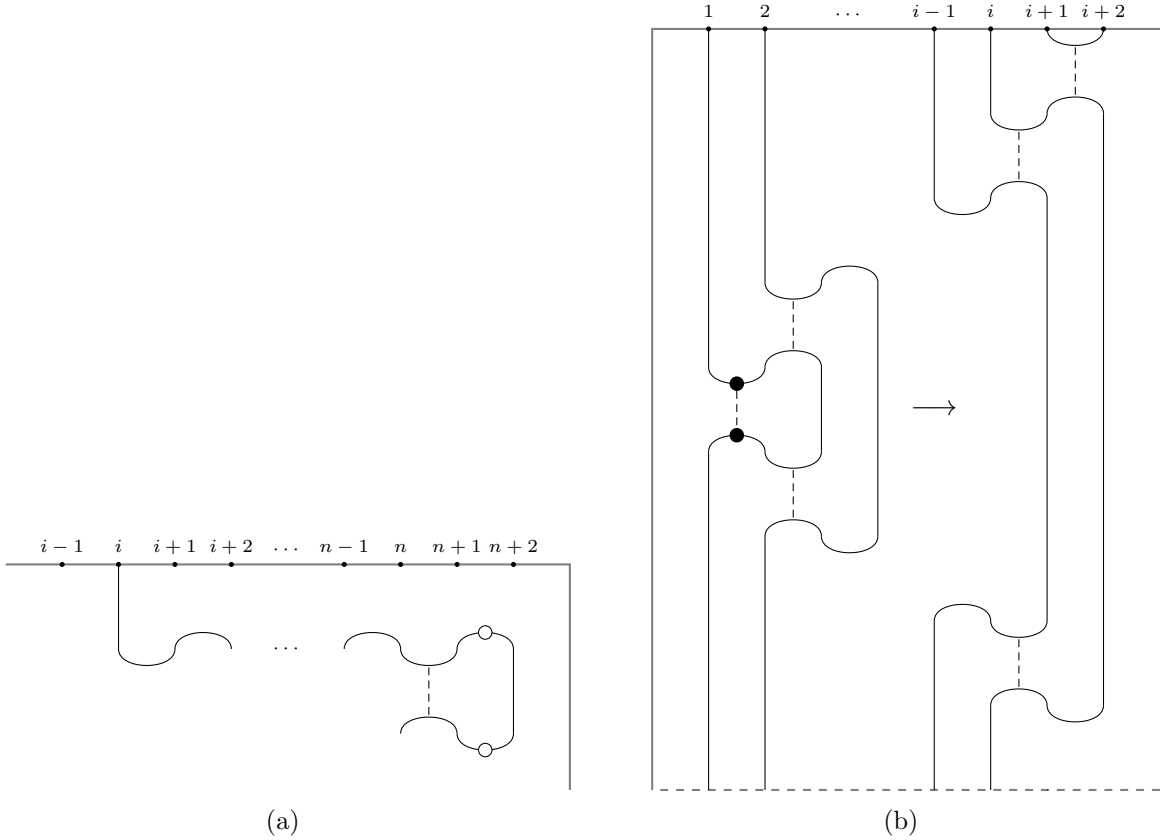
\begin{figure}[!ht]
\subcaptionbox{\label{fig:main result 1}}[.49\linewidth]{
\begin{tikzpicture}[scale=.75]
\draw[gray,thick] (9,2.75) to (9,6.75) to (-1,6.75);
\foreach \x in {0,1,2,3,5,6,7,8} \filldraw (\x,6.75) circle (1pt);
\draw[out=90,in=90] (6,3.75) to node[inner sep=0pt](A){} (7,3.75);
\draw[out=-90,in=-90] (6,5.25) to node[inner sep=0pt](B){} (7,5.25);
\draw[out=-90,in=-90] (1,5.25) to node[inner sep=0pt](G){} (2,5.25);
\draw[out=90,in=90] (5,5.25) to node[inner sep=0pt](E){} (6,5.25);
\draw[out=90,in=90] (2,5.25) to node[inner sep=0pt](F){} (3,5.25);

\draw[out=90,in=90] (7,5.25) to node[whitecirc, pos=0.5](C){} (8,5.25);
\draw[out=-90,in=-90] (7,3.75) to node[whitecirc, pos=0.5](D){} (8,3.75);
\node[above] at (4,6.75) {\tiny $\cdots$};
\node[above] at (4,5) {\tiny $\cdots$};
\node[above] at (0,6.75) {\tiny $\phantom{+}i-1\phantom{+}$};
\node[above] at (1,6.75) {\tiny $\phantom{+}i\phantom{+}$};
\node[above] at (2,6.75) {\tiny $\phantom{+}i+1\phantom{+}$};
\node[above] at (3,6.75) {\tiny $\phantom{+}i+2\phantom{+}$};
\node[above] at (5,6.75) {\tiny $\phantom{+}n-1\phantom{+}$};
\node[above] at (6,6.75) {\tiny $\phantom{+}n\phantom{+}$};
\node[above] at (7,6.75) {\tiny $\phantom{+}n+1\phantom{+}$};
\node[above] at (8,6.75) {\tiny $\phantom{+}n+2\phantom{+}$};
\draw (8,3.75) to (8,5.25);
\draw (1,6.75) to (1,5.25);
\draw[densely dashed] (B) -- (A);
\end{tikzpicture}
}
\subcaptionbox{\label{fig:main result 2}}[.49\linewidth]{
\begin{tikzpicture}[baseline=-0.5ex,scale=.75]
\draw[gray,thick] (0,-6.75) to (0,6.75) to (9,6.75);
\draw[gray,thick,dashed] (9,-6.75) to (0,-6.75);
\foreach \x in {1,2,5,6,7,8} \filldraw (\x,6.75) circle (1pt);
\draw[out=-90,in=-90] (1,.75) to node[blackcirc, pos=0.5](A){} (2,.75);
\draw[out=90,in=90] (1,-.75) to node[blackcirc, pos=0.5](B){} (2,-.75);

\draw[out=90,in=90] (2,.75) to node[inner sep=0pt](C){} (3,.75);
\draw[out=-90,in=-90] (2,2.25) to node[inner sep=0pt](D){} (3,2.25);

\draw[out=-90,in=-90] (2,-.75) to node[inner sep=0pt](E){} (3,-.75);
\draw[out=90,in=90] (2,-2.25) to node[inner sep=0pt](F){} (3,-2.25);

\draw[out=90,in=90] (3,2.25) to node[inner sep=0pt](G){} (4,2.25);
\draw[out=-90,in=-90] (3,-2.25) to node[inner sep=0pt](H){} (4,-2.25);

\draw[out=90,in=90] (5,-3.75) to node[inner sep=0pt](I){} (6,-3.75);
\draw[out=-90,in=-90] (5,3.75) to node[inner sep=0pt](J){} (6,3.75);

\draw[out=90,in=90] (6,3.75) to node[inner sep=0pt](L){} (7,3.75);
\draw[out=-90,in=-90] (6,5.25) to node[inner sep=0pt](K){} (7,5.25);

\draw[out=90,in=90] (7,5.25) to node[inner sep=0pt](M){} (8,5.25);

\draw[out=-90,in=-90] (6,-3.75) to node[inner sep=0pt](O){} (7,-3.75);
\draw[out=90,in=90] (6,-5.25) to node[inner sep=0pt](P){} (7,-5.25);

\draw[out=-90,in=-90] (7,-5.25) to node[inner sep=0pt](Q){} (8,-5.25);
\draw[out=-90,in=-90] (7,6.75) to node[inner sep=0pt](R){} (8,6.75);

\node at (5,0) {$\longrightarrow$};
\node[above] at (3.5,6.75) {\tiny $\cdots$};

\node[above] at (1,6.75) {\tiny $\phantom{+}1\phantom{+}$};
\node[above] at (2,6.75) {\tiny $\phantom{+}2\phantom{+}$};
\node[above] at (5,6.75) {\tiny $\phantom{+}i-1\phantom{+}$};
\node[above] at (6,6.75) {\tiny $\phantom{+}i\phantom{+}$};
\node[above] at (7,6.75) {\tiny $\phantom{+}i+1\phantom{+}$};
\node[above] at (8,6.75) {\tiny $\phantom{+}i+2\phantom{+}$};

\draw (3,.75) to (3,-.75);
\draw (4,2.25) to (4,-2.25);
\draw (1,.75) to (1,6.75);
\draw (1,-.75) to (1,-6.75);
\draw (2,2.25) to (2,6.75);
\draw (2,-2.25) to (2,-6.75);
\draw (5,3.75) to (5,6.75);
\draw (5,-3.75) to (5,-6.75);
\draw (6,5.25) to (6,6.75);
\draw (6,-5.25) to (6,-6.75);
\draw (7,6.75) to (7,6.75);
\draw (7,-6.75) to (7,-6.75);
\draw (7,-3.75) to (7,3.75);
\draw (8,-5.25) to (8,5.25);

\draw[densely dashed] (A) -- (B);
\draw[densely dashed] (C) -- (D);
\draw[densely dashed] (E) -- (F);
\draw[densely dashed] (K) -- (L);
\draw[densely dashed] (O) -- (P);
\draw[densely dashed] (R) -- (M);
\end{tikzpicture}
}
\caption{Portions of the diagrams corresponding to the proof of Theorem~\ref{thm:main result}.}\label{fig:main result}
\end{figure}

For convenience in wording our final result, define $w\in\FC(\C_n)$ to be a \emph{partial type II} element if $w=uxv$ (reduced) such that $x$ is of type II but not equal to $\x_{\E}$ when $n$ is even.

\begin{corollary}\label{cor:finiteNonTypeIandII}
There are finitely many elements in $\FC(\C_n)$ that are neither a type I nor a partial type II element.
\end{corollary}

\begin{proof}
By construction, there are finitely many diagrams in $\DTL(\C_n)$ that do not have minimal or maximal $\a$-value.  It follows that there are finitely many elements in $\FC(\C_n)$ that are neither a type I nor a partial type II element since these elements index diagrams that do not have extremal $\a$-values.
\end{proof}

\end{subsection}

\end{section}


\begin{section}{Closing remarks}\label{sec:closing}

In this paper, we proved that the associative diagram algebra $\DTL(\C_n)$ introduced in~\cite{Ernst2012} is a faithful representation of the generalized Temperley--Lieb algebra $\TL(\C_n)$.  Moreover, we showed that each admissible diagram of $\DTL(\C_n)$ corresponds to a unique monomial basis element of $\TL(\C_n)$.  Our diagrammatic realization of $\TL(\C_n)$ can be of great value when it comes to understanding the otherwise purely abstract structure of the algebra.  

It is natural to wonder whether a shorter proof of Theorem~\ref{thm:main result} exists.  As mentioned in Section~\ref{sec:intro}, independently verifying Property B~\cite{Green2007a} for $\H(\C_n)$ would likely lead to the construction of a basis of diagrams for $\TL(\C_n)$. 

As an application of our diagrammatic representation of $\TL(\C_n)$, we plan to model the approach in~\cite{Green2007} and construct a Jones-type trace on $\H(\C_n)$ that allows us to efficiently compute leading coefficients of the infinitely many Kazhdan--Lusztig polynomials indexed by pairs of FC elements. This will require a change of basis from the admissible diagrams described in this paper to the image of the canonical Kazhdan--Lusztig basis~\cite{Kazhdan1979} of $\H(\C_n)$. 

\end{section}


\section*{Acknowledgements}

I would like to thank R.M.~Green for suggesting Corollary~\ref{cor:finiteNonTypeIandII}.  I am also grateful to the referee for his or her careful reading of the paper and suggestions for improvements. 

\bibliographystyle{plain}
\bibliography{DiagCalc2}

\end{document}